\documentclass{jams-l}

\usepackage{amssymb}

\usepackage{dsfont}

\usepackage{mathabx}

\usepackage{graphicx}

\usepackage{hyperref}

\usepackage{xcolor}

\usepackage{mathbbol}

\usepackage{array} 

\usepackage{float}

\usepackage{geometry}

\usepackage{accents} 

\usepackage[autostyle]{csquotes} 

\usepackage[round]{natbib}

\newtheorem{theorem}{Theorem}[section]

\newtheorem{proposition}[theorem]{Proposition}
\newtheorem{lemma}[theorem]{Lemma}

\theoremstyle{definition}
\newtheorem{definition}[theorem]{Definition}
\newtheorem{example}[theorem]{Example}
\newtheorem{assumption}[theorem]{Assumption}

\theoremstyle{remark}
\newtheorem{remark}[theorem]{Remark}

\numberwithin{equation}{section}

\DeclareSymbolFontAlphabet{\amsmathbb}{AMSb}%

\usepackage{tikz}

\usetikzlibrary{arrows,positioning,calc} 
\usetikzlibrary{decorations.pathreplacing}
\tikzset{
    >=stealth',
    punkt/.style={
           rectangle,
           rounded corners,
           draw=black, thick,
           text width=5.5em,
           minimum height=2em,
           text centered},
    punktl/.style={
           rectangle,
           rounded corners,
           draw=black, thick,
           text width=7em,
           minimum height=2em,
           text centered},
    pil/.style={
           ->,
           shorten <=4pt,
       shorten >=4pt
    },
    pildotted/.style={
           ->,
           shorten <=4pt,
           shorten >=4pt,
  dotted,
  },
    punktf/.style={
           rectangle,
           text width=4.0em,
           minimum height=1.5em,
           text centered},
    punktfTop/.style={
           rectangle,
           text width=4.0em,
           minimum height=1.5em,
           text centered,
           append after command={
               [thick,shorten >=0.2bp, shorten <=0.2bp]
               (\tikzlastnode.north west)edge(\tikzlastnode.north east)
}
    },
    punktfBot/.style={
           rectangle,
           text width=4.0em,
           minimum height=1.5em,
           text centered,
           append after command={
               [thick,shorten >=0.2bp, shorten <=0.2bp]
               (\tikzlastnode.south west)edge(\tikzlastnode.south east)
            }
    }
}

\newcommand{\ti}[1]{\widetilde{#1}}

\newcommand{\ul}[1]{\underaccent{\bar}{#1}}

\usepackage{prodint}
\theoremstyle{plain}

\begin{document}

\title[]{Robust Control for Marked Point Processes under Transition-Rate Uncertainty}


\author{Sascha Desmettre}
\address{Institute of Financial Mathematics and Applied Number Theory, Johannes Kepler University Linz, Altenbergerstr. 69, 4040, Linz, Austira}
\curraddr{}
\email{sascha.desmettre@jku.at}
\thanks{}

\author{Philipp C.\ Hornung}
\address{Department of Mathematical Sciences, University of Copenhagen, Universitetsparken 5, DK-2100 Copenhagen, Denmark}
\curraddr{}
\email{\href{mailto:pcho@math.ku.dk}{pcho@math.ku.dk}}
\thanks{}


\date{}

\dedicatory{\today}

\begin{abstract}
We consider a novel robust utility maximisation problem under bounded cumulative transition rate uncertainty within the class of non-Markovian marked point processes on a finite state-space. Utility is maximised over the class of admissible controls, while Nature chooses a worst-case biometric scenario from the class of admissible, path-dependent cumulative transition rates restricted by path-dependent upper and lower bounds. We prove a martingale optimality principle and a novel existence and uniqueness result for a non-standard worst-case backwards stochastic differential equation, which allows us to establish existence and uniqueness of worst-case and best-case prospective reserves of life and health insurance contracts with reserve-dependent payments. Finally, we find an explicit solution of a novel robust consumption-insurance problem with power utility preferences.
\vspace{3mm}
 \end{abstract}
\maketitle
\section{Introduction}
Multi-state models are a classical and widely used framework for the stochastic modelling of an individual's evolution through a finite set of biometric, contractual, or economic states. This makes multi-state models particularly suitable for the valuation of life and health insurance contracts and the study of an individual's optimal choice of consumption and insurance when exposed to biometric risk, within a unified continuous-time framework. Originally such multi-state models were developed for the purpose of valuation of life and health insurance contracts within the class of continuous-time Markov chains with finite state-spaces, see~\cite{Hoem1969,Norberg1991,MilbrodtStracke1997}, but in parallel the class of continuous-time Markov chains also came to play a central role in the literature on optimal consumption-insurance problems under biometric risk, see~\cite{Richard1975,PliskaYe2007,KraftSteffensen2008}. 

Defining for the class of Markovian multi-state models is that the biometric scenario is modelled by deterministic, time-dependent cumulative transition rates, or, in the smooth case, deterministic, time-dependent transition intensities. However, as the cumulative transition rates for health-related risks such as disability, long-term care, and unemployment may depend on the duration already spent in the current state, or on other features of the observed history, the Markov assumption is often too restrictive for accurate modelling. Thus recent literature on the valuation of life and health insurance contracts has moved towards more general multi-state models, such as semi-Markov and duration-dependent models,~\cite{Christiansen2012,BuchardtMoellerSchmidt2015,AdekambiChristiansen2017}, or so-called non-Markov models,~\cite{Christiansen&Djehiche2020,ChristiansenFurrer2021}, where the biometric scenario may be modelled by path-dependent cumulative transition rates dependent on the entire observed history, see~\cite{Christiansen&Furrer2024}. This move has been mirrored by the statistics literature for multi-state models, see~\cite{PutterSpitoni2018,NiesslEtAl2023,BladtFurrer2024}.

Another difficulty, which is common for all multi-state models, and the main focus of this paper, is model uncertainty. In actuarial practice, regardless of what class of multi-state models is used, the cumulative transition rates have to be estimated from finite, heterogeneous and often non-stationary data sets. This motivates robust or safe-side approaches in which one considers a whole class of admissible models whose cumulative transition rates are allowed to vary within a set of plausible alternatives. Within the class of Markovian models \cite{Christiansen2010} studies the worst-case valuation of multi-state life and health insurance contracts, given deterministic upper and lower bounds for the cumulative transition rates, while~\cite{ChristiansenSteffensen2013} and~\cite{ChristiansenEtAl2016} extend this approach to joint financial and biometric worst-case scenarios. Within the literature on optimal consumption and insurance problems \cite{HanHung2021} consider the robust optimal choice of consumption and annuity demand under mortality rate uncertainty within a Markovian survival model. In all of these contributions, the class of alternative models is restricted to Markovian models, and extensions allowing for non-Markovian alternative models are, to the best of our knowledge, absent from the literature. Thus, motivated by the recent shift towards non-Markov models by not only the actuarial and statistical literature, but also actuarial practice, we aim to provide such a non-Markovian extension.

In this paper, we formulate a general max-min objective which aims at maximising a path-dependent utility process under the least favourable admissible biometric scenario, where the admissible biometric scenarios are described by path-dependent lower and upper bounds on the cumulative transition rates of a finite state non-Markovian marked point process. To formulate our problem, we use the canonical measurable space of non-explosive marked point process paths without trivial jumps on a finite state space, as constructed by~\cite{Christiansen&Furrer2024}. This canonical construction has the crucial advantage that it separates the measurable path space from the choice of probability measure, since, as shown by~\cite{Christiansen&Furrer2024}, each unique set of cumulative transition rates induces a unique probability measure on the same canonical measurable space. In particular, this setup does not require that all admissible probability measures have to be equivalent.

After formulating our max-min objective, we prove a martingale optimality principle which provides a sufficient condition of optimality and which we subsequently use to obtain verification theorems in two specialised cases. The first application of the martingale optimality principle shows that the worst- and best- case prospective reserves of an insurance contract with reserve-dependent payments are characterised by a non-standard backwards stochastic differential equation (BSDE). The second application of the martingale optimality principle yields a verification theorem for a robust and path-dependent generalisation of the consumption-insurance problem of~\cite{KraftSteffensen2008}, where the individual aims to maximise utility of consumption and terminal wealth by choosing a consumption stream and transition-contingent insurance coverage, while Nature chooses, within the admissible upper and lower bounds, the transition intensities that minimise the individual's expected utility. In the case of power utility preferences, the classical separation Ansatz still applies, with the deterministic time-dependent function replaced by a stochastic process satisfying a BSDE of the same non-standard type as the worst- and best-case prospective reserves. This yields explicit path-dependent feedback forms for the robust optimal consumption stream and insurance coverage of the same structural form as in the non-robust Markov-chain model of~\cite{KraftSteffensen2008}, but evaluated under worst-case preference-dependent transition intensities.

The BSDEs characterising the best- and worst-case prospective reserves and the solution of the power utility problem are non-standard in two ways. Firstly, the distribution of the jump times may have point masses, violating the assumption of inaccessible jump times. Secondly, due to the reserve-dependence of the worst/best-case cumulative transition rates, the generator of our BSDE no longer satisfies standard Lipschitz assumptions. Both assumptions are central for the existence and uniqueness results in the papers of \cite{ConfortolaFuhrman2013,ConfortolaFuhrmanJacod2016,Christiansen&Djehiche2020} and while~\cite{Christiansen&Furrer2024} allow for accessible jump times, they require a linear generator. Nevertheless, a key feature of the approach of~\cite{ConfortolaFuhrmanJacod2016} is the reduction of the BSDE to a finite or infinite system of ordinary differential equations between jumps. We similarly prove existence and uniqueness of our non-standard BSDE by first establishing equivalence to the existence and uniqueness of a solution of an infinite system of Lebesgue-Stieltjes integral equations and then by proving that this system indeed has a unique solution. This is aided by our chosen canonical setting, which allows for path-wise sure constructions.

The worst-case biometric scenario we obtain in the case of the worst-case prospective reserve is determined by the sums-at-risk. As long as a particular transition increases the prospective reserve, the upper bound is selected as the worst-case cumulative transition rate; if it decreases the prospective reserve, the lower bound is selected. Thus, our worst-case scenario can be viewed as a direct path-wise extension of~\cite{Christiansen2010}, who, in the Markovian case, obtained a worst-scenario of the same structural form. In the power utility case, a similar worst-case scenario is obtained. The difference is that the worst-case cumulative transition rates are determined by the sums-at-risk of a utility-valued process and the relative risk aversion parameter.

Our robust approach is related to the broader literature on ambiguity and robust decision making. The max-min expected utility criterion originates from~\cite{GilboaSchmeidler1989}, while robust control under model misspecification has been developed systematically by~\cite{HansenSargent2008}. Variational and multiple-prior approaches to ambiguity are discussed in~\cite{MaccheroniMarinacciRustichini2006}. In continuous-time finance, robust portfolio and consumption problems have been studied, for example, by~\cite{Maenhout2004} and by~\cite{RiederWopperer2012}. In contrast to much of this literature, the model uncertainty considered here concerns the cumulative transition rates of a biometric finite-state marked point process rather than the drift or volatility of a financial diffusion.

The main contribution of the paper is twofold. First, we formulate a novel robust utility maximisation problem under bounded cumulative transition rate uncertainty and provide a sufficient condition for optimality. This allows us to characterise best- and worst-case prospective reserves under bounded cumulative transition rate uncertainty and identify the corresponding extremal biometric scenarios. Further, it allows us to characterise the solution of a robust consumption-insurance problem under transition intensity uncertainty, showing how the classical continuous-time Markov-chain solution is modified by the endogenous worst-case scenario. Second, within the framework of marked point processes on a finite state-space, we provide a novel existence and uniqueness result for a class of highly non-linear worst-case BSDEs, which allows us to provide sufficient conditions for the existence of the best- and worst-case prospective reserves and the solution of the robust consumption-insurance problem.

The remainder of the paper is organised as follows. Section 2 introduces the canonical marked point process framework and develops necessary jump-count and state-wise decompositions of optional and predictable processes. Section 3 studies optional projections and proves the existence and uniqueness result of the worst-case BSDE. Section 4 formulates the general robust utility maximisation problem under bounded cumulative transition rate uncertainty and provides the martingale optimality principle. Section 5 characterises best- and worst-case prospective reserves. Section 6 studies the robust consumption-insurance problem, proves the verification theorem, and discusses explicit optimal controls in the power-utility case. Known technical results on finite-variation functions and Stieltjes integration are collected in the appendix for the reader's convenience.

\section{The canonical probability space}\label{sec:2}
In this section, we start with a description of the canonical filtered measurable space of non-explosive marked point process paths without trivial jumps on a finite state space, as constructed in~\cite{Christiansen&Furrer2024}, which will be the filtered measurable space we use throughout the rest of the paper. This is followed by a discussion, which links the jump-count decomposition of optional and predictable processes known from the marked point process literature (see~\cite{Jacobsen2006,Last&Brandt1995}) with the state-wise decomposition of optional processes introduced in~\cite{Christiansen&Furrer2024}. In particular, we introduce a unique jump-count decomposition in Proposition~\ref{prop:Rep_adapted_process_N}, which in Proposition~\ref{prop:Rep_adapted_process_J} allows us to generalise and extend the state-wise decomposition result obtained by~\cite{Christiansen&Furrer2024}. The combination of these two results links the two decompositions and enables us to seamlessly switch between them. This is crucial for our purposes because both decompositions play central (but different) roles throughout the paper. The state-wise representation is central for linking optional projections to BSDEs, and it is less unwieldy when solving the robust optimisation problems we are interested in. The jump-count representation is the key to proving the existence and uniqueness of the worst-case BSDE that characterises the worst- and best-case prospective reserves and the solution of the robust consumption-insurance problem with power utility preferences. Finally, we end this section by outlining how to endow the canonical filtered measurable space with a probability measure given by specified cumulative transition rates, as constructed in~\cite{Christiansen&Furrer2024}.

\subsection{The canonical filtered measurable space}
Let $\mathcal{J}\subset\amsmathbb{N}_0$ be a finite set. A marked point process path without trivial jumps consists of jump times $t_0=0<t_1<t_2<\ldots$ and corresponding states $z_0,z_1,\ldots\in\mathcal{J}$ with $z_0\neq z_1, z_1\neq z_2,z_2\neq z_3,\ldots$ which together form a sequence of events $(0,z_0),(z_1,t_1),\dots\,$. For our purposes, we can think of these events as the initial value and the jump times and values of a piecewise constant jump process describing, for example, the biometric state of an individual. Since we are only interested in non-explosive sequences, the total number of events may be countably infinite on the full time line $[0,\infty)$, but must be finite on any bounded time interval. In the case that the total number of events occurring in $[0,\infty)$ is finite, there exists $n\in\amsmathbb{N}_0$ such that the corresponding marked point process path $(t_{\ell},z_{\ell})_{\ell\leq n}$ is an element of
\begin{align*}
  \Omega^n:=\big\{(t_{\ell},z_{\ell})_{\ell\leq n}\in [0,\infty)^{n+1}\times \mathcal{J}^{n+1}\,:\,0=t_0<t_1<\cdots<t_n,\,z_{\ell}\neq z_{\ell+1}\big\}.
\end{align*}
If there is an infinite number of events in $[0,\infty)$, then the corresponding marked point process path $(t_{\ell},z_{\ell})_{\ell\in\amsmathbb{N}_0}$ is an element of 
\begin{align*}
  \Omega^{\infty}:=\big\{(t_{n},z_{n})_{n\in\amsmathbb{N}_0}\in [0,\infty)^{\amsmathbb{N}_0}\times \mathcal{J}^{\amsmathbb{N}_0}\,:\,t_0=0,\, t_{n}<t_{n+1}<\infty,\,z_{n}\neq z_{n+1},\,\sup_{n}t_n=\infty\big\}.
\end{align*}
All possible non-explosive marked point process paths without trivial jumps could thus be written as
\begin{align*}
  \Omega^{\infty}\cup \bigcup_{n\in\amsmathbb{N}_0}\Omega^n.
\end{align*}
As this is mathematically unwieldy to work with, we will extend any path with finite events to a path with countably infinite events, by extending the finite sequence with the artificial event that never happens $(\infty,\nabla)$. The state $\nabla$ is unreachable, which is why its jump time is $\infty$. By defining the extended state space $\bar{\mathcal{J}}:=\mathcal{J}\cup\{\nabla\}$, the canonical space of non-explosive marked point process paths without trivial jumps may be defined as
\begin{align*}
  \Omega:=\big\{(t_n,z_n)_{n\in\amsmathbb{N}_0}\in [0,\infty]^{\amsmathbb{N}_0}\times\bar{\mathcal{J}}^{\amsmathbb{N}_0}\,:\,&t_0=0,\, t_n<t_{n+1} \text{ for } t_n<\infty,\, t_n=t_{n+1}\text{ for } t_n=\infty,\\
  &z_n\neq z_{n+1} \text{ for } t_n<\infty,\,z_n=\nabla \text{ for } t_n=\infty,\,\sup_n t_n=\infty\big\},
\end{align*}
which can be viewed as a subset of $[0,\infty]^{\amsmathbb{N}_0}\times\bar{\mathcal{J}}^{\amsmathbb{N}_0}$. Using the projection mappings
\begin{align*}
  &\tau_n:\Omega\rightarrow[0,\infty],\quad \tau_n(\omega)=t_n,\\
  &\zeta_n:\Omega\rightarrow \bar{\mathcal{J}},\quad \zeta_n(\omega)=z_n,
\end{align*}
we can define the canonical marked point process $\omega\mapsto (\tau_n(\omega),\zeta_n(\omega))_{n\in\amsmathbb{N}_0}$ and the $n$-event projection mappings
\begin{align*}
  \forall n\in\amsmathbb{N}_0:\quad \xi^n:\Omega\rightarrow [0,\infty]^{n+1}\times \bar{\mathcal{J}}^{n+1},\quad \xi^n(\omega)=(\tau_{\ell}(\omega),\zeta_{\ell}(\omega))_{0\leq\ell\leq n},
\end{align*}
which extracts the initial value and the first $n$ jump times and states of a path $\omega\in\Omega$. The mapping $\xi^n$ has the image $\bar{\Omega}^n:=\xi^n(\Omega)\subset [0,\infty]^{n+1}\times\bar{\mathcal{J}}^{n+1}$ equal to
\begin{align*}
  \bar{\Omega}^n=\big\{(t_{\ell},z_{\ell})_{0\leq \ell\leq n}\in [0,\infty]^{n+1}\times\bar{\mathcal{J}}^{n+1}\,:\,&t_0=0,\, t_{\ell}<t_{\ell+1} \text{ for } t_{\ell}<\infty,\, t_{\ell}=t_{\ell+1}\text{ for } t_{\ell}=\infty,\\
  &z_{\ell}\neq z_{\ell} \text{ for } t_{\ell}<\infty,\,z_{\ell}=\nabla \text{ for } t_{\ell}=\infty\big\}.
\end{align*}
Contrary to the set $\Omega^n$, which contains all possible sequences of exactly $n$ events, the set $\bar{\Omega}^n$ contains all possible sequences with at most $n$ events and thus $\Omega^n\subset\bar{\Omega}^n$, the only exception being the case $n=0$, where $\bar{\Omega}^0=\Omega^0=\{0\}\times\mathcal{J}$. A generic element of $\Omega^n$ or $\bar{\Omega}^n$ will be denoted by $\omega^n=((0,z_0),\dots,(t_n,z_n))$. 

The set $\Omega$ can be equipped with the $\sigma$-algebra generated by the canonical marked point process
\begin{align*}
  \mathcal{F}:=\sigma((\tau_n,\zeta_n)_{n\in\amsmathbb{N}_0})
\end{align*}
and the two filtrations $\amsmathbb{F}:=(\mathcal{F}_t)_{t\geq 0}$ and $\amsmathbb{F}^-:=(\mathcal{F}_{t-})_{t\geq 0}$ generated by the canoncial marked point process given by
\begin{align*}
  \mathcal{F}_t:&=\sigma(\mathds{1}_{(\tau_n\leq t)}(\tau_n,\zeta_n):n\in\amsmathbb{N}_0)\\
  \mathcal{F}_{t-}:&=\sigma(\mathds{1}_{(\tau_n<t)}(\tau_n,\zeta_n):n\in\amsmathbb{N}_0),
\end{align*}
where the former contains past and present information, while the latter only concerns itself with past information. Note that the filtration $\amsmathbb{F}$ is right-continuous. 

The canonical marked point process can be used to define three other types of important fundamental processes, which all generate the same above filtrations $\amsmathbb{F}$ and $\amsmathbb{F}^-$. First, we consider the counting processes
\begin{align*}
  N_{ij}:[0,\infty)\times\Omega\rightarrow \amsmathbb{N}_0,\quad N_{ij}(t):=\sum_{n\in\amsmathbb{N}}\mathds{1}_{(\tau_n\leq t)}\mathds{1}_{(\zeta_{n-1}=i,\zeta_{k}=j)},\quad i,j\in\mathcal{J}:i\neq j,
\end{align*}
which count the number of jumps from state $i$ to state $j$, while 
\begin{align*}
  N:[0,\infty)\times \Omega\rightarrow\amsmathbb{N}_0,\quad N(t):=\sum_{n\in\amsmathbb{N}}\mathds{1}_{(\tau_n\leq t)}=\sum_{i,j:i\neq j}N_{ij}(t)
\end{align*}
counts the total number of jumps. Then, we consider the indicator processes 
\begin{align*}
  &I^n:[0,\infty)\times\Omega\rightarrow\{0,1\},\quad I^n(t):=\mathds{1}_{(\tau_n\leq t< \tau_{n+1})},\quad \text{for }n\in\amsmathbb{N}_0,\\
  &I_i:[0,\infty)\times\Omega\rightarrow\{0,1\},\quad I_i(t):=\sum_{n\in\amsmathbb{N}_0}\mathds{1}_{(\tau_n\leq t< \tau_{n+1})}\mathds{1}_{(\zeta_n=i)},\quad \text{for }i\in\mathcal{J},
\end{align*}
which can be used to construct the jump process
\begin{align*}
  Z:[0,\infty)\times\Omega\rightarrow\mathcal{J},\quad Z(t):=\sum_{n\in\amsmathbb{N}_0}\zeta_n I^n(t)=\sum_{i\in\mathcal{J}}iI_i(t),
\end{align*}
with $Z(0-):=Z(0)=\zeta_0$. By construction, $Z$ has càdlàg paths of finite variation, with a finite number of jumps in finite time. It holds that
\begin{align*}
  \mathcal{F}_t&=\sigma(\zeta_0,N_{ij}(s):s\leq t,i,j\in\mathcal{J},i\neq j)=\sigma(I_i(s):s\leq t,i\in\mathcal{J})=\sigma(Z(s):s\leq t).
\end{align*}
and 
\begin{align*}
  \mathcal{F}_{t-}&=\sigma(\zeta_0,N_{ij}(s):s<t,i,j\in\mathcal{J},i\neq j)=\sigma(I_i(s):s< t,i\in\mathcal{J}=\sigma(Z(s):s< t).
\end{align*}
\begin{definition}
  The set $\Omega$ equipped with the sigma-algebra $\mathcal{F}$ and the filtration $\amsmathbb{F}$ is the canonical filtered measurable space.
\end{definition}

\subsection{Decompositions of optional and predictable processes}
Let $(\Omega,\mathcal{F},\amsmathbb{F})$ be the canonical filtered measurable space. Each $\amsmathbb{F}$-optional or $\amsmathbb{F}$-predictable (henceforth just "optional" and "predictable") process $X$ can be decomposed in two different ways. The first is the jump-count decomposition, which decomposes $X$ into the family of functions $(f^n)_{n\in\amsmathbb{N}_0}$, each being a measurable function of the first $n$ jump times and marks. The second is the state-wise decomposition, which decomposes $X$ into a family of predictable processes $(\ti{X}_i)_{i\in\mathcal{J}}$, which can be interpreted as counterfactuals in a certain sense. We start by discussing the former. 

By Theorem~2.2.22 of~\cite{Last&Brandt1995}, it holds that a process $X$ is optional if and only if there exists a family of functions $f^n:[0,\infty)\times\bar{\Omega}^n\rightarrow\amsmathbb{R}$, $n\in\amsmathbb{N}_0$ such that 
\begin{align}\label{eq:Rep_optional_process_N}
  X(t,\omega)=\sum_{n\in\amsmathbb{N}_0}I^n(t,\omega)f^n(t,\xi^n(\omega)),\quad \forall (t,\omega)\in [0,\infty)\times\Omega,
\end{align}
and a process $X$ is predictable if and only if 
\begin{align}\label{eq:Rep_predictable_process_N}
  X(t,\omega)=\sum_{n\in\amsmathbb{N}_0}I^n(t-,\omega)f^n(t,\xi^n(\omega)),\quad \forall (t,\omega)\in [0,\infty)\times\Omega.
\end{align}
Thus, not only can any optional or predictable process be decomposed into a family of functions $(f^n)_{n\in\amsmathbb{N}_0}$, but this decomposition property actually characterises being optional or predictable. From (\ref{eq:Rep_optional_process_N}) and (\ref{eq:fn_uniqueness_condition_predictable}), it can be seen that each $f^n$ does not have to be uniquely defined on the entire domain $[0,\infty)\times\bar{\Omega}^n$ in order for $X$ to be unique. Since we later want to construct adapted processes by constructing the family $(f^n)_{n\in\amsmathbb{N}_0}$, it is important to know on what domain the $f^n$ uniquely determine $X$ and on what domain they are uniquely determined by $X$.

\begin{proposition}\label{eq:Rep_N_uniqueness}
  Let $X$ and $Y$ be two optional or predictable processes satisfying (\ref{eq:Rep_optional_process_N}) or (\ref{eq:Rep_predictable_process_N}) with $(f^n)_{n\in\amsmathbb{N}_0}$ and $(g^n)_{n\in\amsmathbb{N}_0}$. Then the following two statements hold:
  \begin{enumerate}
    \item[(i)] If $X$ and $Y$ are optional, then $X(t,\omega)=Y(t,\omega)$ for all $(t,\omega)\in [0,\infty)\times\Omega$ if and only if 
    \begin{align*}
      \forall n\in\amsmathbb{N}_0:\quad f^n(t,\omega^n)=g^n(t,\omega^n),\quad \forall\omega^n=((t_0,z_0),\ldots,(t_n,z_n))\in\Omega^n,\quad\forall t\geq t_n.
    \end{align*}
    \item[(ii)] If $X$ and $Y$ are predictable, then $X(t,\omega)=Y(t,\omega)$ for all $(t,\omega)\in [0,\infty)\times\Omega$ if and only if 
    \begin{align*}
      \forall n\in\amsmathbb{N}_0:\quad f^n(t,\omega^n)=g^n(t,\omega^n),\quad \forall\omega^n=((t_0,z_0),\ldots,(t_n,z_n))\in\Omega^n,\quad\forall t> t_n.
    \end{align*}
  \end{enumerate}
\end{proposition}
\begin{proof}
  We only prove (i) as the proof of (ii) is analogous, and we start with the direction "$\Rightarrow$". By (\ref{eq:Rep_optional_process_N}) it holds that
  \begin{align*}
    \forall n\in\amsmathbb{N}_0,\,\omega\in\Omega:\quad f^n(t,\xi^n(\omega))=X(t,\omega)=Y(t,\omega)=g^n(t,\xi^n(\omega)),\quad \tau_n(\omega)\leq t <\tau_{n+1}(\omega).
  \end{align*}
  In particular for any $n\in\amsmathbb{N}_0$ this holds for all paths of the form $\omega=(\omega_n,(\infty,\nabla),\dots)$ with $\omega^n\in\Omega^n$. Thus we obtain
  \begin{align*}
    \forall n\in\amsmathbb{N}_0,\,\omega^n\in \Omega^n:\quad f^n(t,\omega^n)=g^n(t,\omega^n),\quad t_n\leq t <\infty.
  \end{align*}
  For the direction "$\Leftarrow$" let $\omega\in\Omega$ be arbitrary. Note that since $t\in [0,\infty)$, it is only the $f^n$ and $g^n$ with $n\in\{m\in\amsmathbb{N}_0:\tau_m(\omega)<\infty\}$ that determine $X(t,\omega)$ and $Y(t,\omega)$. By definition if $\tau_n(\omega)<\infty$, then $\xi^n(\omega)\in\Omega^n$ and since for all $n\in\amsmathbb{N}_0$ and $\omega^n\in\Omega^n$ it holds that $f^n(t,\omega^n)=g^n(t,\omega^n)$ for $t\geq t_n$, we have that 
  \begin{align*}
    X(t,\omega)=f^n(t,\xi^n(\omega))=g^n(t,\xi^n(\omega))=Y(t,\omega),\quad \tau_n(\omega)\leq t <\tau_{n+1}(\omega),
  \end{align*}
  whenever $\tau_n(\omega)<\infty$. Thus we can conclude $X(t,\omega)=Y(t,\omega)$ for all $(t,\omega)\in [0,\infty)\times\Omega$.
\end{proof}

Proposition~\ref{eq:Rep_N_uniqueness} shows that there are two cases in which the functions $(f^n)_{n\in\amsmathbb{N}_0}$ are not uniquely determined. The first case is $\omega^n\in\Omega^n$ and $t<t_n$, meaning that we would like to evaluate $f^n$ before the $n$'th jump has happened, while the second case is $\omega^n\in\bar{\Omega^n}\setminus\Omega^n$ (implying $t_n=\infty$), meaning that we would like to evaluate $f^n$ when there never will be an $n$'th jump. In order to obtain a unique decomposition $(f^n)_{n\in\amsmathbb{N}_0}$ on the entire domain $[0,\infty)\times\bar{\Omega}^n$, we have to choose a convention for the two previously described cases. For optional processes we choose the convention,
\begin{align}\label{eq:fn_uniqueness_condition_optional}
  f^n(t,\omega^n)=
  \begin{cases}
    f^n(t_n,\omega^n)&\quad \text{when }\omega^n\in\Omega^n\text{ and } t<t_n\\
    f^{n-1}(t,\xi^{n-1}(\omega^n))&\quad \text{when }\omega^n\in\bar{\Omega}^n\setminus\Omega^n
  \end{cases},
\end{align}
and for predictable processes which have limits from the right, we choose 
\begin{align}\label{eq:fn_uniqueness_condition_predictable}
  f^n(t,\omega^n)=
  \begin{cases}
    f^n(t_n+,\omega^n)&\quad\text{when }\omega^n\in\Omega^n\text{ and } t<t_n\\
    f^{n-1}(t,\xi^{n-1}(\omega^n))&\quad \text{when }\omega^n\in\bar{\Omega}^n\setminus\Omega^n
  \end{cases}.
\end{align}
The rationale behind this choice is very natural. In the first case, the $n$'th jump has not happened yet, but will happen, so $f^n$ equals the value that $X$ will take, as soon as the $n$'th jump happens. In the second case, the $n$'th jump will never happen, which is why $f^n$ equals the value $X$ takes at time $t$ after the $n-1$'st jump. Using these two conventions, we can obtain a unique jump-count decomposition.

\begin{proposition}\label{prop:Rep_adapted_process_N}
  Let $X=(X_t)_{t\geq 0}$ be a real-valued stochastic process.
  \begin{enumerate}
    \item[(i)] $X$ is $\amsmathbb{F}$-optional if and only if for every $n\in\amsmathbb{N}_0$ there exists a unique measurable function $f^n:[0,\infty)\times\bar{\Omega}^n\rightarrow\amsmathbb{R}$ satisfying (\ref{eq:fn_uniqueness_condition_optional}) such that:
    \begin{align*}
      X(t,\omega)=\sum_{n\in\amsmathbb{N}_0}I^n(t,\omega)f^n(t,\xi^n(\omega)),\quad\forall\,(t,\omega)\in [0,\infty)\times\Omega.
    \end{align*}
    \item[(ii)] Let $X$ have limits from the right. Then $X$ is $\amsmathbb{F}$-predictable if and only if for every $n\in\amsmathbb{N}_0$ there exists a unique measurable function $f^n:[0,\infty)\times\bar{\Omega}^n\rightarrow\amsmathbb{R}$ satisfying (\ref{eq:fn_uniqueness_condition_predictable}) such that:
    \begin{align*}
      X(t,\omega)=\sum_{n\in\amsmathbb{N}_0}I^n(t-,\omega)f^n(t,\xi^n(\omega)),\quad\forall\,(t,\omega)\in [0,\infty)\times\Omega.
    \end{align*}
  \end{enumerate}
\end{proposition}
\begin{proof}
  We prove only (i), since (ii) is proven analogously. The direction "$\Leftarrow$" follows directly from Theorem~2.2.22 of~\cite{Last&Brandt1995} and Proposition~\ref{eq:Rep_N_uniqueness}. For the other direction Theorem~2.2.22 of~\cite{Last&Brandt1995} yields the existence of measurable functions $\ti{f}^n:[0,\infty)\times\bar{\Omega}^n\rightarrow\amsmathbb{R}$ such that 
  \begin{align*}
    X(t,\omega)=\sum_{n\in\amsmathbb{N}_0}I^n(t,\omega)\ti{f}^n(t,\xi^n(\omega)),\quad\forall\,(t,\omega)\in [0,\infty)\times\Omega.
  \end{align*}
  From this we define $f^n:[0,\infty)\times\bar{\Omega}^n\rightarrow\amsmathbb{R}$ given by
  \begin{align*}
    f^n(t,\omega^n):=
    \begin{cases}
      \ti{f}^n(t,\omega^n) &\quad \text{when }\omega^n\in\Omega\text{ and }t\geq t_n\\
      \ti{f}^n(t_n,\omega^n) &\quad  \text{when }\omega^n\in\Omega\text{ and }t< t_n\\
      \ti{f}^{n-1}(t,\xi^{n-1}(\omega^n)) &\quad \text{when }\omega^n\in\bar{\Omega}^n\setminus\Omega
    \end{cases}.
  \end{align*}
  Rewriting this into the form 
  \begin{align*}
    f^n(t,\omega^n)=\mathds{1}_{\Omega}(\omega^n)\mathds{1}_{[t_n,\infty)}(t)\ti{f}^n(t,\omega^n)+\mathds{1}_{\Omega}(\omega^n)\mathds{1}_{[0,t_n)}(t)\ti{f}^n(t_n,\omega^n)+\mathds{1}_{\bar{\Omega}^n\setminus\Omega^n}(\omega^n)\ti{f}^{n-1}(t,\xi^{n-1}(\omega^n))
  \end{align*}
  shows that $(t,\omega^n)\mapsto f^n(t,\omega^n)$ is measurable. By Proposition~\ref{eq:Rep_N_uniqueness} $\ti{f}^n(t,\omega^n)$ is unique whenever $\omega^n\in\Omega^n$ and $t\geq t_n$ and thus it holds that $f^n$ is unique in the first two cases. In the third case, there are two options. Either $\xi^{n-1}(\omega^n)\in\Omega^{n-1}$, in which case the value is unique, or $\xi^{n-1}(\omega^n)\in\bar{\Omega}^{n-1}\setminus\Omega^{n-1}$. In that case, one repeats this argument over and over again, until one reaches some $0\leq m<n$ such that $\xi^m(\omega^n)\in\Omega^m$. If this fails for all $1\leq m\leq n$, the procedure terminates at $m=0$ where it holds that $\bar{\Omega}^0=\Omega^0=\{0\}\times\mathcal{J}$ and thus Proposition~\ref{eq:Rep_N_uniqueness} yields that $\ti{f}^0(t,\omega^0)$ is unique for all $(t,\omega^0)\in [0,\infty)\times\bar{\Omega}^0$. Thus $\ti{f}^0$ is always uniquely determined by $X$ on all of $[0,\infty)\times\bar{\Omega}^0$.
\end{proof}

\begin{remark}\label{rem:path_properties}
  Note that as a consequence of Proposition~\ref{eq:Rep_N_uniqueness} and the convention (\ref{eq:fn_uniqueness_condition_optional}) or (\ref{eq:fn_uniqueness_condition_predictable}), $t\mapsto X(t,\omega)$ has càdlàg paths of finite variation for all $\omega\in\Omega$ if and only if $t\mapsto f^n(t,\omega^n)$ is càdlàg and of finite variation for all $\omega^n\in\Omega^n$ and $n\in\amsmathbb{N}_0$. Similarly, any boundedness conditions uniform in $(t,\omega)$ that $X$ satisfies carry over to $(f^n)_{n\in\amsmathbb{N}_0}$ and vise-versa.
\end{remark}

Proposition~\ref{prop:Rep_adapted_process_N} provides a unique jump-count decomposition of optional and predictable processes, and throughout the rest of the paper, when referring to the jump-count decomposition of a process, we refer to the one given by Proposition~\ref{prop:Rep_adapted_process_N}.

We now turn towards the state-wise decomposition of optional and predictable processes given by 
\begin{align*}
  X(t)=\sum_{i\in\mathcal{J}}I_i(t)\ti{X}_i(t)\quad \text{ or }\quad X(t)=\sum_{i\in\mathcal{J}}I_i(t-)\ti{X}_i(t)
\end{align*}
for predictable processes $(\ti{X}_i)_{i\in\mathcal{J}}$. The idea is that $\ti{X}_i(t,\omega)$ always equals the value that $X(t,\omega)$ would have if $Z(t,\omega)$ were equal to $i$, even though $Z(t,\omega)$ for that particular $(t,\omega)$ might be different from $i$. In that sense, $\ti{X}_i$ can be understood as a counterfactual. Following~\cite{Christiansen&Furrer2024}, the construction of such processes $(\ti{X}_i)_{i\in\mathcal{J}}$ requires the construction of counterfactual paths. For this define for each $i\in\mathcal{J}$ the $(t,i)$-stopped path given by
\begin{align*}
  \omega_{i}^t((t_{\ell},z_{\ell})_{\ell\in\amsmathbb{N}_0}):=
  \begin{cases}
    ((0,z_0),\dots,(t_n,z_n),(\infty,\nabla),\dots),\quad &\text{for } t_n<t\leq t_{n+1}, z_n=i\\
    ((0,z_0),\dots,(t_n,z_n),(t,i),(\infty,\nabla),\dots),\quad &\text{for } t_n<t\leq t_{n+1}, z_n\neq i\\
    ((0,i),(\infty,\nabla),\dots),\quad &\text{for } t=0\\
    ((t_{\ell},z_{\ell})_{\ell\in\amsmathbb{N}_0}),\quad &\text{for } t=\infty.
  \end{cases}
\end{align*}
If the current state at time $t$ is equal to $i$, then all future jumps are removed, and the path remains in state $i$ forever. If the current state at time $t$ is different from $i$, then a jump to state $i$ at time $t$ is added, and all future jumps are removed. Let $\omega^{t_1,t_2}_{i,j}$ for $t_1\leq t_2$ denote the composition $\omega_j^{t_2}(\omega_i^{t_1})$, which first adds a jump to $i$ at time $t_1$ (if the state at $t_1$ is different from $i$) followed by a jump to $j$ at time $t_2$. Using this, we construct the two counterfactuals
\begin{align}
  X_i(t,\omega)&:=\sum_{n\in\amsmathbb{N}}I^n(t-)X(t,\omega_i^{\tau_n}(\omega)),\label{eq:CF_sojourn}\\
  X_{ij}(t,\omega)&=\sum_{n\in\amsmathbb{N}_0}I^n(t-)X(t,\omega^{\tau_n,t}_{i,j}(\omega)).\label{eq:CF_jump}
\end{align}
For any $t\geq 0$ the process $X_i(t)$ equals the value $X(t)$ would have whenever $Z(t)=i$ for $t$ in between jump times, while $X_{ij}(t)$ equals the value $X(t)$ would have if $Z$ were to jump from state $i$ to state $j$ at time $t$. Invoking Proposition~\ref{prop:Rep_adapted_process_N} we can also write 
\begin{align}\label{eq:CF_n}
  X_i(t,\omega)=\sum_{n\in\amsmathbb{N}_0}I^n(t-)f^n(t,\xi^n(\omega_i^{\tau_n}))\quad\text{and}\quad X_{ij}(t,\omega)=\sum_{n\in\amsmathbb{N}_0}I^{n}(t-)f^{n+1}(t,\xi^{n+1}(\omega^{\tau_n,t}_{i,j})),
\end{align}
which shows that both (\ref{eq:CF_sojourn}) and (\ref{eq:CF_jump}) define predictable processes. In fact, they are not only predictable, but $i$-predictable:
\begin{definition}
  A process $X$ is $i$-predictable if $X(t)$ is deterministic for $0\leq t\leq\tau_1$ and\\ $\sigma((\tau_{\ell},\zeta_{\ell})_{\ell\leq n-1},\mathds{1}_{(\zeta_{n-1}\neq i)}\tau_n)$-measurable for $\tau_n<t\leq \tau_{n+1}$ for all $n\in\amsmathbb{N}$.
\end{definition}
This is slightly stricter than just being predictable, since by Proposition~\ref{prop:Rep_adapted_process_N}(ii) any predictable process is $\sigma((\tau_{\ell},\zeta_{\ell})_{\ell\leq n})$-measurable when $\tau_n<t\leq \tau_{n+1}$. In the case of (\ref{eq:CF_sojourn}) and (\ref{eq:CF_jump}), we have set $\zeta_n=i$ whenever $\zeta_n\neq i$ or removed the $n$'th jump whenever $\zeta_n=i$, thus in both cases we have removed the dependence on $\zeta_n$ and in the last case we have removed the dependence on $\tau_n$. By p.7 of~\cite{Christiansen&Furrer2024} and Proposition~\ref{prop:Rep_adapted_process_N}(ii) a process $X$ with jump-count decomposition $(f^n)_{n\in\amsmathbb{N}_0}$ being $i$-predictable is equivalent to 
\begin{align}\label{eq:i-predictable}
  \forall n\in\amsmathbb{N}_0:\quad X(t,\omega)=X(t,\omega_i^{\tau_n})\quad \text{and}\quad f^n(t,\xi^n(\omega))=f^n(t,\xi^n(\omega_i^{\tau_n})),\quad \tau_n<t\leq\tau_{n+1}.
\end{align}
Using the concept of $i$-predictability, we now get the following state-wise decomposition result.

\begin{proposition}\label{prop:Rep_adapted_process_J}
  Let $X$ be a real-valued stochastic process.
  \begin{enumerate}
    \item[(i)] $X$ is $\amsmathbb{F}$-optional if and only if there for all $i\in\mathcal{J}$ exist unique $i$-predictable processes $X_i$ and $(X_{ij})_{j:j\neq i}$ satisfying (\ref{eq:CF_sojourn}) and (\ref{eq:CF_jump}) such that 
    \begin{align*}
      X(t)=\sum_{i\in\mathcal{J}}I_i(t)\ti{X}_i(t)\quad\text{with}\quad \ti{X}_i(t)=I_i(t-)X_i(t)+\sum_{j:j\neq i}I_j(t-)X_{ji}(t).
    \end{align*}
    \item[(ii)] $X$ is $\amsmathbb{F}$-predictable if and only if there for all $i\in\mathcal{J}$ exist a unique $i$-predictable process $X_i$ satisfying (\ref{eq:CF_sojourn}) such that 
    \begin{align*}
      X(t)=\sum_{i\in\mathcal{J}}I_i(t-)X_i(t)
    \end{align*}
  \end{enumerate}
\end{proposition} 
\begin{proof}
  (i): We start with the direction "$\Rightarrow$". For this let $t\in[0,\infty)$ be arbitrary and let $(X_i,(X_{ij})_{j:j\neq j})_{i\in\mathcal{J}}$ be the $i$-predictable processes uniquely defined by (\ref{eq:CF_sojourn}) and (\ref{eq:CF_jump}). It remains to check whether the decomposition holds. We begin with the case where $Z$ does not have a jump at time $t$. Then for all $n\in\amsmathbb{N}_0$ the definition of $\ti{X}_i$, Equation (\ref{eq:CF_n}) and Proposition~\ref{prop:Rep_adapted_process_N}(i) yields
  \begin{align*}
    \forall\omega\in (\tau_n<t<\tau_{n+1},\zeta_n=i):\quad \ti{X}_i(t,\omega)=X_i(t,\omega)=f^n(t,\xi^n(\omega_i^{\tau_n}))=f^n(t,\xi^n(\omega))=X(t,\omega),
  \end{align*}
  since $\xi^n(\omega^{\tau_n}_i)=\xi^n(\omega)$ when $\zeta_n=i$. It remains to check the case where $Z$ has a jump at time $t$. Then for all $n\in\amsmathbb{N}_0$ and $j\neq i$ we have by the definition of $\ti{X}_i$, Equation (\ref{eq:CF_n}) and Proposition~\ref{prop:Rep_adapted_process_N}(i) that
  \begin{align*}
    \forall\omega\in(\tau_n=t,\zeta_{n-1}=j,\zeta_n=i):\quad \ti{X}_i(t,\omega)=X_{ji}(t,\omega)=f^n(t,\xi^n(\omega_{j,i}^{\tau_{n-1},t}))=f^n(t,\xi^n(\omega))=X(t,\omega),
  \end{align*}
  as $\xi^n(\omega_{j,i}^{\tau_{n-1},t})=\xi^n(\omega)$ when $\zeta_{n-1}=j$, $\zeta_n=i$ and $\tau_n=t$. This concludes the proof of the direction "$\Rightarrow$". For the direction "$\Leftarrow$", note that $X$ defined by the decomposition consists of sums and products of optional processes yielding optionality of $X$ itself. Finally the definition of $X$ and $\ti{X}_i$ in conjunction with (\ref{eq:i-predictable}) yields $X(t,\omega^{\tau_n}_i)=X_i(t,\omega^{\tau_n}_i)=X_i(t,\omega)$ and
  \begin{align*}
    X(t,\omega^{\tau_n,t}_{i,j})&=X_{ij}(t,\omega^{\tau_n,t}_{i,j})=X_{ij}(t,\omega^{\tau_n}_i(\omega^{\tau_n,t}_{i,j}))=X_{ij}(t,\omega)
  \end{align*}
  for all $\tau_n<t\leq\tau_{n+1}$ and all $n\in\amsmathbb{N}_0$. This concludes the proof of (i).

  (ii): We start with the direction "$\Rightarrow$". Let $t\in[0,\infty)$ be arbitrary and the $i$-predictable $X_i$ be given by (\ref{eq:CF_sojourn}). Then for all $n\in\amsmathbb{N}_0$ Equation (\ref{eq:CF_n}) and Proposition~\ref{prop:Rep_adapted_process_N}(ii) yields
  \begin{align*}
    \forall\omega\in (\tau_n<t\leq \tau_{n+1},\zeta_n=i):\quad X_i(t,\omega)=f^n(t,\xi^n(\omega_i^{\tau_n}))=f^n(t,\xi^n(\omega))=X(t,\omega),
  \end{align*}
  as we again have $\xi^n(\omega^{\tau_n}_i)=\xi^n(\omega)$, concluding the proof of the direction "$\Rightarrow$". The direction "$\Leftarrow$" is proven as in (i). 
\end{proof}

\begin{remark}\label{rem:minimal}
  Note that $\ti{X}_i(t)$ as defined in Proposition~\ref{prop:Rep_adapted_process_J}(i) is equal to $X(t,\omega_i^t)$, which is how the state-wise decomposition is defined in Proposition~6.3 of~\cite{Christiansen&Furrer2024}. They do not explicitly discuss the decomposition of predictable processes, but since any predictable process $X$ is optional, and since $I_i(t-)\ti{X}_i(t)=I_i(t-)X_i(t)$ for all $t\geq 0$, the decomposition $X(t)=\sum_{i\in\mathcal{J}}I_i(t-)\ti{X}_i(t)$ would be just as valid for predictable $X$. But unlike $X_i(t)$, the process $\ti{X}_i(t)$ is only predictable and not $i$-predictable, thus using superfluous information. In that sense using the decomposition $(X_i)_{i\in\mathcal{J}}$ is minimal.
\end{remark}

\begin{remark}\label{rem:path_properties_i}
  Note that since $t\mapsto I_i(t-,\omega)$ and $t\mapsto I^n(t-,\omega)$ for $i\in\mathcal{J}$ and $n\in\amsmathbb{N}_0$ are not right-continuous in the points $t\in\{\tau_n:n\in\amsmathbb{N}_0\}$, the paths $t\mapsto X_i(t,\omega)$, $t\mapsto X_{ij}(t,\omega)$ and $t\mapsto \ti{X}_i(t,\omega)$ might not be right-continuous in those same points, even though $t\mapsto X(t,\omega)$ is. But due to (\ref{eq:CF_n}) and Remark~\ref{rem:path_properties} it always holds that $
  (\ti{X}_i)_{i\in\mathcal{J}}$, $(X_i,(X_{ij})_{j:j\neq i})_{i\in\mathcal{J}}$ are of finite variation, whenever $X$ is of finite variation.
\end{remark}

Proposition~\ref{prop:Rep_adapted_process_J} provides a unique state-wise decomposition of optional and predictable processes, and throughout the rest of the paper, when referring to the state-wise decomposition of a process, we refer to the one given by Proposition~\ref{prop:Rep_adapted_process_J}.

Having introduced both the jump-count representation of Proposition~\ref{prop:Rep_adapted_process_N} and the state-wise representation of Proposition~\ref{prop:Rep_adapted_process_J}, we now restrict ourselves to càdlàg processes of finite variation. In this case, one can obtain dynamic (in the sense of Lebesgue-Stieltjes integration) versions of the decomposition result. Throughout the paper we use for any subintervals $I,J\subseteq [0,\infty)$ the shorthand notation
\begin{align*}
  X(\mathrm{d}t)=H(t)Y(\mathrm{d}t)\quad \forall t\in I \quad \Leftrightarrow\quad \int_J F(\mathrm{d}t)=\int_J H(t)Y(\mathrm{d}t),\quad\forall J\subseteq I.
\end{align*}
\begin{proposition}\label{prop:Rep_dyn_N}
  Let $X$ be a real-valued càdlàg, finite variation process.
  \begin{enumerate}
    \item[(i)] If $X$ is adapted with jump-count decomposition $(f^n)_{n\in\amsmathbb{N}_0}$, then
    \begin{align*}
      X(\mathrm{d}t)&=\sum_{n\in\amsmathbb{N}_0}I^n(t-)\Big(f^n(\mathrm{d}t,\xi^{n})+\sum_{j:j\neq\zeta_n}\ti{\Delta} f^{n+1}_{\zeta_n j}(t,\xi^n) N_{\zeta_n j}(\mathrm{d}t)\Big),\quad\forall t\in (0,\infty)
    \end{align*}
    where $\ti{\Delta} f^{n+1}_{\zeta_n j}(t,\xi^n):=f^{n+1}(t,\xi^{n+1}(\omega^{\tau_n,t}_{\zeta_n,j}))-f^n(t,\xi^{n})$.
    \item[(ii)] If $X$ is predictable with jump-count decomposition $(f^n)_{n\in\amsmathbb{N}_0}$, then
    \begin{align*}
      X(\mathrm{d}t)=\sum_{n\in\amsmathbb{N}_0}I^n(t-)f^n(\mathrm{d}t,\xi_{n}),\quad\forall t\in (0,\infty).
    \end{align*}
  \end{enumerate}
\end{proposition}
\begin{proof}
  (i): Applying integration by parts to the decomposition of Proposition~\ref{prop:Rep_adapted_process_N}(i), we arrive at
  \begin{align*}
    X(\mathrm{d}t)=\sum_{n\in\amsmathbb{N}_0}I^n(t-)f^n(\mathrm{d}t,\xi^n)+\sum_{n\in\amsmathbb{N}_0}f^n(t,\xi^n)I^n(\mathrm{d}t),\forall t\in(0,\infty).
  \end{align*}
  Utilising that $I^n(\mathrm{d}t)=(I^{n-1}(t-)-I^{n}(t-))N(\mathrm{d}t)$ we obtain 
  \begin{align*}
    X(\mathrm{d}t)=\sum_{n\in\amsmathbb{N}_0}I^n(t-)\bigg(f^n(\mathrm{d}t,\xi^n)+\big(f^{n+1}(t,\xi^{n+1})-f^n(t,\xi^n)\big)N(\mathrm{d}t)\bigg),\quad\forall t\in(0,\infty).
  \end{align*}
  The desired result then follows from the fact that $I^n(t-)N(\mathrm{d}t)=I^n(t-)\sum_{j:j\neq\zeta_n}N_{\zeta_n j}(\mathrm{d}t)$ and the fact that $\xi^{n+1}(\omega^{\tau_n,t}_{\zeta_n,j})=\xi^{n+1}$ when $\Delta N_{\zeta_n j}(t)=1$.

  (ii): Since any predictable process is optional, (i) holds true, and we obtain
  \begin{align*}
    X(t)-X(0)-\int_{(0,t]}\sum_{n\in\amsmathbb{N}_0}I^n(s-)f^n(\mathrm{d}s,\xi^n)=\sum_{0<s\leq t}\sum_{n\in\amsmathbb{N}_0}I^n(s-)\sum_{j:j\neq\zeta_n}\ti{\Delta} f^{n+1}_{\zeta_n j}(s,\xi^{n+1})\Delta N_{\zeta_n j}(s),
  \end{align*}
  for $t\geq 0$. The left-hand side is predictable, and consequently, the right-hand side must be as well. However as the processes $\Delta N_{ij}$ are optional but not predictable, the quantities $\ti{\Delta} f^{n+1}_{ij}(s,\xi^{n+1})$ must be zero whenever $\Delta N_{ij}(s)\neq 0$, implying that the entire right-hand side must be zero.
\end{proof}

\begin{proposition}\label{prop:Rep_dyn_J}
  Let $X$ be a real-valued, càdlàg finite variation process.
  \begin{enumerate}
    \item[(i)] If $X$ is adapted with state-wise decomposition $(\ti{X}_i)_{i\in\mathcal{J}}$ then
    \begin{align*}
      X(\mathrm{d}t)&=\sum_{i\in\mathcal{J}}I_i(t-)\bigg(\ti{X}_i(\mathrm{d}t)+\sum_{j:j\neq i}\ti{X}_j(t)-\ti{X}_i(t) N_{ij}(\mathrm{d}t)\bigg)\\
      &=\sum_{i\in\mathcal{J}}I_i(t-)\bigg(X_i(\mathrm{d}t)+\sum_{j:j\neq i}\ti{\Delta} X_{ij}(t) N_{ij}(\mathrm{d}t)\bigg),\quad\forall t\in(0,\infty),
    \end{align*}
    where $\ti{\Delta} X_{ij}(t):=X_{ij}(t)-X_i(t)$.
    \item[(ii)] If $X$ is predictable with state-wise decomposition $(X_i)_{i\in\mathcal{J}}$ then
    \begin{align*}
      X(\mathrm{d}t)=\sum_{i\in\mathcal{J}}I_i(t-)\ti{X}_i(\mathrm{d}t)=\sum_{i\in\mathcal{J}}I_i(t-)X_i(\mathrm{d}t),\quad\forall t\in(0,\infty).
    \end{align*}
  \end{enumerate}
\end{proposition}
\begin{proof}
Note that by (\ref{eq:CF_n}), Proposition~\ref{prop:Rep_dyn_N}(ii), and the fact that $\sum_{i\in\mathcal{J}}I_i(t-)=1$ we obtain 
\begin{align*}
    I_i(t-)I^n(t-)f^n(\mathrm{d}t,\xi^n)=I_i(t-)I^n(t-)f^n(\mathrm{d}t,\xi^n(\omega^{\tau_n}_i)=I_i(t-)I^n(t-)\ti{X}_i(\mathrm{d}t)=I_i(t-)I^n(t-)X_i(\mathrm{d}t)
\end{align*}
By (\ref{eq:CF_n}) and the definition of $\ti{X}_i$ we obtain that 
\begin{align*}
    I_i(t-)I^n(t-)(f^{n+1}(t,\xi^{n+1}(\omega^{\tau_n,t}_{i,j}))-f^n(t,\xi^n))&=I_i(t-)I^n(t-)(X_{ij}(t)-X_i(t))\\
    &=I_i(t-)I^n(t-)(\ti{X}_j(t)-\ti{X}_i(t)).
\end{align*}
The desired results follow now directly from Proposition~\ref{prop:Rep_dyn_N}.
\end{proof}

\subsection{The probability measure}
In order to specify a probability measure on the filtered canonical measurable space $(\Omega,\mathcal{F},\amsmathbb{F})$, we will use the approach developed in~\cite{Christiansen&Furrer2024}, which constructs a probability measure by specifying the cumulative transition rates and an initial distribution. 
\begin{definition}\label{def:Cumulative_transition_rates}
  Given a probability measure $\amsmathbb{P}$ the predictable processes $(\Lambda_{ij})_{i\neq j}$ with $\Lambda_{ij}:[0,\infty)\times\Omega\rightarrow [0,\infty)$ are cumulative transition rates if for all $n\in\amsmathbb{N}_0$
\begin{align}\label{eq:Cumulative_transition_rates}
  I_i(t-)\Lambda_{ij}(\mathrm{d}t)=I_i(t-)\frac{\amsmathbb{E}[\mathds{1}_{(\tau_{n+1}\geq t)}N_{ij}(\mathrm{d}t)|\mathcal{F}_{\tau_n}]}{\amsmathbb{E}[\mathds{1}_{(\tau_{n+1}\geq t)}|\mathcal{F}_{\tau_n}]},\quad \tau_n<t\leq\tau_{n+1}.
\end{align}
\end{definition}
Note that (\ref{eq:Cumulative_transition_rates}) does not uniquely define a set of processes, which then are the cumulative transition rates, as it does not define the values of $\Lambda_{ij}$ whenever $Z(t-)\neq i$. As an example, let $C^j$ be the unique predictable compensator of $N_j(t)=\sum_{i:i\neq j}N_{ij}(t)$ given by 
\begin{align*}
    C^j(\mathrm{d}t)=\sum_{i:i\neq j}I_i(t-)\Lambda_{ij}(\mathrm{d}t),\quad C^j(0)=0.
\end{align*}
By Proposition~\ref{prop:Rep_dyn_J}(ii) and Remark~\ref{rem:minimal} it holds that 
\begin{align*}
    C^j(\mathrm{d}t)=\sum_{i:i\neq j}I_i(t-)\ti{C}_{i}^j(\mathrm{d}t)=\sum_{i:i\neq j}I_i(t-)C_{i}^j(\mathrm{d}t),\quad C^j(0)=0,
\end{align*}
meaning that both $(\ti{C}_i^j)_{i:i\neq j}$ and $(C_i^j)_{i:i\neq j}$ for $j\in\mathcal{J}$ satisfy Definition~\ref{def:Cumulative_transition_rates}. Therefore, we have to choose a unique set of processes $(\Lambda_{ij})_{i\neq j}$, which then should end up satisfying Definition~\ref{def:Cumulative_transition_rates}. Thus following~\cite{Christiansen&Furrer2024} we make the following assumptions about $(\Lambda_{ij})_{i\neq j}$:

\begin{assumption}\label{ass:cumulative_tr}
  Assume that 
  \begin{enumerate}
    \item[(i)] For each $i\in\mathcal{J}$ the processes $\Lambda_i$ and $(\Lambda_{ij})_{j:j\neq i}$ are $i$-predictable.
    \item[(ii)] $t\mapsto\Lambda_{ij}(t,\omega)$ is non-negative, non-decreasing and càdlàg with 
    \begin{align*}
      \sum_{j\in\mathcal{J}}\Delta\Lambda_{ij}(t,\omega)\leq 1,\quad \forall t\in[0,\infty),\,\omega\in\Omega.
    \end{align*}
    \item[(iii)] For each $i,j\in\mathcal{J}$ with $i\neq j$ there exists a non-decreasing, non-negative càdlàg function $\alpha_{ij}:[0,\infty)\rightarrow [0,\infty)$, with $\alpha_{ij}(0)=0$ and bounded on compact sets, such that 
    \begin{align*}
      \Lambda_{ij}(s,\omega)-\Lambda_{ij}(t,\omega)\leq \alpha_{ij}(s)-\alpha_{ij}(t),\quad \forall 0\leq t\leq s,\,\forall\omega\in\Omega.
    \end{align*}
  \end{enumerate}
\end{assumption}
\begin{remark}
  Note that compared to~\cite{Christiansen&Furrer2024}, we make more restrictive assumptions, as we do not allow the cumulative transition rates to have poles followed by reset points. The exclusion of poles and reset points brings our setup closer to the marked point process literature~\cite{Jacobsen2006,Last&Brandt1995}, which does not include poles and reset points.
\end{remark}
The $i$-predictability of Assumption~\ref{ass:cumulative_tr}(i) in conjunction with (\ref{eq:i-predictable}) ensures that
\begin{align*}
  \sum_{n\in\amsmathbb{N}_0}I_n(t-)C^j(\mathrm{d}t,\omega_i^{\tau_n})=\sum_{n\in\amsmathbb{N}_0}I_n(t-)\Lambda_{ij}(\mathrm{d}t,\omega_i^{\tau_n})=\Lambda_{ij}(\mathrm{d}t,\omega),\quad\forall (t,\omega)\in [0,\infty)\times\Omega,
\end{align*}
which therefore guarantees that our chosen cumulative transition rates $(\Lambda_{ij})_{i:i\neq j}$ for each $j\in\mathcal{J}$ equal the state-wise decomposition of the compensator $C_j$, while using the minimal amount of information, see also~\ref{rem:minimal}. Assumption~\ref{ass:cumulative_tr}(ii) ensures that the transition probabilities are always in the interval $[0,1]$, while Assumption~\ref{ass:cumulative_tr}(iii) ensures non-explosiveness. Theorem~4.1 of~\cite{Christiansen&Furrer2024} now immediately gives the following result:
\begin{theorem}\label{th:Probability_measure_existence}
  Let $\pi:\mathcal{J}\rightarrow [0,1]$ with $\sum_{i\in\mathcal{J}}\pi(i)=1$ and assume that $(\Lambda_{ij})_{i\neq j}$ satisfy Assumption~\ref{ass:cumulative_tr}. Then there exists a uniqe probability measure $\amsmathbb{P}^{\Lambda}$ and unique $\mathcal{F}_{t-}$-measurable probability kernels $(\amsmathbb{P}_{t,i}^{\Lambda})_{t\in [0,\infty),i\in\mathcal{J}}$ on $(\Omega,\mathcal{F})$ such that 
  \begin{enumerate}
    \item[(i)] For all $i\in\mathcal{J}$ it holds that $\amsmathbb{P}^{\Lambda}(Z(0)=i)=\pi(i)$.
    \item[(ii)] For all $A\in\mathcal{F}$ and $i\in\mathcal{J}$ it holds almost surely that 
    \begin{align*}
      I_i(t)\amsmathbb{P}_{t,i}^{\Lambda}(A)=I^i(t)\amsmathbb{P}^{\Lambda}(A|\mathcal{F}_t).
    \end{align*}
    \item[(iii)] Let $\amsmathbb{E}_{t,i}^{\Lambda}$ denote the expectation under $\amsmathbb{P}_{t,i}^{\Lambda}$. Then for all $i\neq j$ the processes $(\Lambda_{ij})_{i\neq j}$ satisfy
    \begin{align*}
      \Lambda_{ij}(\mathrm{d}t)=\frac{\amsmathbb{E}_{\tau_n,i}^{\Lambda}[\mathds{1}_{(\tau_{n+1}\geq t)}N_{ij}(\mathrm{d}t)]}{\amsmathbb{E}^{\Lambda}_{\tau_n,i}[\mathds{1}_{(\tau_{n+1}\geq t)}]},\quad \tau_n<t\leq\tau_{n+1},\quad \forall n\in\amsmathbb{N}_0.
    \end{align*}
  \end{enumerate}
\end{theorem}
Using the kernels $\amsmathbb{P}_{t,i}^{\Lambda}$ we can for any $\mathcal{F}$-measurable random variable $Y$, which for all $i\in\mathcal{J}$ is $\amsmathbb{P}_{t,i}$-integrable, define 
\begin{align*}
  \amsmathbb{E}^{\Lambda}_{t,Z(t)}[Y]:=\sum_{i\in\mathcal{J}}I_i(t)\amsmathbb{E}^{\Lambda}_{t,i}[Y],
\end{align*}
which has the following properties:
\begin{proposition}\label{prop:conditional_expectation}
  Let $\pi$ be any initial distribution, let $\Lambda=(\Lambda_{ij})_{i\neq j}$ satisfy Assumption~\ref{ass:cumulative_tr} and let $Y$ be an $\mathcal{F}$-measurable random variable. 
  \begin{enumerate}
    \item[(i)] If $Y$ is $\amsmathbb{P}_{t,i}$-integrable for all $i\in\mathcal{J}$ then 
    \begin{align*}
      \amsmathbb{E}^{\Lambda}_{t,Z(t)}[Y]=\amsmathbb{E}^{\Lambda}[Y|\mathcal{F}_t],\quad \amsmathbb{P}^{\Lambda}-\text{a.s.}
    \end{align*}
    \item[(ii)] If $Y$ is $\mathcal{F}_{t-}$-measurable and $\amsmathbb{P}_{t,i}$-integrable for all $i\in\mathcal{J}$ then
    \begin{align*}
      \forall i\in\mathcal{J}: \amsmathbb{E}^{\Lambda}_{t,i}[Y](\omega)=Y(\omega),\quad\forall\omega\in\Omega.
    \end{align*}
    \item[(iii)] If $Y$ is $\mathcal{F}_t$-measurable and $\amsmathbb{P}_{t,i}$-integrable for all $i\in\mathcal{J}$ then
    \begin{align*}
      \forall i\in\mathcal{J}: \amsmathbb{E}^{\Lambda}_{t,i}[Y](\omega)&=\ti{Y}_i(\omega),\quad\forall\omega\in\Omega,\\
      \amsmathbb{E}^{\Lambda}_{t,Z(t)}[Y](\omega)&=Y(\omega),\quad\forall\omega\in\Omega.
    \end{align*}
  \end{enumerate}
\end{proposition}
\begin{proof}
  (i): This follows from Theorem~\ref{th:Probability_measure_existence}(ii) by splitting $Y$ in a positive and negative part, approximating each from below by an increasing sequence of simple functions, and applying the dominated convergence theorem (for both conditional and unconditional expectations).

  (ii): This follows from Proposition~4.4(i) of~\cite{Christiansen&Furrer2024} by splitting $Y$ in a positive and negative part, approximating each from below by an increasing sequence of simple functions, and applying the dominated convergence theorem.

  (iii): By Proposition~\ref{prop:Rep_adapted_process_J}(i), the just proven (ii) and Proposition~4.4(i) of~\cite{Christiansen&Furrer2024} it holds that
  \begin{align*}
    \amsmathbb{E}_{t,i}^{\Lambda}[Y]=\sum_{j\in\mathcal{J}}\amsmathbb{E}_{t,i}^{\Lambda}[I_j(t)\ti{Y}_j]=\sum_{j\in\mathcal{J}}\ti{Y}_j\amsmathbb{E}_{t,i}^{\Lambda}[I_j(t)]=\sum_{j\in\mathcal{J}}\ti{Y}_j\mathds{1}_{j=i}=\ti{Y}_i,\quad \forall i\in\mathcal{J},
  \end{align*}
  which is the first statement. Using this, we thus obtain 
  \begin{align*}
    \amsmathbb{E}_{t,Z(t)}^{\Lambda}[Y]=\sum_{i\in\mathcal{J}}I_i(t)\amsmathbb{E}_{t,i}^{\Lambda}[Y]=\sum_{i\in\mathcal{J}}I_i(t)\ti{Y}_i=Y.
  \end{align*}
\end{proof}

The non-explosiveness assumption, Assumption~\ref{ass:cumulative_tr}(iii), implies the following integrability condition.
\begin{proposition}\label{prop:exp_bound_N}
  Let $c>0$. Then for any $\Lambda=(\Lambda_{ij})_{i\neq j}$ satisfying Assumption~\ref{ass:cumulative_tr} it holds that
  \begin{align*}
    \sup_{0\leq t\leq s\leq T}\sup_{\omega\in\Omega}\amsmathbb{E}_{t,i}^{\Lambda}\big[e^{c (N(s)-N(t))}\big](\omega)<\infty,\quad\forall i\in\mathcal{J}.
  \end{align*}
\end{proposition}
\begin{proof}
  By the change of variables formula for càdlàg finite variation functions, we obtain 
  \begin{align*}
    e^{c(N(s)-N(t))}=1+\int_{(t,s]}(e^c-1)e^{c (N(u-)-N(t))}N(\mathrm{d}u).
  \end{align*}
  Taking the expectation and applying Proposition~4.2 of~\cite{Christiansen&Furrer2024} in conjunction with the monotone convergence theorem yields
  \begin{align*}
    \amsmathbb{E}_{t,i}^{\Lambda}\big[e^{c(N(s)-N(t))}\big]=1+\amsmathbb{E}_{t,i}^{\Lambda}\bigg[\int_{(t,s]}(e^c-1)e^{c (N(u-)-N(t))}\sum_{i,j:i\neq j}\Lambda_{ij}(\mathrm{d}u)\bigg]
  \end{align*}
  Using Assumption~\ref{ass:cumulative_tr}(iii) and Tonelli's theorem we get 
  \begin{align*}
    \amsmathbb{E}_{t,i}^{\Lambda}\big[e^{c(N(s)-N(t))}\big]\leq 1+\int_{(t,s]}\amsmathbb{E}_{t,i}^{\Lambda}\big[e^{c (N(u-)-N(t))}\big](e^c-1)\sum_{i,j:i\neq j}\alpha_{ij}(\mathrm{d}u).
  \end{align*}
  Set $\alpha(t):=(e^c-1)\sum_{i,j:i\neq j}\alpha_{ij}(t)$. An application of Gronwall's inequality (Lemma~\ref{lem:gronwall}) yields
  \begin{align*}
    \amsmathbb{E}_{t,i}^{\Lambda}\big[e^{c(N_{ij}(s)-N_{ij}(t))}\big]\leq\frac{\mathcal{E}_{\alpha}(s)}{\mathcal{E}_{\alpha}(t)}\leq \mathcal{E}_{\alpha}(T).
  \end{align*}
  As this bound does not depend on $s$, $t$ or $\omega$, taking the suprema yields the desired result. 
\end{proof}

In many practical applications, one assumes the cumulative transition rates to be absolutely continuous with respect to the Lebesgue measure, that is 
\begin{equation}\label{eq:absolutely_continuous_Lambda}
  \Lambda_{ij}(\mathrm{d}t)=\mu_{ij}(t)\mathrm{d}t,\quad t\in[0,\infty),\,i,j\in\mathcal{J},i\neq j
\end{equation}
for non-negative, $i$-predictable transition intensities $(\mu_{ij})_{i\neq j}$. In order for $(\Lambda_{ij})_{i\neq j}$ defined as in (\ref{eq:absolutely_continuous_Lambda}) to satisfy Assumption~\ref{ass:cumulative_tr} and Theorem~\ref{th:Probability_measure_existence} to remain valid, the following set of assumptions for $(\mu_{ij})_{i\neq j}$ is sufficient:
\begin{assumption}\label{ass:intensities}
  Assume that 
  \begin{enumerate}
    \item[(i)] For any $i\in\mathcal{J}$ the processes $(\mu_{ij})_{j:j\neq i}$ are $i$-predictable.
    \item[(ii)] For each $i,j$ with $i\neq j$ there exists a measurable, non-negative function $\alpha_{ij}:[0,\infty)\rightarrow [0,\infty)$ bounded on compact sets such that $\mu_{ij}(t,\omega)\leq \alpha_{ij}(t)$ for all $t\geq 0$ and $\omega\in\Omega$.
  \end{enumerate}
\end{assumption}
Assumption~\ref{ass:intensities}(i)+(ii) guarantee the validity of Assumption~\ref{ass:cumulative_tr}(i)+(iii), while the non-negativity of the $(\mu_{ij})_{i\neq j}$ and (\ref{eq:absolutely_continuous_Lambda}) implies the automatic satisfaction of Assumption~\ref{ass:cumulative_tr}(ii).

\begin{example}[Smooth semi-Markov process]\label{ex:semi-Markov_1}
  In the case of the smooth semi-Markov process, the cumulative transition rates are of the form
  \begin{align*}
    \Lambda_{ij}(\mathrm{d}t)=\mu_{ij}(t)\mathrm{d}t\quad\text{with}\quad \mu_{ij}(t)=\ti{\mu}_{ij}(t,t-\tau(t-)),
  \end{align*}
  where $\ti{\mu}_{ij}:[0,\infty)\times [0,\infty)\rightarrow\amsmathbb{R}$ is a non-negative measurable function bounded on compact sets and 
  \begin{align*}
    \tau(t):=\sum_{n\in\amsmathbb{N}_0}I_n(t)\tau_n
  \end{align*}
  is the most recent jump time. Thus, the transition intensities are allowed to depend on the duration $U(t):=t-\tau(t)$ since the last jump. The issue here is that $\mu_{ij}$ defined in this way is not $i$-predictable. This is best seen by finding a violation of (\ref{eq:i-predictable}). Evaluating $\tau(t-,\omega)$ in $\omega^{\tau_n}_i$ yields
  \begin{align*}
    \forall n\in\amsmathbb{N}_0: \quad\tau(t-,\omega^{\tau_n}_i)=\mathds{1}_{(\zeta_{n-1}=i)}\tau_{n-1}+\mathds{1}_{(\zeta_{n-1}\neq i)}\tau_n,\quad \tau_n<t\leq\tau_{n+1},
  \end{align*}
  as setting $\zeta_n$ to be equal to $i$ when $\zeta_{n-1}$ already equals $i$ removes the $n$'th jump. We therefore obtain 
  \begin{align*}
    \mu_{ij}(t)=\ti{\mu}_{ij}(t,t-\tau_n)\neq\ti{\mu}_{ij}(t,t-\tau_{n-1})=\mu_{ij}(t,\omega^{\tau_n}_i),\quad\text{whenever },\tau_n<t\leq\tau_{n+1},\zeta_{n-1}=i
  \end{align*}
  which violates (\ref{eq:i-predictable}). The process $\mu_{ij}$ is not its own counterfactual. Instead, we can via (\ref{eq:CF_n}) define 
  \begin{align*}
    \tau_i(t):=I^0(t-)\tau_0+\sum_{n\in\amsmathbb{N}}I^n(t-)(\mathds{1}_{(\zeta_{n-1}=i)}\tau_{n-1}+\mathds{1}_{(\zeta_{n-1}\neq i)}\tau_n),\quad i\in\mathcal{J},
  \end{align*}
  and set $\mu_{ij}(t):=\ti{\mu}_{ij}(t,t-\tau_i(t))$, which is $i$-predictable. 
\end{example}

\section{Optional projections}\label{sec:3}
The central object of the robust optimisation problems we are interested in, are quantities such as the expected present value of future insurance payments or the expected present value of future utility. In both cases, the mathematical object of interest is thus the optional projection of either the (discounted) insurance payment process or the (discounted) utility process. As we are using the canonical filtered probability space, the filtration and the probability space are not complete, and thus the standard results guaranteeing the existence of the optional projection do not apply. Therefore, in Subsection~\ref{subsec:3.1} we start by showing in Proposition~\ref{prop:optional_projection_path} that the optional projection exists under minimal regularity conditions. This is followed by Theorem~\ref{th:optional_projection} and Theorem~\ref{th:optional_projection_discounted}, which for two different types of processes show that the optional projection is characterised by a BSDE. Note that Theorem~\ref{th:optional_projection} and Theorem~\ref{th:optional_projection_discounted} are closely related to Theorem~6.4 and Theorem~7.2 of~\cite{Christiansen&Furrer2024}. While they formulate their result in terms of state-wise conditional expectation processes and stochastic Thiele equations, we use the optional projection and BSDEs and generalise the type of process that the optional projection can be taken of. We end Subsection~\ref{subsec:3.1} by using the link between the jump-count and state-wise decomposition established in Section~\ref{sec:2} to show in Proposition~\ref{prop:BSDE_equiv_ODE} and Proposition~\ref{prop:BSDE_equiv_ODE_discounted} that the existence and uniqueness of the BDSs obtained in Theorem~\ref{th:optional_projection} and Theorem~\ref{th:optional_projection_discounted} is equivalent to the existence and uniqueness of an infinite-dimensional system of Lebesgue-Stieltjes integral equations.  In Subsection~\ref{subsec:3.2} we wish to specify the optional projection of a process, where both the cumulative transition rates determining the probability measure under which the optional projection is taken and the process which the optional projection is taken of can depend on the optional projection itself. This circular case brings us to the two main results of this section, Theorem~\ref{th:BSDE_existence} and Theorem~\ref{th:BSDE_existence_discounted}. They show existence and uniqueness of two worst-case BSDEs that characterise this type of circular optional projection, which will turn out to characterise the worst- and best-case prospective reserves of life insurance contracts.

\subsection{Optional projections and BSDEs}\label{subsec:3.1} Let $(\Omega,\mathcal{F},\amsmathbb{F})$ be the canonical filtered measurable space, let $\pi$ be any initial distribution and let $\Lambda=(\Lambda_{ij})_{i\neq j}$ be any set of cumulative transition rates satisfying Assumption~\ref{ass:cumulative_tr}. We endow the space $(\Omega,\mathcal{F},\amsmathbb{F})$ with the probability measure $\amsmathbb{P}^{\Lambda}$ generated by the initial distribution $\pi$ and the cumulative transition rates $\Lambda=(\Lambda_{ij})_{i\neq j}$ as given by Theorem~\ref{th:Probability_measure_existence} and let $(\amsmathbb{P}_{t,i}^{\Lambda})_{t\in [0,\infty),i\in\mathcal{J}}$ be the associated kernels. 

Let $T>0$ be a deterministic terminal time and for a càdlàg process of finite variation $C$ and an $\mathcal{F}_T$-measurable random variable $Y$ we define the process $L(t,T)$ by 
\begin{align*}
  L(t,T):=C(T)-C(t)+Y,\quad t\in [0,T].
\end{align*}
We are now interested in the optional projection of $L(t,T)$, that is a càdlàg process $X$ which satisfies
\begin{align*}
   X(t)=\amsmathbb{E}^{\Lambda}[L(t,T)|\mathcal{F}_t]\quad \amsmathbb{P}^{\Lambda}-\text{a.s.}\quad \forall t\in [0,T].
\end{align*}
As the filtered probability space $(\Omega,\mathcal{F},\amsmathbb{F},\amsmathbb{P}^{\Lambda})$ is not complete, it does not satisfy the usual conditions, and therefore we cannot rely on the standard results which guarantee the existence of such an optional projection. We therefore aim to show that the process 
\begin{align*}
  X(t):=\amsmathbb{E}^{\Lambda}_{t,Z(t)}[L(t,T)]=\sum_{i\in\mathcal{J}}I_i(t)\amsmathbb{E}_{t,i}[L(t,T)]
\end{align*}
is an optional projection of $L(t,T)$. We start with the simple case $L(t,T)=Y$ and prove the following extension of the martingale representation result of Remark~6.5 in~\cite{Christiansen&Furrer2024}.

\begin{proposition}\label{prop:MG-representation}
  Let $Y$ be an $\mathcal{F}_T$-measurable, real-valued random variable such that 
  \begin{align*}
    \forall i\in\mathcal{J}:\quad \sup_{0\leq t\leq T}\amsmathbb{E}_{t,i}^{\Lambda}[|Y|](\omega)<\infty,\quad\forall\omega\in\Omega.
  \end{align*}
  Then the process $X(t)=\amsmathbb{E}_{t,Z(t)}^{\Lambda}[Y]$ is an optional projection with paths of finite variation and 
  \begin{align*}
    X(t)=Y-\int_{(t,T]}\sum_{i,j:i\neq j}(\ti{X}_{j}(t)-\ti{X}_i(t))(N_{ij}(\mathrm{d}s)-I_i(t-)\Lambda_{ij}(\mathrm{d}t)),\quad t\in[0,T].
  \end{align*}
\end{proposition}
\begin{proof}
  The fact that $X(t)=\amsmathbb{E}^{\Lambda}[Y|\mathcal{F}_t]$ $\amsmathbb{P}^{\Lambda}$--a.s. for all $t\in[0,T]$ follows directly from Proposition~\ref{prop:conditional_expectation}(i). It remains to show the path properties. We start with the simple case $Y=\mathds{1}_A$ for some $A\in\mathcal{F}_T$. Set $\ti{X}_i(t)=\amsmathbb{E}_{t,i}^{\Lambda}[\mathds{1}_A]$ for $i\in\mathcal{J}$. Then by Theorem~5.2 of~\cite{Christiansen&Furrer2024} the processes 
  \begin{align*}
    (t,\omega)\mapsto\int_{(0,t]}I_i(s-)\ti{X}_i(\mathrm{d}s),\quad i\in\mathcal{J}
  \end{align*}
  are predictable, càdlàg and of finite variation, which implies that 
  \begin{align*}
    X(t)=\sum_{i\in\mathcal{J}}I_i(T)\ti{X}_i(T)-\sum_{i\in\mathcal{J}}\int_{(t,T]}I_i(s-)\ti{X}_i(\mathrm{d}s)-\int_{(t,T]}\sum_{i,j:i\neq j}(\ti{X}_j(s)-\ti{X}_i(s))N_{ij}(\mathrm{d}s).
  \end{align*}
  Since $t\mapsto N_{ij}(t,\omega)$ is càdlàg and of finite variation for all $i,j\in\mathcal{J}:i\neq j$, and $t\mapsto\ti{X}_i(t,\omega)$ is bounded for all $i\in\mathcal{J}$ it follows that $X$ is càdlàg and of finite variation. The integration by parts formula now yields $X(t)=\sum_{i\in\mathcal{J}}I_i(t)\ti{X}_i(t)=\amsmathbb{E}_{t,Z(t)}^{\Lambda}[\mathds{1}_A]$. By Theorem~5.2 of~\cite{Christiansen&Furrer2024} it also holds that 
  \begin{align*}
    I_i(t-)\ti{X}_i(\mathrm{d}t)=-\sum_{j:j\neq i}(\ti{X}_j(t)-\ti{X}_i(t))I_i(t-)\Lambda_{ij}(\mathrm{d}t),\quad \ti{X}_i(T)=\amsmathbb{E}_{T,i}^{\Lambda}[\mathds{1}_A],\quad \forall i\in\mathcal{J},
  \end{align*}
  which, when inserted into the above expression for $X$, yields the desired
  \begin{align*}
    X(t)=\mathds{1}_A-\sum_{i,j:i\neq j}\int_{(t,T]}(\ti{X}_{j}(s)-\ti{X}_i(s))(N_{ij}(\mathrm{d}s)-I_i(s-)\Lambda_{ij}(\mathrm{d}s)),\quad \forall t\in [0,T].
  \end{align*}
  Next, we let $Y$ be non-negative such that the above assumed integrability condition is satisfied. Set $X(t):=\amsmathbb{E}^{\Lambda}_{t,Z(t)}[Y]$ and $\ti{X}_i(t)=\amsmathbb{E}_{t,i}^{\Lambda}[Y]$ for $i\in\mathcal{J}$. Then there exists $(Y_m)_{m\in\amsmathbb{N}}$ given by 
  \begin{align*}
    Y_m(\omega)=\sum_{\ell=1}^m a_{\ell,m}\mathds{1}_{A_{\ell,m}}(\omega)\quad\text{with}\quad \sup_{m\in\amsmathbb{N}}Y_m(\omega)=Y(\omega),\quad \forall\omega\in\Omega,
  \end{align*}
  where $(a_{\ell,m})_{\ell=1,\ldots,n}\subset[0,\infty)$ and $A_{\ell,m}\subset\mathcal{F}_T$ for all $m\in\amsmathbb{N}$. The dominated convergence theorem then yields
  \begin{align*}
    \ti{X}_i(t)=\amsmathbb{E}_{t,i}^{\Lambda}\Big[\lim_{m\rightarrow\infty}Y_m\Big]=\lim_{m\rightarrow\infty}\amsmathbb{E}_{t,i}[Y_m]=\lim_{m\rightarrow\infty}\sum_{\ell=1}^m a_{\ell,m}\amsmathbb{E}_{t,i}^{\Lambda}[\mathds{1}_{A_{\ell,m}}],\quad \forall t\in [0,T],\,\forall i\in\mathcal{J}.
  \end{align*}
  Combining this with the fact that our desired result holds for indicator random variables, we obtain
  \begin{align*}
    X(s)-X(t)&=\lim_{m\rightarrow\infty}\sum_{\ell=1}^m a_{\ell,m}(\amsmathbb{E}_{s,Z(s)}^{\Lambda}[\mathds{1}_{A_{\ell,m}}]-\amsmathbb{E}_{t,Z(t)}^{\Lambda}[\mathds{1}_{A_{\ell,m}}])\\
    &=-\lim_{m\rightarrow\infty}\sum_{\ell=1}^n a_{\ell,m}\bigg(\sum_{i,j:i\neq j}\int_{(t,s]}(\amsmathbb{E}_{u,j}^{\Lambda}[\mathds{1}_{A_{\ell,m}}]-\amsmathbb{E}_{u,i}^{\Lambda}[\mathds{1}_{A_{\ell,m}}])M_{ij}(\mathrm{d}u)\bigg)\\
    &=-\lim_{m\rightarrow\infty}\sum_{i,j:i\neq j}\int_{(t,s]}(\amsmathbb{E}_{u,j}^{\Lambda}[Y_m]-\amsmathbb{E}_{u,i}^{\Lambda}[Y_m])M_{ij}(\mathrm{d}u),\quad0\leq t\leq s\leq T.
  \end{align*}
  The fact that all $Y_m$ are dominated by $Y$ in combination with the above assumed integrability condition allows another application of the dominated convergence theorem, which yields
  \begin{align*}
    X(s)-X(t)=-\sum_{i,j:i\neq j}\int_{(t,s]}(\ti{X}_j(u)-\ti{X}_i(u))M_{ij}(\mathrm{d}u),\quad \forall 0\leq t\leq s\leq T,
  \end{align*}
  which together with $X(T)=\amsmathbb{E}_{T,Z(T)}^{\Lambda}[Y]=Y$ yields the desired result. That $t\mapsto X(t,\omega)$ is càdlàg and of finite variation follows again from the fact that $t\mapsto N_{ij}(t,\omega)$ and $t\mapsto\Lambda_{ij}(t,\omega)$ are càdlàg and of finite variation for all $i,j\in\mathcal{J}:i\neq j$, and that $t\mapsto\ti{X}_i(t,\omega)$ and $t\mapsto I_i(t-,\omega)$ are bounded for all $i\in\mathcal{J}$. The result for the case of real-valued random variables $Y$ now follows from the result for non-negative random variables and from the fact that $Y=Y^+-Y^-$, where $Y^+=\max(0,Y)$ and $Y^-=-\min(0,Y)$ both are both non-negative.
\end{proof}

\begin{proposition}\label{prop:optional_projection_path}
  Assume that $C$ is càdlàg and of finite variation and that 
  \begin{align*}
    \forall i\in\mathcal{J}:\quad \sup_{t\in[0,T]}\amsmathbb{E}_{t,i}^{\Lambda}[|L(t,T)|](\omega)<\infty,\quad \forall\omega\in\Omega.
  \end{align*}
  Then $X(t):=\amsmathbb{E}^{\Lambda}_{t,Z(t)}[L(t,T)]$ is an optional projection of $L(t,T)$ with paths of finite variation.
\end{proposition}
\begin{proof}
  Proposition~\ref{prop:Rep_adapted_process_J}(i) and Proposition~\ref{prop:conditional_expectation}(iii) yield
  \begin{align*}
    \amsmathbb{E}_{t,i}^{\Lambda}[|L(0,T)|](\omega)\leq |C(0,\omega)|+|\ti{C}_i(t,\omega)|+\amsmathbb{E}_{t,i}^{\Lambda}[|L(t,T)|](\omega)<\infty,\quad \forall (t,\omega)\in[0,T]\times\Omega.
  \end{align*}
  Since $C$ has paths of finite variation, so has $\ti{C}_i$, see Remark~\ref{rem:path_properties_i}, and thus we obtain
  \begin{align*}
    \sup_{t\in[0,T]}\amsmathbb{E}_{t,i}^{\Lambda}[|L(0,T)|](\omega)\leq |C(0,\omega)|+|\ti{C}_i(0,\omega)|+|\ti{C}_i|(T,\omega)+\sup_{t\in[0,T]}\amsmathbb{E}_{t,i}^{\Lambda}[|L(t,T)|](\omega)<\infty,\quad \forall \omega\in \Omega,
  \end{align*}
  where $|\ti{C}_i|(T,\omega)$ denotes the variation of $t\mapsto\ti{C}_i(t,\omega)$ on the interval $[0,T]$. By Proposition~\ref{prop:MG-representation} we thus obtain that $W(t)=\amsmathbb{E}_{t,Z(t)}[L(0,T)]=C(t)-C(0)+\amsmathbb{E}_{t,Z(t)}^{\Lambda}[L(t,T)]$ is càdlàg and of finite variation, which implies the same for $X(t)=\amsmathbb{E}_{t,Z(t)}^{\Lambda}[L(t,T)]$.
\end{proof}

Now that we have proven the existence of the optional projection, we turn towards establishing that the optional projection constructed in Proposition~\ref{prop:optional_projection_path} can be characterised by a BDS. This requires additional integrability conditions, as one has to establish a priori that the integral on the right-hand side of the martingale representation result of Proposition~\ref{prop:MG-representation} is a martingale. For this, we follow~\cite{Christiansen&Furrer2024} and define the following class of processes, which guarantees sufficient integrability, regardless of which specific pair $(\pi,\Lambda)$ is used to generate $\amsmathbb{P}^{\Lambda}$ and $(\amsmathbb{P}_{t,i}^{\Lambda})_{t\in [0,\infty),i\in\mathcal{J}}$.

\begin{definition}
  Let $m\in\amsmathbb{N}_0$ and let $\amsmathbb{Y}^+_m$ denote the set of all $\amsmathbb{F}$-adapted, non-negative, non-decreasing, univariate càdlàg processes $X$ such that 
  \begin{align*}
    X(s,\omega)-X(t,\omega)\leq g(t,s)\big(1+N(s,\omega)-N(t,\omega)\big)^m,\quad \forall 0\leq t\leq s,\,\forall\omega\in\Omega
  \end{align*}
  for some non-negative measurable function $g:[0,\infty)\times [0,\infty)\rightarrow [0,\infty)$ which is bounded on compact sets. Set $\amsmathbb{Y}^+:=\cup_{m\in\amsmathbb{N}_0}\amsmathbb{Y}^+_m$ and denote by $\amsmathbb{Y}$ the class of processes that can be written as a difference of two processes from $\amsmathbb{Y}^+$. 
\end{definition}

\begin{proposition}\label{prop:ClassY_bound}
  For any $X\in\amsmathbb{Y}$ and any $\Lambda=(\Lambda_{ij})_{i\neq j}$ satisfying Assumption~\ref{ass:cumulative_tr} it holds that 
  \begin{align*}
    \sup_{0\leq t\leq s\leq T}\sup_{\omega\in\Omega}\amsmathbb{E}_{t,i}^{\Lambda}\big[|X(s)-X(t)|\big]<\infty.
  \end{align*}
\end{proposition}
\begin{proof}
  This follows directly from the fact that $X$ is the difference of two processes from $\amsmathbb{Y}^+$, the two inequalities $(1+x)^m\leq e^{mx}$ and $(x+y)^m\leq 2^{m-1}(x^m+y^m)$ and Proposition~\ref{prop:exp_bound_N}.
\end{proof}

We say that a process $X$ is bounded on finite intervals if the mapping $X:I\times\Omega\rightarrow\amsmathbb{R}$ is bounded for all finite intervals $I\subset[0,\infty)$.

\begin{theorem}\label{th:optional_projection}
  Assume that $C\in\amsmathbb{Y}$ and that $Y$ is bounded. The following two statements are equivalent:
  \begin{enumerate}
    \item[(i)] The process $X(t)=\amsmathbb{E}^{\Lambda}_{t,Z(t)}[L(t,T)]$ has state-wise decomposition $\ti{X}_i(t)=\amsmathbb{E}_{t,i}^{\Lambda}[L(t,T)]$ and satisfies 
    \begin{align*}
      \forall t\in [0,T]: X(t)=\amsmathbb{E}^{\ul{\Lambda}}[L(t,T)|\mathcal{F}_t]\quad\amsmathbb{P}^{\ul{\Lambda}}-\text{a.s.}
    \end{align*}
    \item[(ii)] The process $X$ is càdlàg, of finite variation, bounded on finite intervals and a solution of the BSDE
    \begin{align}\label{eq:BSDE}
      X(\mathrm{d}t)&= - C(\mathrm{d}t) + \sum_{i,j:i\neq j} R_{ij}(t)(N_{ij}(\mathrm{d}t)-I_i(t-)\Lambda_{ij}(\mathrm{d}t)),\quad X(T)=Y,
    \end{align}
    where $R_{ij}(t):=\ti{\Delta} C_{ij}(t)+X_{ij}(t)-X_i(t)$.
  \end{enumerate}
\end{theorem}
\begin{proof}
  (i) $\Rightarrow$ (ii): Let $X(t)=\amsmathbb{E}^{\Lambda}_{t,Z(t)}[L(t,T)]$. By definition $\ti{X}_i(t)=\amsmathbb{E}_{t,i}^{\ul{\Lambda}}[L(t,T)]$ is the state-wise decomposition of $X$ and since $C\in\amsmathbb{Y}$ and $Y$ is bounded, Proposition~\ref{prop:ClassY_bound} yields that $\ti{X}_i(t)$ (and thus $X$) is bounded on finite intervals. As a consequence, Proposition~\ref{prop:optional_projection_path} yields that $X$ is càdlàg and of finite variation and we can apply Proposition~\ref{prop:MG-representation} to conclude that $W(t):=\amsmathbb{E}_{t,Z(t)}^{\Lambda}[L(0,T)]$ satisfies
  \begin{align*}
    W(t)=L(0,T)-\sum_{i,j:i\neq j}\int_{(t,T]} \ti{W}_j(s)-\ti{W}_i(s) (N_{ij}(\mathrm{d}s)-I_i(s-)\ul{\Lambda}_{ij}(\mathrm{d}s)),\quad t\in[0,T]
  \end{align*}
  where $\ti{W}_i(t)=\amsmathbb{E}_{t,i}^{\Lambda}[L(0,T)]$. By applying Proposition~\ref{prop:conditional_expectation}(iii) we obtain
  \begin{align*}
    W(t)=\amsmathbb{E}_{t,Z(t)}^{\Lambda}[Y+C(T)-C(0)]=C(t)-C(0)+\amsmathbb{E}_{t,Z(t)}^{\Lambda}[Y+C(T)-C(t)]=C(t)-C(0)+X(t).
  \end{align*}
  Equating this expression for $W$ with the previous one, we get 
  \begin{align}\label{eq:MG-representation}
    X(t)=Y+C(T)-C(t)-\sum_{i,j:i\neq j}\int_{(t,T]} \ti{W}_j(s)-\ti{W}_i(s) (N_{ij}(\mathrm{d}s)-I_i(s-)\ul{\Lambda}_{ij}(\mathrm{d}s)).
  \end{align}
  By Proposition~\ref{prop:Rep_dyn_J}(i) and Proposition~\ref{prop:conditional_expectation}(ii) we have for all $i\in\mathcal{J}$ that 
  \begin{align*}
    \ti{W}_i(t)=&\amsmathbb{E}_{t,i}^{\Lambda}[L(0,T)]=\int_{(0,t]}\sum_{j\in\mathcal{J}}I_j(s-)C_j(\mathrm{d}t)+\int_{(0,t)}\sum_{j,k:j\neq k}\ti{\Delta} C_{jk}(s) N_{jk}(\mathrm{d}s)\\
    &+\amsmathbb{E}_{t,i}^{\ul{\Lambda}}\Bigg[\sum_{j,k:j\neq k}\ti{\Delta} C_{jk}(t)\Delta N_{jk}(t)\Bigg]+\ti{X}_i(t),\quad t\in[0,T],
  \end{align*}
  and utilising that $\Delta N_{jk}(t)=I_j(t-)I_k(t)$ Proposition~4.4(i) of~\cite{Christiansen&Furrer2024} yields
  \begin{align*}
    \amsmathbb{E}_{t,i}^{\ul{\Lambda}}\Bigg[\sum_{j,k:j\neq k}\ti{\Delta} C_{jk}(t)\Delta N_{jk}(t)\Bigg]=\sum_{j,k:j\neq k}\ti{\Delta} C_{jk}(t)\amsmathbb{E}_{t,i}[I_j(t-)I_k(t)]=\sum_{j:j\neq i}I_j(t-)\ti{\Delta} C_{ji}(t).
  \end{align*}
  These two identities after~(\ref{eq:MG-representation}) yield
  \begin{align*}
    I_i(s-)(\ti{W}_j(s)-\ti{W}_i(s))=I_i(s-)(\ti{\Delta}C_{ij}(s)+\ti{X}_j(s)-\ti{X}_i(s))=I_i(s-)(\ti{\Delta}C_{ij}(s)+X_{ij}(s)-X_i(s)),
  \end{align*}
  for all $s\in(0,T]$, which in conjunction with (\ref{eq:MG-representation}) allows us to conclude that $X$ is a solution of (\ref{eq:BSDE}). 

  (ii) $\Rightarrow$ (i): Using (\ref{eq:BSDE}) we get
  \begin{align*}
    X(t)=L(t,T)-\sum_{i,j:i\neq j}\int_{(t,T]}R_{ij}(s)(N_{ij}(\mathrm{d}s)-I_i(s-)\Lambda_{ij}(\mathrm{d}s)),\quad t\in [0,T].
  \end{align*}
  Since $C\in\amsmathbb{Y}$ and $X$ is bounded on finite intervals, $R_{ij}(t)$ is sufficiently integrable in order for Proposition~4.2 of~\cite{Christiansen&Furrer2024} to hold. Thus taking the expectations with respect to $\amsmathbb{P}^{\Lambda}_{t,i}$ on both sides and applying Proposition~\ref{prop:conditional_expectation}(iii) yields
  \begin{align*}
    \forall i\in\mathcal{J}:\ti{X}_i(t)=\amsmathbb{E}^{\Lambda}_{t,i}[X(t)]=\amsmathbb{E}_{t,i}^{\Lambda}[L(t,T)]\quad\Rightarrow\quad X(t)=\amsmathbb{E}_{t,Z(t)}^{\Lambda}[L(t,T)],\quad t\in [0,T],
  \end{align*}
  which in conjunction with Proposition~\ref{prop:conditional_expectation}(i) yields
  \begin{align*}
    X(t)=\amsmathbb{E}[L(t,T)|\mathcal{F}_t],\quad\amsmathbb{P}^{\Lambda}-\text{a.s.}\quad\forall t\in[0,T].
  \end{align*}
\end{proof}

Often we are also interested in the optional projection of a discounted version of the process $L(t,T)$ given by:
\begin{align*}
  L^*(t,T):=\int_{(t,T]}\frac{\mathcal{E}_{\Phi}(t)}{\mathcal{E}_{\Phi}(s)}C(\mathrm{d}s)+\frac{\mathcal{E}_{\Phi}(t)}{\mathcal{E}_{\Phi}(T)}Y,\quad t\in (0,T],
\end{align*}
where the process $\mathcal{E}_{\Phi}$ is defined as the generalised exponential
\begin{align*}
  \mathcal{E}_{\Phi}(\mathrm{d}t)=\mathcal{E}_{\Phi}(t-)\Phi(\mathrm{d}t),\quad \mathcal{E}_{\Phi}(0)=1\quad\Leftrightarrow\quad\mathcal{E}_{\Phi}(t):=e^{\Phi^c(t)-\Phi^c(0)}\prod_{0<s\leq t}(1+\Delta \Phi(s)).
\end{align*}
We assume that $\Phi$ is a predictable, càdlàg, finite variation process that is bounded on finite intervals and satisfies $\inf_{(t,\omega)\in[0,T]\times\Omega}\Delta\Phi(t,\omega)>-1$. Since this implies that $\mathcal{E}_{\Phi}$ and $\mathcal{E}_{\Phi}^{-1}$ are bounded on finite intervals, Proposition~\ref{prop:optional_projection_path} can be utilised to conclude that the process 
\begin{align*}
  X(t)=\amsmathbb{E}^{\Lambda}_{t,Z(t)}[L^*(t,T)]
\end{align*}
indeed is an optional projection of $L^*(t,T)$ with paths of finite variation. Furthermore, we obtain the following analogue to Theorem~\ref{th:optional_projection}:
\begin{theorem}\label{th:optional_projection_discounted}
  Assume that $C\in\amsmathbb{Y}$, that $Y$ is bounded, and that $\Phi$ is a predictable, càdlàg finite variation process bounded on finite intervals such that $\inf_{(t,\omega)\in[0,T]\times\Omega}\Delta \Phi(t,\omega)>-1$. The following two statements are equivalent:
  \begin{enumerate}
    \item[(i)] The process $X(t)=\amsmathbb{E}_{t,Z(t)}^{\Lambda}[L(t,T)]$ has state-wise decomposition $\ti{X}_i(t)=\amsmathbb{E}^{\Lambda}_{t,i}[L^*(t,T)]$ and satisfies
    \begin{align*}
      \forall t\in [0,T]: X(t)=\amsmathbb{E}^{\ul{\Lambda}}[L^*(t,T)|\mathcal{F}_t]\quad\amsmathbb{P}^{\ul{\Lambda}}-\text{a.s.}
    \end{align*}
    \item[(ii)] The process $X$ is càdlàg, of finite variation, bounded on finite intervals and a solution of the BSDE
    \begin{align}\label{eq:BSDE_discounted}
      X(\mathrm{d}t)&=X(t-)\Phi(\mathrm{d}t) - C(\mathrm{d}t) + \sum_{i,j:i\neq j} R_{ij}(t)(N_{ij}(\mathrm{d}t)-I_i(t-)\Lambda_{ij}(\mathrm{d}t)),\quad X(T)=Y,
    \end{align}
     where $R_{ij}(t):=\ti{\Delta} C_{ij}(t)+X_{ij}(t)-X_i(t)$.
  \end{enumerate}
\end{theorem}
\begin{proof}
  Let $Y^*:=\frac{Y}{\mathcal{E}_{\Phi}(t)}$ and $C^*(t):=C(0)+\int_{(0,T]}\frac{1}{\mathcal{E}_{\Phi}(s)}C(\mathrm{d}s)$. Due to our assumptions on $\Phi$, we have that $\frac{1}{\mathcal{E}_{\Phi}(t)}$ is strictly positive and bounded on finite intervals, which implies that $Y^*$ and $C^*$ satisfy the assumptions of Theorem~\ref{th:optional_projection}.

  (ii) $\Rightarrow$ (i): Let $W(t):=\frac{X(t)}{\mathcal{E}_{\Phi}(t)}$. By Ito's formula, we get
  \begin{align*}
    W(\mathrm{d}t)=&-\frac{X(t-)}{\mathcal{E}_{\Phi}(t)}\Phi(\mathrm{d}t)+\frac{1}{\mathcal{E}_{\Phi}(t)}X(\mathrm{d}t)\\
    =&-C^*(\mathrm{d}t) + \sum_{i,j:i\neq j} \frac{1}{\mathcal{E}_{\Phi}(t)}R_{ij}(t)(N_{ij}(\mathrm{d}t)-I_i(t-)\Lambda_{ij}(\mathrm{d}t)),\quad t\in (0,T].
  \end{align*}
  As $\mathcal{E}_{\Phi}$ is predictable and $W$ is optional, Proposition~\ref{prop:Rep_adapted_process_J} yields that $W_i(t)=\frac{X_i(t)}{\mathcal{E}_{\Phi,i}(t)}$ and $W_{ij}(t)=\frac{X_{ij}(t)}{\mathcal{E}_{\Phi,i}(t)}$ and thus we obtain
  \begin{align*}
    I_i(t-)\frac{1}{\mathcal{E}_{\phi}(t)}R_{ij}(t)=I_i(t-)\bigg(\ti{\Delta}C^*_{ij}(t)+W_{ij}(t)-W_i(t)\bigg)=I_i(t-)R_{ij}^W(t),
  \end{align*}
  and we can conclude that $W$ satisfies the BSDE
  \begin{align*}
    W(\mathrm{d}t)=-C^*(\mathrm{d}t)-\sum_{i,j:i\neq j}R_{ij}^W(t)(N_{ij}(\mathrm{d}t)-I_i(t-)\Lambda_{ij}(\mathrm{d}t)),\quad W(T)=Y^*.
  \end{align*}
  Theorem~\ref{th:optional_projection} in conjunction with the relationship $X(t)=\mathcal{E}_{\Phi}(t)W(t)$ and Proposition~\ref{prop:conditional_expectation}(iii) now yields the desired result.

  (i) $\Rightarrow$ (ii): By Proposition~\ref{prop:conditional_expectation}(iii) we have that 
  \begin{align*}
    X(t)=\amsmathbb{E}_{t,Z(t)}^{\Lambda}[L(t,T)]=\mathcal{E}_{\Phi}(t)\amsmathbb{E}_{t,Z(t)}[Y^*+C^*(T)-C^*(t)].
  \end{align*}
  By Theorem~\ref{th:optional_projection} we obtain the above BSDE for $W(t):=\amsmathbb{E}_{t,Z(t)}[Y^*+C^*(T)-C^*(t)]$. Since $X(t)=\mathcal{E}_{\Phi}(t)W(t)$, integration by parts yields the desired result.
\end{proof}

Theorem~\ref{th:optional_projection} and Theorem~\ref{th:optional_projection_discounted} provide a one-to-one correspondence between a BSDE and an optional projection. This turns out to be very useful for two reasons: The first reason is that one can establish existence and uniqueness of the optional projection by establishing existence and uniqueness of the BSDE, which is key when considering the case where the process $C$ or the cumulative transition rates $\Lambda$ depend on the optional projection itself. The second reason is that by using the jump-count representation, it is possible to show that the BSDE is equivalent to an infinite-dimensional system of Lebesgue-Stieltjes integral equations. This allows, in principle, to calculate the optional projection numerically and reduces the question of existence and uniqueness of the BSDE to the question of existence and uniqueness of the infinite-dimensional system of Lebesgue-Stieltjes integral equations.

By Proposition~\ref{prop:Rep_adapted_process_N} the processes $C$ and $(\Lambda_{ij})_{i\neq j}$ each have a unique jump-count representation $(C^n)_{n\in\amsmathbb{N}_0}$ and $(\Lambda_{ij}^n)_{n\in\amsmathbb{N}_0}$. Before we state the system of integral equations properly, we have to define counterfactual mappings in the spirit of $\omega_i^t$, which have $\Omega^n$ as their domain. For $i\in\mathcal{J}$ and $n\in\amsmathbb{N}_0$ define the mapping 
\begin{align*}
  \omega^n_{i}:\Omega^n\rightarrow\bar{\Omega}^{n},\quad \omega^n_{i}(\omega^{n}):=
  \begin{cases}
    ((t_{\ell},z_{\ell})_{\ell\leq n-1},(t_n,i)) & \text{for } z_{n-1}\neq i\\
    ((t_{\ell},z_{\ell})_{\ell\leq n-1},(\infty,\nabla)) & \text{for } z_{n-1}=i
  \end{cases},
\end{align*}
which for any given sequence of $n$ jumps sets $z_n=i$ whenever $z_{n-1}\neq i$, while it removes the $n$'th jump if $z_{n-1}=i$. Thus for any $\omega\in\Omega$ with $\omega=(\omega^n,\ldots)$ for some $\omega^n\in\Omega^n$ it holds that $\omega^n_i=\xi^n(\omega^{t_n}_i)$. Furthermore for $i,j\in\mathcal{J}$ with $i\neq j$ and $n\in\amsmathbb{N}_0$ we define the mapping
\begin{align*}
  \omega^{n+1}_{t,ij}:[0,\infty)\times\Omega^n\rightarrow\bar{\Omega}^{n+1},\quad \omega^{n+1}_{t,ij}(\omega^{n}):=
  \begin{cases}
    ((t_{\ell},z_{\ell})_{\ell\leq n-1},(t_n,i),(t,j)) & \text{for } z_{n-1}\neq i,\,t>t_{n}\\
    ((t_{\ell},z_{\ell})_{\ell\leq n-1},(t,j),(\infty,\nabla)) & \text{for } z_{n-1}=i,\,t>t_{n}
  \end{cases},
\end{align*}
which for any given sequence of $n$ jumps sets $z_n=i$ whenever $z_{n-1}\neq i$ and then adds a jump to $j$ at time $t$, while it removes the $n$'th jump if $z_{n-1}=i$ and then adds a jump to $j$ at time $t$. Thus for any $\omega\in\Omega$ with $\omega=(\omega^n,\ldots)$ for some $\omega^n\in\Omega^n$ it holds that $\omega^{n+1}_{t,ij}=\xi^{n+1}(\omega^{t_n,t}_{i,j})$.

\begin{proposition}\label{prop:BSDE_equiv_ODE}
  Let $X$ be an adapted càdlàg process of finite variation with jump-count representation $(f^n)_{n\in\amsmathbb{N}_0}$. Then $X$ is the unique solution of (\ref{eq:worst_case_BSDE}) if and only if $(f^n)_{n\in\amsmathbb{N}_0}$ is the unique solution of the system 
  \begin{align}\label{eq:system_of_equations}
      \begin{split}
      f^n(\mathrm{d}t,\omega^n)=&-C^n(\mathrm{d}t,\omega^n)-\sum_{j:j\neq z_n}R_{z_nj}^{n+1}(t,\omega^n)\Lambda^n_{z_n j}(\mathrm{d}t,\omega^n),\quad t\in(0,T],\\
      f^n(T,\omega^n)=&Y^n(\omega^n),\quad \omega^n\in\Omega^n,\,n\in\amsmathbb{N}_0     
    \end{split}
    \end{align}
    and where $R_{ij}^{n+1}(t,\omega^n):=\ti{\Delta}C^{n+1}_{ij}(t)+f^{n+1}(t,\omega^{n+1}_{t,ij})-f^n(t,\omega^n_i)$.
\end{proposition}
\begin{proof}
  Since $X$ is càdlàg and of finite variation Proposition~\ref{prop:Rep_dyn_N}(i) yields
  \begin{align*}
    X(\mathrm{d}t)=\sum_{n\in\amsmathbb{N}_0}I^n(t-)\bigg(f^n(\mathrm{d}t,\xi^n)+\sum_{j:j\neq \zeta_n} f^{n+1}(t,\xi^{n}(\omega_{\zeta_n,j}^{\tau_n,t}))-f^n(t,\xi^n)N_{\zeta_n j}(\mathrm{d}t)\bigg).
  \end{align*}
  Using Proposition~\ref{prop:Rep_adapted_process_N}, Proposition~\ref{prop:Rep_dyn_N} and (\ref{eq:CF_n}), the BDS (\ref{eq:worst_case_BSDE}) can equivalently be written as 
  \begin{align*}
    X(\mathrm{d}t)=\sum_{n\in\amsmathbb{N}_0}I^n(t-)\bigg(&-C^n(\mathrm{d}t,\xi^n)-\sum_{j:j\neq \zeta_n}\ti{\Delta} C^{n}_{\zeta_n j}(t,\xi^n)N_{ij}(\mathrm{d}t)\\
    &+\sum_{j:j\neq \zeta_n} R_{\zeta_n j}^{n+1}(t,\xi^n)(N_{\zeta_n j}(\mathrm{d}t)-I_{\zeta_n}(t-)\Lambda_{\zeta_n j}^n(\mathrm{d}t,\xi^n))\bigg).
  \end{align*}
  Equating the first and second equation for $X(\mathrm{d}t)$ and rearranging yields
  \begin{align*}
    0&=I^n(t-)\bigg(f^n(\mathrm{d}t,\xi^n)+C^n(\mathrm{d}t,\xi^n)+\sum_{j:j\neq \zeta_n} R_{\zeta_n j}^{n+1}(t,\xi^n)\Lambda_{\zeta_n j}^n(\mathrm{d}t,\xi^n)\bigg).
  \end{align*}
  If $X$ is an adapted, càdlàg finite variation solution of (\ref{eq:worst_case_BSDE}) with jump-count decomposition $(f^n)_{n\in\amsmathbb{N}_0}$, then by Remark~\ref{rem:path_properties} each $f^n$ is càdlàg and of finite variation and both equations for $X(\mathrm{d}t)$ are valid implying the above third equation. Since this equation holds for all $\omega=(\omega^n,(\infty,\nabla),\ldots)$ for $\omega^n\in\Omega^n$, it holds for all $t> t_n$. Due to the conventions~(\ref{eq:fn_uniqueness_condition_optional}) and~(\ref{eq:fn_uniqueness_condition_predictable}) used to define the jump-count representation, it holds that $f^n(\mathrm{d}t)=C^n(\mathrm{d}t)=\Lambda_{z_n j}^n(\mathrm{d}t)=0$ for $t\in(0,t_n]$. Thus $(f^n)_{n\in\amsmathbb{N}_0}$ is a solution of the system (\ref{eq:system_of_equations}) for all $t\in(0,T]$. If on the other hand $(f^n)_{i\in\amsmathbb{N}_0}$ is a solution of the system (\ref{eq:system_of_equations}), each $f^n$ is càdlàg and of finite variation which implies the same for $X(t)=\sum_{n\in\amsmathbb{N}_0}I^n(t)f^n(t,\xi^n)$. Thus, the first equation of $X(\mathrm{d}t)$ is valid and inserting the above third equation and rearranging yields the second, which proves that $X$ is a solution of (\ref{eq:worst_case_BSDE}).
\end{proof}
Similar to the undiscounted case, the existence and uniqueness of the BSDE (\ref{eq:BSDE_discounted}) is equivalent to the existence and uniqueness of a system of Lebesgue-Stieltjes integral equations, which is the statement of the following analogue of Proposition~\ref{prop:BSDE_equiv_ODE}.

\begin{proposition}\label{prop:BSDE_equiv_ODE_discounted}
  Let $X$ be an adapted càdlàg process of finite variation with jump-count representation $(f^n)_{n\in\amsmathbb{N}_0}$. Then $X$ is the unique solution of (\ref{eq:worst_case_BSDE_discounted}) if and only if $(f^n)_{n\in\amsmathbb{N}_0}$ is the unique solution of the system 
  \begin{align}\label{eq:system_of_equations}
      \begin{split}
      f^n(\mathrm{d}t,\omega^n)=&f^n(t-,\omega^n)\Phi^n(\mathrm{d}t,\omega^n)-C^n(\mathrm{d}t,\omega^n)-\sum_{j:j\neq z_n}R_{z_nj}^{n+1}(t,\omega^n)\Lambda^n_{z_n j}(\mathrm{d}t,\omega^n),\quad t\in(0,T],\\
      f^n(T,\omega^n)=&Y^n(\omega^n),\quad \omega^n\in\Omega^n,\,n\in\amsmathbb{N}_0     
    \end{split}
    \end{align}
    and where $R_{ij}^{n+1}(t,\omega^n):=\ti{\Delta}C^{n+1}_{ij}(t)+f^{n+1}(t,\omega^{n+1}_{t,ij})-f^n(t,\omega^n_i)$.
\end{proposition}

\subsection{The worst-case BSDE}\label{subsec:3.2} We now consider a case where both the process $C$ and the cumulative transition rates may depend on the optional projection and provide an existence and uniqueness result for this type of circular problem. For this, let $(U_{ij})_{i\neq j}$ and $(L_{ij})_{i\neq j}$ be cumulative transition rates which in addition to Assumption~\ref{ass:cumulative_tr} satisfy 
\begin{align*}
  L_{ij}(s,\omega)-L_{ij}(t,\omega)\leq U_{ij}(s,\omega)-U_{ij}(t,\omega),\quad \forall 0\leq t \leq s,\forall\omega\in\Omega,\quad \forall i,j\in\mathcal{J}:i\neq j
\end{align*}
and consider the BSDE given by
\begin{align}\label{eq:worst_case_BSDE}
  X(\mathrm{d}t)= - C(\mathrm{d}t) + \sum_{i,j:i\neq j} R_{ij}(t)(N_{ij}(\mathrm{d}t)-I_i(t-)\ul{\Lambda}_{ij}(\mathrm{d}t)),\quad X(T)=Y,
\end{align}
where $R_{ij}(t):=\Delta C_{ij}(t)+X_{ij}(t)-X_i(t)$ and $Y$ is $\mathcal{F}_T$-measurable. We assume that the cumulative transition rates $\ul{\Lambda}=(\ul{\Lambda}_{ij})_{i\neq j}$ have one of the following two forms
\begin{align*}
  \ul{\Lambda}_{ij}(\mathrm{d}t)&=\mathds{1}_{(R_{ij}(t)\geq 0)}U_{ij}(\mathrm{d}t)+\mathds{1}_{(R_{ij}(t)<0)}L_{ij}(\mathrm{d}t),\quad\forall i,j\in\mathcal{J}:i\neq j,\\
  \ul{\Lambda}_{ij}(\mathrm{d}t)&=\mathds{1}_{(R_{ij}(t)< 0)}U_{ij}(\mathrm{d}t)+\mathds{1}_{(R_{ij}(t)\geq 0)}L_{ij}(\mathrm{d}t),\quad\forall i,j\in\mathcal{J}:i\neq j.
\end{align*}
The $(\ul{\Lambda}_{ij})_{i\neq j}$ are cumulative transition rates, as the processes $(L_{ij})_{i\neq j}$ and $(U_{ij})_{i\neq j}$ being cumulative transition rates ensures that $(\ul{\Lambda}_{ij})_{i\neq j}$ satisfies Assumption~\ref{ass:cumulative_tr}(ii)+(iii), and the fact that the processes ($R_{ij})_{i\neq j}$ are $i$-predictable ensures that the processes $\ul{\Lambda}_{ij}$ are $i$-predictable (Assumption~\ref{ass:cumulative_tr}(i)). The process $C$ is an optional process with càdlàg paths of finite variation given by
\begin{align*}
  C(\mathrm{d}t)&=\sum_{i\in\mathcal{J}}I_i(t-)\bigg(C_i(\mathrm{d}t)+\sum_{j:j\neq i}\Delta C_{ij}(t)N_{ij}(\mathrm{d}t)\bigg),
\end{align*} 
for $i$-predictable $(C_i,(\Delta C_{ij})_{j:j\neq i})_{i\in\mathcal{J}}$, which are assumed to be of the form 
\begin{align*}
  C_{i}(\mathrm{d}t,\omega)&=b_i(t,\omega,X_i(t,\omega))\gamma_i(\mathrm{d}t,\omega),\\
  \Delta C_{ij}(\mathrm{d}t,\omega)&=b_{ij}(t,\omega,X_i(t,\omega),X_{ij}(t,\omega)).
\end{align*}
Clearly if there exists a unique solution $X$ of the BSDE~(\ref{eq:worst_case_BSDE}) that is càdlàg, of finite variation and bounded on finite intervals, we can by Theorem~\ref{th:optional_projection} conclude that $X(t)=\amsmathbb{E}_{t,Z(t)}^{\ul{\Lambda}}[Y+C(T)-C(t)]$, but since both the process $C$ and the cumulative transition rates $\ul{\Lambda}$ determining the measure $\amsmathbb{P}^{\ul{\Lambda}}$ depend on $X$, it is not clear, whether such a solution exists. In order to prove existence and uniqueness of the BSDE (\ref{eq:worst_case_BSDE}), we assume that the processes $(\gamma_i)_{i\in\mathcal{J}}$, $(b_i)_{i\in\mathcal{J}}$ and $(b_{ij})_{i\neq j}$ satisfy the following regularity conditions.
\begin{assumption}\label{ass:regularity_BSDE}
  Assume that
  \begin{enumerate}
    \item[(i)] The processes $(\gamma_i)_{i\in\mathcal{J}}$ are $i$-predictable càdlàg processes of finite variation, and there exists a non-decreasing càdlàg function $\bar{\gamma}:[0,\infty)\rightarrow[0,\infty)$ bounded on compacts such that for all $\omega\in\Omega$:
    \begin{align*}
      |\gamma_i(t,\omega)|\leq \bar{\gamma}(t),\,t\geq 0\quad\text{and}\quad |\gamma_i(t,\omega)-\gamma_i(s,\omega)|\leq \bar{\gamma}(t)-\bar{\gamma}(s),\,0\leq s\leq t
    \end{align*}
    \item[(ii)] The functions $(b_i,(b_{ij})_{j:j\neq i})_{i\in\mathcal{J}}$ with $b_i:[0,\infty)\times\Omega\times\amsmathbb{R}\rightarrow\amsmathbb{R}$ and $b_{ij}:[0,\infty)\times\Omega\times\amsmathbb{R}\times\amsmathbb{R}\rightarrow \amsmathbb{R}$ are measurable, and $(t,\omega)\mapsto b_i(t,\omega,x)$ and $(t,\omega)\mapsto b_{ij}(t,\omega,x,y)$ are $i$-predictable for all fixed $x,y\in\amsmathbb{R}$.
    \item[(iii)] There exists a positive function $\beta:[0,\infty)\rightarrow (0,\infty)$ which is bounded on compacts such that
    \begin{align*}
      |b_i(t,\omega,0)|&\leq \beta(t),\quad \forall\,i\in\mathcal{J}\\
      |b_{ij}(t,\omega,0,0)|&\leq\beta(t),\quad \forall\,i,j\in\mathcal{J}:i\neq j
    \end{align*}
    for all $(t,\omega)\in[0,\infty)\times\Omega$.
    \item[(iv)] There exists a positive function $\beta:[0,\infty)\rightarrow (0,\infty)$ which is bounded on compacts such that 
    \begin{align*}
      |b_i(t,\omega,x_1)-b_i(t,\omega,x_1)|&\leq \beta(t)|x_1-x_2|,\quad \forall i\in\mathcal{J}.\\
      |b_{ij}(t,\omega,x_1,y_1)-b_{ij}(t,\omega,x_2,y_2)|&\leq \beta(t)(|x_1-x_2|+|y_1-y_2|),\quad \forall i,j\in\mathcal{J}:i\neq j.
    \end{align*}
    for all $(t,\omega)\in[0,\infty)\times\Omega$ and $(x_1,x_2),(y_1,y_2)\in\amsmathbb{R}^2$.
  \end{enumerate}
\end{assumption}

Assuming those regularity conditions brings us to the first of two main results of this section.
\begin{theorem}\label{th:BSDE_existence}
  Assume that Assumption~\ref{ass:regularity_BSDE} is satisfied and that $Y$ is bounded. Then there exists a unique solution $X$ of the BDS (\ref{eq:worst_case_BSDE}) that is càdlàg, of finite variation and bounded on finite intervals.
\end{theorem}
\begin{proof}
  See Subsection~\ref{subsection:Proof_existence}.
\end{proof}

Returning to the discounted case, we can consider the BDS
\begin{align}\label{eq:worst_case_BSDE_discounted}
    X(\mathrm{d}t)&=X(t-)\Phi(\mathrm{d}t) - C(\mathrm{d}t) + \sum_{i,j:i\neq j} R_{ij}(t)(N_{ij}(\mathrm{d}t)-I_i(t-)\ul{\Lambda}_{ij}(\mathrm{d}t)),\quad X(T)=Y,
\end{align}
where $R_{ij}(t):=\Delta C_{ij}(t)+X_{ij}(t)-X_i(t)$ and $Y$ is $\mathcal{F}_T$-measurable. The process $\Phi$ is predictable, cádlág, and of finite variation and the cumulative transition rates $\ul{\Lambda}$ and the process $C$ are assumed to be of the same form as in the undiscounted case. Assuming that Assumption~\ref{ass:regularity_BSDE} is satisfied, we obtain the second main result of this section and analogue of Theorem~\ref{th:BSDE_existence}.

\begin{theorem}\label{th:BSDE_existence_discounted}
  Assume that Assumption~\ref{ass:regularity_BSDE} is satisfied, that $Y$ is bounded and that $\Phi$ is predictable, càdlàg, of finite variation and bounded on finite intervals with $\inf_{(t,\omega)\in[0,T]\times\Omega}\Phi(t,\omega)>-1$. Then there exists a unique solution $X$ of the BDS (\ref{eq:worst_case_BSDE_discounted}) that is càdlàg, of finite variation and bounded on finite intervals.
\end{theorem}
\begin{proof}
  We start by considering the process $W$ solving the BDS 
  \begin{align*}
    W(\mathrm{d}t)&=-C^*(\mathrm{d}t)+\sum_{i,j:i\neq j}R_{ij}^W(t)(N_{ij}(\mathrm{d}t)-I_i(t-)\ul{\Lambda}_{ij}^W(\mathrm{d}t)),\quad W(T)=Y^*,\\
    R_{ij}^W(t)&:=\ti{\Delta}\bar{C}_{ij}(t)+W_{ij}(t)-W_i(t),
  \end{align*}
  and where $Y^*:=\frac{Y}{\mathcal{E}_{\Phi}(T)}$. The cumulative transition rates $\ul{\Lambda}_{ij}^W$ all have one of the following forms:
  \begin{align*}
    \ul{\Lambda}_{ij}^W(\mathrm{d}t)&=\mathds{1}_{(R_{ij}^W\geq 0)}U_{ij}(\mathrm{d}t)+\mathds{1}_{(R_{ij}^W<0)}L_{ij}(\mathrm{d}t)\\
    \ul{\Lambda}_{ij}^W(\mathrm{d}t)&=\mathds{1}_{(R_{ij}^W<0)}U_{ij}(\mathrm{d}t)+\mathds{1}_{(R_{ij}^W\geq 0)}L_{ij}(\mathrm{d}t),
  \end{align*}
  while the process $C^*$ is given by 
  \begin{align*}
    C^*(\mathrm{d}t)=\sum_{i\in\mathcal{J}}I_i(t-)\bigg(C^*_i(\mathrm{d}t)+\sum_{j:j\neq i}\ti{\Delta}C^*_{ij}(t)N_{ij}(\mathrm{d}t)\bigg),
  \end{align*} 
  where $C^*_i$ and $\ti{\Delta}C^*_{ij}$ have the form
  \begin{align*}
    C^*_i(\mathrm{d}t,\omega)&=b^*_{i}(t,\omega,W_i(t,\omega))\gamma_i(\mathrm{d}t,\omega)=\frac{1}{\mathcal{E}_{\Phi,i}(t)}b_i(t,\omega,\mathcal{E}_{\Phi,i}(t,\omega)W_i(t,\omega))\gamma_i(\mathrm{d}t,\omega),\\
    \ti{\Delta}C^*_{ij}(\mathrm{d}t,\omega)&=b^*_{ij}(t,\omega,W_i(t,\omega),W_{ij}(t,\omega))=\frac{1}{\mathcal{E}_{\Phi,i}(t)}b_{ij}(t,\omega,\mathcal{E}_{\Phi,i}(t,\omega)W_i(t,\omega),\mathcal{E}_{\Phi,i}(t,\omega)W_{ij}(t,\omega)).
  \end{align*}
  Since $\frac{1}{\mathcal{E}_{\Phi}(t)}$ is bounded on finite intervals, it follows that $(\bar{b}_i,(\bar{b}_{ij})_{j:j\neq i})_{i\in\mathcal{J}}$ satisfy Assumption~\ref{ass:regularity_BSDE} and Theorem~\ref{th:BSDE_existence} yields existence and uniqueness of $W$. 

  Set $X=\mathcal{E}_{\Phi}(t)W(t)$. Since $X_i(t)=\mathcal{E}_{\Phi,i}(t)W_i(t)$ and $X_{ij}(t)=\mathcal{E}_{\Phi,i}(t)W_{ij}(t)$ we get 
  \begin{align*}
    R_{ij}(t)=\ti{\Delta}C_{ij}(t)+X_{ij}(t)-X_i(t)=\mathcal{E}_{\Phi,i}(t)(\ti{\Delta}C^*_{ij}(t)+W_{ij}(t)-W_i(t))=\mathcal{E}_{\Phi,i}(t)R_{ij}^W(t)
  \end{align*}
  and since $\mathcal{E}_{\Phi}(t)$ is strictly positive, we have that $\text{sign}(R_{ij}^W(t))=\text{sign}(R_{ij}(t))$, thus allowing us to conclude that $\ul{\Lambda}_{ij}^W=\ul{\Lambda}_{ij}$. Using this, and an integration by parts argument as in the proof of Theorem~\ref{th:optional_projection_discounted} it follows that $X$ is a solution of (\ref{eq:worst_case_BSDE_discounted}) if and only if $X(t)=\mathcal{E}_{\Phi}(t)W(t)$, where $W$ is as just described. As $W$ and $\mathcal{E}_{\Phi}$ both exist and are unique, so is $X$, and since $W$ and $\mathcal{E}_{\Phi}$ both are càdlàg, of finite variation and bounded on finite intervals, so is $X$.
\end{proof}

\subsection{Proof of Theorem~\ref{th:BSDE_existence}}\label{subsection:Proof_existence}
The key to the proof of Theorem~\ref{th:BSDE_existence} is to use the, by Proposition~\ref{prop:BSDE_equiv_ODE} given, equivalence between the existence and uniqueness of a solution $X$ of (\ref{eq:worst_case_BSDE}) and existence and uniqueness of the following system of Lebesgue-Stieltjes integral equations satisfied by the jump-count representation $(f^n)_{n\in\amsmathbb{N}_0}$ of $X$:
\begin{align}\label{eq:system_of_equations}
  \begin{split}
    f^n(\mathrm{d}t,\omega^n)&=-b_{z_n}^n(t,\omega^n,f^n(t,\omega^n))\gamma_{z_n}^n(\mathrm{d}t,\omega^n)+\sum_{j:j\neq z_n}R_{z_nj}^{n+1}(t,\omega^n)\ul{\Lambda}^n_{z_n j}(\mathrm{d}t,\omega^n),\quad t\in(0,T],\\
    f^n(T,\omega^n)&=Y^n(\omega^n),\quad \omega^n\in\Omega^n,\,n\in\amsmathbb{N}_0,
  \end{split}
\end{align}
where $R_{ij}^{n+1}(t,\omega^n)=b^{n}_{ij}(t,\omega^n,f^n(t,\omega^n_i),f^{n+1}(t,\omega^{n+1}_{t,ij}))+f^{n+1}(t,\omega^{n+1}_{t,ij})-f^n(t,\omega^n_i)$ and where all $\ul{\Lambda}_{ij}^n$ are given by one of the following two variants:
\begin{align*}
  \ul{\Lambda}_{ij}^n(\mathrm{d}t,\omega^n)&=\mathds{1}_{(R_{ij}^{n+1}(t,\omega^n)\geq 0)}U_{ij}^n(\mathrm{d}t,\omega^n)+\mathds{1}_{(R_{ij}^{n+1}(t,\omega^n)< 0)}L_{ij}^n(\mathrm{d}t,\omega^n),\quad \forall i,j\in\mathcal{J}:i\neq j,\,\forall n\in\amsmathbb{N}_0\\
  \ul{\Lambda}_{ij}^n(\mathrm{d}t,\omega^n)&=\mathds{1}_{(R_{ij}^{n+1}(t,\omega^n)< 0)}U_{ij}^n(\mathrm{d}t,\omega^n)+\mathds{1}_{(R_{ij}^{n+1}(t,\omega^n)\geq 0)}L_{ij}^n(\mathrm{d}t,\omega^n),\quad \forall i,j\in\mathcal{J}:i\neq j,\,\forall n\in\amsmathbb{N}_0.
\end{align*}
Here $(L_{ij}^n)_{n\in\amsmathbb{N}_0}$, $(U_{ij}^n)_{n\in\amsmathbb{N}_0}$, $(\gamma_{i}^n)_{n\in\amsmathbb{N}_0}$, $(b_{i}^n)_{n\in\amsmathbb{N}_0}$ and $(b_{ij}^n)_{n\in\amsmathbb{N}_0}$ denote the jump-count representation of the processes $(L_{ij})_{i\neq j}$, $(U_{ij})_{i\neq j}$, $(\gamma_i)_{i\in\mathcal{J}}$, $(b_i)_{i\in\mathcal{J}}$ and $(b_{ij})_{i\neq j}$ and by Proposition~\ref{eq:Rep_N_uniqueness} and Remark~\ref{rem:path_properties} conditions (i), (iii) and (iv) of Assumption~\ref{ass:regularity_BSDE} remain valid for the functions of the relevant jump-count decomposition. We could thus equivalently state Assumption~\ref{ass:regularity_BSDE} in terms of the corresponding jump-count decomposition by adding a superscript $n$ and replacing $\omega\in\Omega$ with $\omega^n\in\bar{\Omega}^n$ in the relevant places.

For a function $f:[0,T]\rightarrow\amsmathbb{R}$ let $|f|(t)$ denote the variation of $f$ on the interval $[0,t]$ and set $\text{BV}_{[0,T]}:\{f:[0,T]\rightarrow\amsmathbb{R}:|f|(T)<\infty,\,f\text{ is càdlàg}\}$. Let $\|\cdot\|_{BV}$ denote the finite variation norm given by $\|f\|_{BV}=|f(T)|+|f|(T)$ and note that $(\text{BV}_{[0,T]},\|\cdot\|_{BV})$ is a Banach space, see also Appendix~\ref{appendix:FV}. In the proofs, we will also use a family of equivalent norms
\begin{align*}
  \|f\|_{BV}^{c,\nu}:=|f(T)|+\int_{(0,T]}e^{-c(\nu(T)-\nu(t-))}|f|(\mathrm{d}t),
\end{align*}
where $\nu:[0,T]\rightarrow [0,\infty)$ is any non-decreasing càdlàg function and $c>0$ is an arbitrary constant, see also Appendix~\ref{appendix:FV}.

The proof that (\ref{eq:system_of_equations}) has a unique solution is now done in a sequence of lemmas. First we prove in Lemma~\ref{lem:existence_1} that for each $n\in\amsmathbb{N}_0$ the function $f^n$ exists, given that $f^{n+1}$ is known. In order to then obtain existence and uniqueness of the infinite dimensional system (\ref{eq:system_of_equations}), we first construct a sequence $((f^{n,M})_{n\in\amsmathbb{N}_0})_{M\in\amsmathbb{N}}$ of finite dimensional systems and prove in Lemma~\ref{lem:existence_M} that each member of this sequence has a unique solution. This is followed by Lemma~\ref{lem:Cauchy}, showing that for each $n\in\amsmathbb{N}_0$ the sequence $(f^{n,M})_{M\in\amsmathbb{N}}$ has a limit $f^n$, and Lemma~\ref{lem:limit}, which shows that $(f^n)_{n\in\amsmathbb{N}_0}$ is a solution of the system~(\ref{eq:system_of_equations}). Finally, we conclude with Lemma~\ref{lem:uniqueness}, which shows that (\ref{eq:worst_case_BSDE}) has a unique solution.

\textbf{Notation:} Throughout the proofs in this subsection, we will often suppress the dependence on $\omega^n$ in order to ease the notational burden. In that case we will use the shorthand notation $f^n(t)$ for $f^n(t,\omega^n)$ and $f^{n+1}_{t,ij}(t)$ for $f^{n+1}(t,\omega^{n+1}_{t,ij}(\omega^n))$.

\begin{lemma}\label{lem:existence_1}
  Fix $n\in\amsmathbb{N}_0$ and assume that $f^{n+1}:[0,T]\times\Omega^{n+1}\rightarrow\amsmathbb{R}$ is a known càdlàg function of finite variation such that
  \begin{align*}
    \forall i\neq j:\quad \sup_{\omega^{n+1}\in \Omega^{n+1}}\|f^{n+1}\|_{BV}(\omega^{n+1})<\infty
  \end{align*}
  Then (\ref{eq:system_of_equations}) considered for fixed $n$ has a unique solution.
\end{lemma}
\begin{proof}
  Let $\omega^n\in\Omega^n$ be given and fixed and set $K_{\beta}:=\sup_{t\in[0,T]}\beta(t)$. The goal is to apply the Banach Fixed Point Theorem for the Banach space $\text{BV}_{[0,T]}$ with the norm $\|\cdot\|_{BV}^{c,\nu}$, where $c>0$ and 
  \begin{align*}
    \nu(t,\omega^n):=(1+K_{\beta})\bigg(|\gamma_{z_n}^n|(t,\omega^n)+\sum_{j:j\neq z_n}U_{z_nj}(t,\omega^n)\bigg).
  \end{align*}
  Now fix an arbitrary $\omega^n\in\Omega^n$. During the rest of the proof, we will suppress the dependence on $\omega^n$ for ease of notation. Let $f\in\text{BV}_{[0,T]}$ and define the mapping $F:BV_{[0,T]}\rightarrow BV_{[0,T]}$ by
  \begin{align*}
    F(f)(t):=Y^n+\int_{(t,T]}b_{z_n}^n(s,f(s))\gamma_{z_n}^n(\mathrm{d}s)-\sum_{j:j\neq z_n}\int_{(t,T]}R_{t,z_n j}^{n+1}(t)\ul{\Lambda}_{z_n j}^{n}(\mathrm{d}t),
  \end{align*}
  where $R_{t,z_nj}^{n+1}(t):=b^{n}_{z_n j}(t,f(t),f^{n+1}_{t,z_nj}(t))+f^{n+1}_{t,z_n j}(t)-f(t)$. First, we have to check that $T(f)$ actually is of finite variation, given that $f$ is. Using that $x^{+}:=\mathds{1}_{(x>0)}x$ and $x^-:=\mathds{1}_{(x<0)}x$ can be written as
  \begin{align*}
    x^{+}=\frac{x+|x|}{2} \quad \text{and} \quad x^{-}=\frac{x-|x|}{2},
  \end{align*}
  we can rewrite $R_{t,ij}^{n+1}\ul{\Lambda}_{ij}^{n}(\mathrm{d}t)$ for any $i\neq j$ as
  \begin{align*}
    R_{t,ij}^{n+1}(t)\ul{\Lambda}_{ij}^{n}(\mathrm{d}t)&=\frac{R_{t,ij}^{n+1}(t)+| R_{t,ij}^{n+1}(t)|}{2}U_{ij}^n(\mathrm{d}t)+\frac{R_{t,ij}^{n+1}(t)-| R_{t,ij}^{n+1}(t)|}{2}L_{ij}^n(\mathrm{d}t)\\
    &=R_{t,ij}^{n+1}(t)\bigg(\frac{U_{ij}^n+L_{ij}^n}{2}\bigg)(\mathrm{d}t)+| R_{t,ij}^{n+1}(t)|\bigg(\frac{U_{ij}^n-L_{ij}^n}{2}\bigg)(\mathrm{d}t),\quad t\in(0,T].
  \end{align*}
  Applying the triangle inequality and the fact that $U_{ij}^n+L_{ij}^n$ and $U_{ij}^n-L_{ij}^n$ are increasing, we arrive at the bound
  \begin{align*}
    |R_{t,ij}^{n+1}(t)||\ul{\Lambda}_{ij}^{n}|(\mathrm{d}t)\leq |R_{t,ij}^{n+1}(t)|U_{ij}^n(\mathrm{d}t),\quad \forall i,j\in\mathcal{J}:i\neq j
  \end{align*}
  Furthermore by Assumption~\ref{ass:regularity_BSDE} we obtain the bounds 
  \begin{align*}
    |b_{z_n}^n(t,f(t))||\gamma_{z_n}^n|(\mathrm{d}t)&\leq \beta(t)(1+|f(t)|)|\bar{\gamma}|(\mathrm{d}t)\leq K_{\beta}(1+\|f\|_{BV})|\bar{\gamma}|(\mathrm{d}t)\\
    |b_{z_nj}^n(t,f^{n+1}_{t,z_n j}(t),f(t))|&\leq \beta(t)(1+|f(t)|+|f^{n+1}_{t,z_n j}(t)|)\leq K_{\beta}\bigg(1+\|f\|_{BV}+\sup_{\omega^{n+1}\in\Omega^{n+1}}\|f^{n+1}\|_{BV}(\omega^{n+1})\bigg).
  \end{align*}
  By the triangle inequality and Proposition~\ref{prop:variation_properties}(iv) we thus have 
  \begin{align*}
    \|F(f)\|_{BV}\leq& |Y^n|+K_{\beta}(1+\|f\|_{BV})|\gamma^n_{z_n}|(T)\\
    &+ (1+K_{\beta})\bigg(1+\|f\|_{BV}+\sup_{\omega^{n+1}\in\Omega^{n+1}}\|f^{n+1}\|_{BV}(\omega^{n+1})\bigg)\sum_{j:j\neq z_n}|U_{z_n j}^n|(T)<\infty.
  \end{align*}
  Next, the goal is to show that the mapping $F$ is a contraction. Let $f,g\in\text{BV}_{[0,T]}$ and note that
  \begin{align*}
    F(f)(t)-F(g)(t)=&\int_{(t,T]}b_{z_n}^n(s,f(s))-b_{z_n}^n(s,g(s))\gamma_{z_n}(\mathrm{d}s)\\
    &+\sum_{j:j\neq z_n}\int_{(t,T]}|R_{t,z_nj}^{n+1,f}(s)|-|R_{t,z_nj}^{n+1,g}(s)|\bigg(\frac{U_{z_nj}^n-L_{z_nj}^n}{2}\bigg)(\mathrm{d}t)\\
    &+\sum_{j:j\neq z_n}\int_{(t,T]}R_{t,z_nj}^{n+1,f}(s)-R_{t,z_nj}^{n+1,g}(s)\bigg(\frac{U_{z_nj}^n+L_{z_nj}^n}{2}\bigg)(\mathrm{d}t),\quad \forall t\in[0,T].
  \end{align*}
  Let now $0\leq t_1<t_2\leq T$. By applying the triangle inequality, the reverse triangle inequality, the definitions of $R_{t,ij}^{n+1,f}$ and $R_{t,ij}^{n+1,g}$, and Assumption~\ref{ass:regularity_BSDE}(iii)-(iv), we arrive at 
  \begin{align*}
    |F(f)(t_2)-F(g)(t_2))&-(F(f)(t_1)-F(g)(t_1))|\leq \int_{(t_1,t_2]}K_{\beta}|f(s)-g(s)||\gamma^n|(\mathrm{d}s) \\
    &+\sum_{j:j\neq z_n}\int_{(t_1,t_2]}(1+K_{\beta})|f(s)-g(s)|U_{z_nj}^n(\mathrm{d}s).
  \end{align*}
  By inserting this into the definition of the variation, we find that this above bound with $t_1=0$ and $t_2=T$ applies to $\|f^{n,k+1}_i-f^{n,k}_i\|_{BV,T}$. Using Proposition~\ref{prop:variation_properties}(iii) and the definition of $\nu$ we get 
  \begin{align*}
    \|F(f)-F(g)\|_{BV}^{c,\nu}&\leq \int_{(0,T]}e^{-c(\nu(T)-\nu(s-))}\bigg(|f(T)-g(T)|+\int_{(s,T]}|f-g|(\mathrm{d}u)\bigg)\nu(\mathrm{d}s)
  \end{align*}
  which, after an application of Tonelli's Theorem, yields
  \begin{align*}
    \|F(f)-F(g)\|_{BV}^{c,\nu}\leq& |f(T)-g(T)|\int_{(0,T]}e^{-c(\nu(T)-\nu(s-))}\nu(\mathrm{d}s)\\
    &+\int_{(0,T]}\int_{(0,u)}e^{-c(\nu(T)-\nu(s-))}\nu(\mathrm{d}s)|f-g|(\mathrm{d}u).
  \end{align*}
  An application of Lemma~\ref{lem:important_bound} yields
  \begin{align*}
    \|F(f)-F(g)\|_{BV}^{c,\nu}\leq \frac{1}{c} \|f-g\|_{BV}^{c,\nu},
  \end{align*}
  and by choosing $c>1$, we thus have shown that $F$ is a contraction. An application of the Banach fixed-point theorem finishes the proof.
\end{proof}

Since $f^n$ depends on $f^{n+1}$, the system~(\ref{eq:system_of_equations}) is coupled upwards in $n$, and since $n\in\amsmathbb{N}_0$ there is no natural starting point from which to iterate backwards. Therefore, we consider a truncated version of~(\ref{eq:system_of_equations}). For $M\in\amsmathbb{N}$, set
\begin{align*}
  L_{ij}^{n,M}(t):=\mathds{1}_{(n< M)}L_{ij}^{n}(t),\quad n\in\amsmathbb{N}_0\\
  U_{ij}^{n,M}(t):=\mathds{1}_{(n<M)}U_{ij}^{n}(t),\quad n\in\amsmathbb{N}_0
\end{align*}
and use this to define the system
\begin{align}\label{eq:system_of_equations_M}
  \begin{split}
      f^{n,M}(\mathrm{d}t,\omega^n)=&-b_{z_n}^n(t,\omega^n,f^{n,M}(t,\omega^n))\gamma_{z_n}^n(\mathrm{d}t,\omega^n)-\sum_{j:j\neq z_n}R_{z_nj}^{n+1,M}(t,\omega^n)\ul{\Lambda}^{n,M}_{z_n j}(\mathrm{d}t,\omega^n)\\
      f^{n,M}(T,\omega^n)=&Y^n(\omega^n)
      \end{split}
\end{align}
for $t\in (0,T]$, $\omega^n\in\Omega^n$ and $n\in\amsmathbb{N}_0$, where 
\begin{align*}
  R_{ij}^{n+1,M}(t)&=b_{ij}^n(t,f^{n,M}(t),f^{n+1,M}_{t,ij}(t))+f^{n+1,M}_{t,ij}(t)-f^{n,M}(t)\\
  \ul{\Lambda}_{ij}^{n,M}(\mathrm{d}t)&=\mathds{1}_{(R_{t,ij}^{n+1,M}(t)\geq 0)}U_{ij}^{n,M}(\mathrm{d}t)+\mathds{1}_{(R_{t,ij}^{n+1,M}(t)< 0)}L_{ij}^{n,M}(\mathrm{d}t).
\end{align*}
Since for $n\geq M$ the function $f^{n,M}$ does not depend on $f^{n+1,M}$, we have successfully constructed a system in which only finitely many $f^n$ are coupled.

From now on we let $K_{\beta}:=\sup_{t\in[0,T]}\beta(t)$, $K_Y:=\sup_{\omega\in\Omega}|Y(\omega)|$ and define $\kappa:=1+K_{\beta}$ and $K_1:=1+2K_Y$. We let $c>1$ and set
\begin{align*}
  \nu(t):=(1+K_{\beta})\bigg(\bar{\gamma}(t)+\sum_{i,j:i\neq j}\alpha_{ij}(t)\bigg).
\end{align*}
Note that $\nu$ does not depend on $\omega^n$ for any $n\in\amsmathbb{N}_0$. We now have the following result:
\begin{lemma}\label{lem:existence_M}
  It holds that 
  \begin{align*}
    \sup_{M\in\amsmathbb{N}_0}\sup_{n\in\amsmathbb{N}_0,\omega^n\in\Omega^n}\|f^{n,M}\|_{BV}(\omega^n)\leq K_Y + e^{(4e^{2\nu(T)}+1)\nu(T)}K_1,
  \end{align*}
  and there exists a unique solution of the system~(\ref{eq:system_of_equations_M}) for all $M\in\amsmathbb{N}_0$.
\end{lemma}
\begin{proof}
  Fix $M\in\amsmathbb{N}_0$ and let $n\geq M$. In this case, we have that 
  \begin{align*}
    f^{n,M}(t)=Y^n+\int_{(t,T]}b^n_{z_n}(s,f^{n,M}(s))\gamma_{z_n}(\mathrm{d}s),\quad \forall t\in(0,T] 
  \end{align*}
  Since $f^{n,M}$ does not depend on $f^{n+1,M}$, Lemma~\ref{lem:existence_1} immidiately yields existence and uniqueness. For any $t\in [0,T]$ similar arguments as in the proof of Lemma~\ref{lem:existence_1} yield
  \begin{align*}
    \int_{(t,T]}e^{-c(\nu(T)-\nu(s-))}|f^{n,M}|(\mathrm{d}s)\leq& \int_{(t,T]}e^{-c(\nu(T)-\nu(s-))}(1+K_{Y})\nu(\mathrm{d}s)\\
    &+\int_{(t,T]}e^{-c(\nu(T)-\nu(s-))}\int_{(s,T]}|f^{n,M}|(\mathrm{d}u)\nu(\mathrm{d}s)\\
    \leq& \frac{c}{c-1}\int_{(t,T]}e^{-c(\nu(T)-\nu(s-))}K_1\nu(\mathrm{d}s) \leq \frac{1}{c-1}K_1<\infty,\quad \forall t\in[0,T].
  \end{align*}
  Choosing $c=2$, we consequently obtain
  \begin{align*}
    \|f^{n,M}\|_{BV}(\omega^n)\leq K_Y+e^{2\nu(T)}K_1\leq K_Y+e^{(4e^{2\nu(T)}+1)\nu(T)}K_1\quad \forall \omega^n\in\Omega^n.
  \end{align*}
  which implies the boundedness condition.

  For $n<M$ we will prove existence, uniqueness and the boundedness condition by proving the auxiliary result that for each $n\in\amsmathbb{N}_0$, there exists a non-decreasing function $g^{n,M}\in\text{BV}_{[0,T]}$ such that 
  \begin{align}\label{eq:g_bound}
    \int_{(t,T]}|f^{n,M}|(\omega^n)(\mathrm{d}s)\leq\int_{(t,T]}g^{n,M}(\mathrm{d}s),\quad\forall t\in[0,T],\,\forall\omega^n\in\Omega^n,
  \end{align}
  which we will prove by backwards induction starting at $n=M$. For $n\geq M$ we can set $c=2$ and use the previous calculations to obtain (\ref{eq:g_bound}) with $g^{n,M}(\mathrm{d}s)=2e^{2\nu(T)}K_1\nu(\mathrm{d}s)$. Let now $n<M$ and suppose that $f^{n+1,M}$ satisfies (\ref{eq:g_bound}), exists and is unique. Then since the bound (\ref{eq:g_bound}) does not depend on $\omega^{n+1}$, Lemma~\ref{lem:existence_1} immediately yields existence and uniqueness of $f^{n,M}$. By similar calculations as before, we obtain
  \begin{align*}
    \int_{(t,T]}e^{-c(\nu(T)-\nu(s-))}&|f^{n,M}|(\mathrm{d}s)\leq \int_{(t,T]}e^{-c(\nu(T)-\nu(s-))}\int_{(s,T]}|f^{n,M}|(\mathrm{d}u)\nu(\mathrm{d}s)\\
    &+\int_{(t,T]}e^{-c(\nu(T)-\nu(s-))}\bigg(K_1\nu(\mathrm{d}s)+\sum_{j:j\neq z_n}\int_{(s,T]}|f^{n+1,M}_{s,z_n j}|(\mathrm{d}u)\kappa U_{z_n j}(\mathrm{d}s)\bigg),
  \end{align*}
  for all $t\in[0,T]$. Using the induction hypothesis, switching the order of integration by virtue of Tonelli's Theorem, using Lemma~\ref{lem:important_bound} and rearranging yields
  \begin{align*}
    \int_{(t,T]}e^{-c(\nu(T)-\nu(s-))}&|f^{n,M}|(\mathrm{d}s)\leq \frac{c}{c-1}\int_{(t,T]}e^{-c(\nu(T)-\nu(s-))}\bigg(K_1+\int_{(s,T]}g^{n+1,M}(\mathrm{d}u)\bigg)\nu(\mathrm{d}s),
  \end{align*}
  for all $t\in[0,T]$. Setting $c=2$ and rearranging we obtain (\ref{eq:g_bound}) with 
  \begin{align*}
    g^{n,M}(\mathrm{d}s)=2e^{2\nu(T)}\bigg(K_1+\int_{(s,T]}g^{n+1,M}(\mathrm{d}u)\bigg)\nu(\mathrm{d}s).
  \end{align*}
  We thus established that (\ref{eq:g_bound}) is valid for all $n<M$ and that existence and uniqueness hold for all $n<M$ as well. Using this in conjunction with switching the order of integration by virtue of Tonelli's Theorem and Lemma~\ref{lem:important_bound} to iterate upwards from $n$ to $M$, we obtain:
  \begin{align*}
    \int_{(t,T]}e^{-c(\nu(T)-\nu(s-))}&|f^{n,M}|(\mathrm{d}s)\leq \frac{c}{c-1}\int_{(t,T]}e^{-c(\nu(T)-\nu(s-))}\bigg(K_1+\int_{(s,T]}g^{n+1,M}(\mathrm{d}u)\bigg)\nu(\mathrm{d}s)\\
    \leq&\frac{1}{c-1}K_1+\frac{1}{c-1}\int_{(0,T]}e^{-c(\nu(T)-\nu(s-))}g^{n+1,M}(\mathrm{d}s)\\
    \leq&\frac{1}{c-1}K_1+\frac{1}{c-1}\int_{(0,T]}e^{-c(\nu(T)-\nu(s-))}2e^{2\nu(T)}\bigg(K_1+\int_{(s,T]}g^{n+2,M}(\mathrm{d}u)\bigg)\nu(\mathrm{d}s)\\
    \leq&\frac{2e^{2\nu(T)}}{c-1}K_1+\bigg(\frac{2e^{2\nu(T)}}{c-1}\bigg)^2 K_1 +\bigg(\frac{2e^{2\nu(T)}}{c-1}\bigg)^2\int_{(0,T]}e^{-c(\nu(T)-\nu(s-))}g^{n+2,M}(\mathrm{d}s)\\
    \leq&\cdots\leq K_1\sum_{\ell=1}^{M-n}\bigg(\frac{2e^{2\nu(T)}}{c-1}\bigg)^{\ell}\leq K_1\sum_{\ell=1}^{\infty}\bigg(\frac{2e^{2\nu(T)}}{c-1}\bigg)^{\ell},\quad \forall t\in[0,T].
  \end{align*}
  Setting $c=4e^{2\nu(T)}+1$ yields
  \begin{align*}
    \int_{(0,T]}|f^{n,M}|(\mathrm{d}s,\omega^n)\leq e^{(4e^{2\nu(T)}+1)\nu(T)}K_1\sum_{\ell=1}^{\infty}\frac{1}{2^{\ell}}=e^{(4e^{2\nu(T)}+1)\nu(T)}K_1,\quad \forall\omega^n\in\Omega^n.
  \end{align*}
  As this bound is independent of $M$, $n$, and $\omega^n$, taking the supremums yields the desired result.
\end{proof}

\begin{lemma}\label{lem:Cauchy}
  For each $M\in\amsmathbb{N}_0$ let $(f^{n,M})_{n\in\amsmathbb{N}_0}$ be a solution of the system (\ref{eq:system_of_equations_M}). For any $n\in\amsmathbb{N}_0$ the sequence $(f^{n,M})_{M\in\amsmathbb{N}_0}$ is a Cauchy sequence in $(\text{BV}_{[0,T]},\|\cdot\|_{\text{BV}})$.
\end{lemma}
\begin{proof}
  Fix $n\in\amsmathbb{N}_0$ and suppose that $n\leq N<M$. We will start by proving the auxiliary result that for each $n<N$, there exists a non-decreasing function $g^{N,M}_n\in\text{BV}_{[0,T]}$ such that 
  \begin{align}\label{eq:g_bound2}
    \int_{(t,T]}|f^{n,M}-f^{n,M}|(\mathrm{d}s,\omega^n)\leq\int_{(t,T]}g^{N,M}_{n+1}(\mathrm{d}s),\quad\forall t\in[0,T],\,\forall\omega^n\in\Omega^n.
  \end{align}
  As in the proof of Lemma~\ref{lem:existence_M}, we will use backwards induction. For $n=N$, we can, by similar calculations, Tonelli's Theorem, and an application of Lemma~\ref{lem:existence_1} and Lemma~\ref{lem:important_bound}, arrive at
  \begin{align*}
    \int_{(t,T]}e^{-c(\nu(T)-\nu(s-))}&|f^{N,M}-f^{N,N}|(\mathrm{d}s)\leq \int_{(t,T]}e^{-c(\nu(T)-\nu(s-))}\int_{(s,T]}|f^{N,M}-f^{N,N}|(\mathrm{d}u)\nu(\mathrm{d}s)\\
    &+\int_{(t,T]}e^{-c(\nu(T)-\nu(s-))}\sum_{j:j\neq z_n}\kappa\big(1+\|f^{N+1,M}_{s,z_nj}\|_{BV}+\|f^{N,M}\|_{BV}\big)U_{z_nj}(\mathrm{d}s)\\
    \leq & \frac{c}{c-1}\int_{(t,T]}e^{-c(\nu(T)-\nu(s-))}(1+2K_2)\nu(\mathrm{d}s),\quad \forall t\in [0,T],\,\forall\omega^N\in\Omega^N,
  \end{align*}
  where $K_2:=K_Y + e^{(4e^{2\nu(T)}+1)\nu(T)}K_1$. Thus for $c=2$, we have that (\ref{eq:g_bound2}) is satisfied with $g^{N,M}_{N}(\mathrm{d}s)=2e^{2\nu(T)}(1+2K_2)\nu(\mathrm{d}s)$. Assume now that (\ref{eq:g_bound2}) holds for $n+1$. Then calculations similar to those in the proof of Lemma~\ref{lem:existence_1} and the induction hypothesis yield
  \begin{align*}
    \int_{(t,T]}e^{-c(\nu(T)-\nu(s-))}|f^{n,M}-f^{n,M}|(\mathrm{d}s)&\leq \frac{c}{c-1}\int_{(t,T]}e^{-c(\nu(T)-\nu(s-))}\int_{(s,T]}g^{N,M}_{n+1}(\mathrm{d}u)\nu(\mathrm{d}s),
  \end{align*}
  for all $t\in[0,T]$, which for $c=2$ yields (\ref{eq:g_bound2}) with
  \begin{align*}
    g^{N,M}_{n}(\mathrm{d}s)=2e^{2\nu(T)}\int_{(s,T]}g^{N,M}_{n+1}(\mathrm{d}u)\nu(\mathrm{d}s).
  \end{align*}
  Iterating forwards from $n$ to $N$ as in the proof of Lemma~\ref{lem:existence_M} yields
  \begin{align*}
    \int_{(0,T]}e^{-c(\nu(T)-\nu(s-))}&|f^{N,M}-f^{N,N}|(\mathrm{d}s)\leq \frac{1}{c-1}\int_{(0,T]}e^{-c(\nu(T)-\nu(s-))}g_{n+1}^{N,M}(\mathrm{d}s)\\
    \leq & \frac{2e^{2\nu(T)}}{c-1}\int_{(0,T]}2e^{2\nu(T)}\int_{(s,T]}g_{n+2}^{N,M}(\mathrm{d}u)\nu(\mathrm{d}s)\\
    \leq & \bigg(\frac{2e^{2\nu(T)}}{c-1}\bigg)^2 \int_{(0,T]}e^{-c(\nu(T)-\nu(s-))}g_{n+2}^{N,M}(\mathrm{d}s)\leq\cdots\leq \bigg(\frac{2e^{2\nu(T)}}{c-1}\bigg)^{N-n}(1+2K_2)
  \end{align*}
  Setting $c=4e^{2\kappa\nu(T)}+1$ and rearranging yields the bound 
  \begin{align*}
    \|f^{n,M}-f^{n,N}\|_{BV}(\omega^n)\leq \frac{1}{2^{N-n}}(1+2K_2)e^{(4e^{2\nu(T)}+1)\nu(T)},\quad\forall\omega^n\in\Omega^n,
  \end{align*}
  which vanishes for $N\rightarrow\infty$, thus finishing the proof.
\end{proof}

\begin{lemma}\label{lem:limit}
  Assume that for all $n\in\amsmathbb{N}_0$ the sequence $(f^{n,M})_{M\in\amsmathbb{N}_0}$ has the limit $f^n$ in $\text{BV}_{[0,T]}$. Then $(f^n)_{n\in\amsmathbb{N}_0}$ is a solution of the system~(\ref{eq:system_of_equations}) and the process $X$ with jump-count representation $(f^n)_{n\in\amsmathbb{N}_0}$ is bounded on finite intervals.
\end{lemma}
\begin{proof}
  Let $n\in\amsmathbb{N}_0$ be given. Let $f^n$ be the limit of $(f^{n,M})_{M\in\amsmathbb{N}_0}$ and set
  \begin{align*}
    \ti{f}^n(t):=Y^n+\int_{(t,T]}b^n_{z_n}(t,f^n(t))\gamma_{z_n}^n(\mathrm{d}s)-\sum_{j:j\neq z_n}\int_{(t,T]}R_{s,z_nj}^{n+1}(s)\ul{\Lambda}_{z_nj}^n(\mathrm{d}s),\quad t\in[0,T].
  \end{align*}
  Let $M$ be large enough such that $n<M$. Then $L_{ij}^{n,M}=L_{ij}^n$ and $U_{ij}^{n,M}=U_{ij}^n$ by definition and thus by similar arguments as in the previous proofs, we get
  \begin{align*}
    \|f^{n,M}-\ti{f}^{n}\|_{BV}\leq& \nu(T)\|f^{n,M}-f^n\|_{BV}+\sum_{j:j\neq z_n}\int_{(0,T]}\|f^{n+1,M}_{t,z_nj}-f^{n+1}_{t,z_nj}\|_{BV}\nu(\mathrm{d}t)
  \end{align*}
  By assumption $\lim_{n\rightarrow\infty}\|f^{n,M}-f^n\|_{BV}=0$ and $\lim_{n\rightarrow\infty}\|f^{n+1,M}_{t,z_nj}-f^{n+1}_{t,z_nj}\|_{BV}=0$ for all $t\in [0,T]$. Thus, the first term vanishes, and the bound in Lemma~\ref{lem:existence_1} allows us to apply the dominated convergence theorem to the second term, which allows us to conclude that 
  \begin{align*}
    \lim_{M\rightarrow\infty}\|f^{n,M}-\ti{f}^{n}\|_{BV}=0.
  \end{align*}
  Thus the uniqueness of limits yields $(\ti{f}^{n})_{n\in\amsmathbb{N}_0}=(f^{n})_{n\in\amsmathbb{N}_0}$, which means that $(f^{n})_{n\in\amsmathbb{N}_0}$ indeed is a solution of (\ref{eq:system_of_equations}). Finally, as the limit operator preserves bounds, we have by Lemma~\ref{lem:existence_M} that
  \begin{align*}
    \sup_{n\in\amsmathbb{N}_0,\omega^n\in\Omega^n}\|f^n\|_{BV}(\omega^n)\leq K_Y+e^{(4e^{2\nu(T)}+1)\nu(T)}K_1,
  \end{align*}
  which implies that $X(t)=\sum_{n\in\amsmathbb{N}_0}I^n(t)f^n(t)$ is bounded on finite intervals.
\end{proof}

\begin{lemma}\label{lem:uniqueness}
  Let $X$ be a solution of (\ref{eq:worst_case_BSDE}) that is càdlàg, of finite variation, and bounded on finite intervals. Then $X$ is unique.
\end{lemma}
\begin{proof}
  Assume that $X_1,X_2$ both are solutions of (\ref{eq:worst_case_BSDE}), bounded on finite intervals, with jump-count representations $(f^n)_{n\in\amsmathbb{N}_0}$ and $(g^n)_{n\in\amsmathbb{N}_0}$. Then the process $X_1-X_2$ is bounded on finite intervals, and thus there exists a constant $K>0$ such that 
  \begin{align*}
    \sup_{t\in[0,T]}\sup_{n\in\amsmathbb{N}_0,\omega^n\in\Omega^n}|f^n(t,\omega^n)-g^n(t,\omega^n)|\leq K.
  \end{align*}
  By a backwards induction argument starting at any $M\in\amsmathbb{N}_0$ followed by a forward iteration from any $n<M$ to $M$, exactly as in the proof of Lemma~\ref{lem:Cauchy}, we obtain that for all $M\in\amsmathbb{N}$ and all $n< M$ it holds that 
  \begin{align*}
    \int_{(0,T]}e^{-c(\nu(T)-\nu(s-))}|f^n-g^n|(\mathrm{d}s)\leq \bigg(\frac{2e^{2\nu(T)}}{c-1}\bigg)^{M-n}K.
  \end{align*}
  Choosing $c=4e^{2\nu(T)}+1$ yields for any $M\in\amsmathbb{N}_0$ and $n<M$ the bound
  \begin{align*}
    \|f^n-g^n\|_{BV}(\omega^n)\leq \frac{1}{2^{M-n}}h(T)e^{(4e^{2\nu(T)}+1)\nu(T)},\quad\forall\omega^n\in\Omega^n.
  \end{align*}
  By taking the infimum over $M$ on both sides we obtain $\|f^n-g^n\|_{BV}(\omega^n)=0$ for all $\omega^n\in\Omega^n$ and $n\in\amsmathbb{N}_0$, thus yielding the desired uniqueness.
\end{proof}

\section{The martingale optimality principle}
We now formulate our robust utility maximisation problem and show the associated martingale optimality principle, which is a sufficient condition for the existence of a solution to our posed problem. Both the problem of determining worst- and best-case prospective reserves and the robust consumption-insurance problem are special cases of the problem presented in this section. Consider again the filtered canonical measurable space $(\Omega,\mathcal{F},\amsmathbb{F})$. Let $\pi$ be a known and fixed initial distribution and assume that the cumulative transition rates $\Lambda=(\Lambda_{ij})_{i\neq j}$ determining the true measure $\amsmathbb{P}^{\Lambda}$ are unknown. Instead, statistical methods or expert judgement provide known upper and lower bounds for the true cumulative transition rates, which we assume are given by the cumulative transition rates $(L_{ij})_{i\neq j}$ and $(U_{ij})_{i\neq j}$ satisfying Assumption~\ref{ass:cumulative_tr} such that 
\begin{equation}\label{eq:LLU}
  L_{ij}(s,\omega)-L_{ij}(t,\omega)\leq U_{ij}(s,\omega)-U_{ij}(t,\omega),\quad\forall 0\leq t\leq s,\, \forall\omega\in\Omega,\quad\forall i,j\in\mathcal{J}:i\neq j.
\end{equation}
We are only interested in biometric scenarios within these bounds, and thus, all admissible cumulative transition rates are defined as follows:
\begin{definition}
  The processes $\Lambda=(\Lambda_{ij})_{i\neq j}$ are admissible cumulative transition rates if they satisfy Assumption~\ref{ass:cumulative_tr} and if it holds that 
  \begin{align}\label{eq:LLU_admissible}
    L_{ij}(s,\omega)-L_{ij}(t,\omega)\leq \Lambda_{ij}(s,\omega)-\Lambda_{ij}(t,\omega)\leq U_{ij}(s,\omega)-U_{ij}(t,\omega),\quad\forall 0\leq t\leq s,\, \forall \omega\in\Omega.
  \end{align}
  for all $i,j\in\mathcal{J}$ with $i\neq j$. The set of admissible $\Lambda$ is denoted by $\mathcal{M}$.
\end{definition}
With the admissible biometric scenarios in place, we fix a deterministic terminal time $T>0$ and consider the following problem:
\begin{equation}\label{eq:MaxMinProblem}
  \max_{\theta\in\mathcal{A}}\min_{\Lambda\in\mathcal{M}}\amsmathbb{E}^{\Lambda}\left[\int_{(0,T]}\frac{1}{\mathcal{E}_{\Phi}(s)}U^{\theta,\Lambda}(\mathrm{d}s)+\frac{1}{\mathcal{E}_{\Phi}(T)}Y^{\theta,\Lambda}\right].
\end{equation}
The process $U^{\theta,\Lambda}$ is the cumulative utility process, which measures the total utility gained over time, depending on the control $\theta$ and the biometric scenario $\Lambda$, while the $\mathcal{F}_T$-measurable random variable $Y^{\theta,\Lambda}$ measures the terminal utility. The control $\theta$ is assumed to be an optional process taking values in a measurable space $(E,\mathcal{E})$ and the process $\mathcal{E}_{\Phi}$ is defined as the generalised exponential
\begin{align*}
  \mathcal{E}_{\Phi}(\mathrm{d}t)=\mathcal{E}_{\Phi}(t-)\Phi(\mathrm{d}t),\quad \mathcal{E}_{\Phi}(0)=1\quad\Leftrightarrow\quad  \mathcal{E}_{\Phi}(t):=e^{\Phi^c(t)-\Phi^c(0)}\prod_{0<s\leq t}(1+\Delta\Phi(s))
\end{align*}
for some predictable, càdlàg, finite variation process $\Phi$ that is bounded on finite intervals and satsifies $\inf_{(t,\omega)\in[0,T]\times\Omega}\Delta\Phi(t,\omega)>-1$. The process $\Phi$ can be interpreted as a subjective cumulative discount rate taking into account the subjective time-value of utility. The class of admissible controls is as follows:
\begin{definition}
  The optional process $\theta:[0,T]\times\Omega\rightarrow E$ is admissible if for all $\Lambda\in\mathcal{M}$ it holds that
  $U^{\theta,\Lambda}$ is a càdlàg, finite variation process and
  \begin{gather*}
    \sup_{t\in[0,T]}\amsmathbb{E}_{t,i}^{\Lambda}[|U^{\theta,\Lambda}(T)-U^{\theta,\Lambda}(t)+Y^{\theta,\Lambda}|](\omega)<\infty,\quad \forall\omega\in\Omega,\,\forall i\in\mathcal{J}.
  \end{gather*}
  The set of admissible $\theta$ is denoted by $\mathcal{A}$.
\end{definition}

The objective of (\ref{eq:MaxMinProblem}) is to find an admissible pair $(\ul{\theta},\ul{\Lambda})$ such that $\ul{\theta}\in\mathcal{A}$ is the best possible control in the worst possible biometric scenario $\ul{\Lambda}\in\mathcal{M}$. In order to formalise this mathematically, we associate with the problem (\ref{eq:MaxMinProblem}), the expected utility process $\mathcal{V}(t,\theta,\Lambda)$ given by 
\begin{align*}
  \mathcal{V}(t,\theta,\Lambda):=\amsmathbb{E}_{t,Z(t)}^{\Lambda}\bigg[\int_t^T\frac{\mathcal{E}_{\Phi}(t)}{\mathcal{E}_{\Phi}(s)}U^{\theta,\Lambda}(\mathrm{d}s)+\frac{\mathcal{E}_{\Phi}(t)}{\mathcal{E}_{\Phi}(T)}Y^{\theta,\Lambda}\bigg].
\end{align*}
Note that $\mathcal{V}(t,\theta,\Lambda)$ is an optional projection of the discounted form and it can be interpreted as the subjective expected present value of future utility associated with the control $\theta$ under the biometric scenario $\Lambda$ at time $t$. This process is important because it allows us to define what constitutes optimality.
\begin{definition}
  The admissible pair $(\ul{\theta},\ul{\Lambda})$ constitutes the robust optimal control and the worst-case biometric scenario if and only if for all $t\in[0,T]$ it holds that
  \begin{align*}
    \mathcal{V}(t,\theta,\ul{\Lambda})\leq\mathcal{V}(t,\ul{\theta},\ul{\Lambda})\leq\mathcal{V}(t,\ul{\theta},\Lambda),\quad \forall\theta\in\mathcal{A},\,\Lambda\in\mathcal{M}.
  \end{align*}
\end{definition}
We now come to the martingale optimality principle, which serves as a sufficient condition for the existence of robust optimal controls and worst-case biometric scenarios. Let $\ul{\theta}\in\mathcal{A}$ and $\ul{\Lambda}\in\mathcal{M}$ be the candidate robust optimal control and worst-case biometric scenario and assume that there exists a càdlàg, finite variation process $V$, such that $V(T)=Y^{\ul{\theta},\ul{\Lambda}}$. For arbitrary $\theta\in\mathcal{A}$ and $\Lambda\in\mathcal{M}$ let $M^{\theta,\Lambda}(t)$ be given by
\begin{align*}
  M^{\theta,\Lambda}(t):=\int_{(0,t]}\frac{1}{\mathcal{E}_{\Phi}(s)}U^{\theta,\Lambda}(\mathrm{d}s)+\frac{1}{\mathcal{E}_{\Phi}(t)}V(t).
\end{align*}
The process $M^{\theta,\Lambda}$ consists of two parts. The first part is the utility accumulated until time $t$ under the control $\theta$ and the biometric scenario $\Lambda$, discounted back to time zero. The second part is the expected present value of future utility under the control $\ul{\theta}$ and the biometric scenario $\ul{\Lambda}$, discounted back to time zero. We now have the following result.
\begin{theorem}[Martingale optimality principle]\label{th:MOP}
  Let $(\ul{\theta},\ul{\Lambda})$ be admissible and let $V$ be càdlàg and of finite variation satisfying $V(T)=Y^{\ul{\theta},\ul{\Lambda}}$. If
  \begin{enumerate}
    \item[(i)] $M^{\ul{\theta},\ul{\Lambda}}$ is a $(\amsmathbb{F},\amsmathbb{P}^{\ul{\Lambda}})$-martingale
    \item[(ii)] $M^{\theta,\ul{\Lambda}}$ is a $(\amsmathbb{F},\amsmathbb{P}^{\ul{\Lambda}})$-supermartingale for all $\theta\in\mathcal{A}$
    \item[(iii)] $M^{\ul{\theta},\Lambda}$ is a $(\amsmathbb{F},\amsmathbb{P}^{\Lambda})$-submartingale for all $\Lambda\in\mathcal{M}$
  \end{enumerate}
  then $\ul{\theta}$ and $\ul{\Lambda}$ are robust optimal controls and
  \begin{align*}
    V(t)=\sup_{\theta\in\mathcal{A}}\inf_{\Lambda\in\mathcal{M}}\mathcal{V}(t,\theta,\Lambda)=\inf_{\Lambda\in\mathcal{M}}\sup_{\theta\in\mathcal{A}}\mathcal{V}(t,\theta,\Lambda)=\mathcal{V}(t,\ul{\theta},\ul{\mu}).
  \end{align*}
\end{theorem}
\begin{proof}
  We have to show that $(\ul{\theta},\ul{\Lambda})$ is a saddle point of $\mathcal{V}(t,\theta,\Lambda)$ for all $t\in [0,T]$. Since $M^{\ul{\theta},\ul{\Lambda}}$ is a $(\amsmathbb{F},\amsmathbb{P}^{\ul{\Lambda}})$-martingale, we have that $M^{\ul{\theta},\ul{\Lambda}}(t)=\amsmathbb{E}^{\ul{\Lambda}}_{t,Z(t)}[M^{\ul{\theta},\ul{\Lambda}}(T)]$, which yields
  \begin{align*}
    \int_0^t\frac{1}{\mathcal{E}_{\Phi}(s)}U^{\ul{\theta},\ul{\Lambda}}(\mathrm{d}s)+\frac{1}{\mathcal{E}_{\Phi}(t)}V(t)=\amsmathbb{E}_{t,Z(t)}^{\ul{\Lambda}}\bigg[\int_0^T \frac{1}{\mathcal{E}_{\Phi}(s)}U^{\ul{\theta},\ul{\Lambda}}(\mathrm{d}s)+\frac{1}{\mathcal{E}_{\Phi}(T)}u_{Z(T)}^T(X^{\ul{\theta},\ul{\Lambda}}(T))\bigg].
  \end{align*}
  Rearranging gives us 
  \begin{align*}
    V(t)=\amsmathbb{E}_{t,Z(t)}^{\ul{\Lambda}}\bigg[\int_t^T\frac{\mathcal{E}_{\Phi}(t)}{\mathcal{E}_{\Phi}(s)}U^{\ul{\theta},\ul{\Lambda}}(\mathrm{d}s)+\frac{\mathcal{E}_{\Phi}(t)}{\mathcal{E}_{\Phi}(T)}u_{Z(T)}^T(X^{\ul{\theta},\ul{\Lambda}}(T))\bigg]=\mathcal{V}(t,\ul{\theta},\ul{\Lambda}).
  \end{align*}
  Let $\theta\in\mathcal{A}$ be arbitrary. Since $M^{\theta,\ul{\Lambda}}$ is a $(\amsmathbb{F},\amsmathbb{P}^{\ul{\Lambda}})$-supermartingale, we have that $M^{\ul{\theta},\ul{\Lambda}}(t)\geq \amsmathbb{E}^{\ul{\Lambda}}_{t,Z(t)}[M^{\ul{\theta},\ul{\Lambda}}(T)]$, which yields
  \begin{align*}
    \int_0^t\frac{1}{\mathcal{E}_{\Phi}(s)}U^{\theta,\ul{\Lambda}}(\mathrm{d}s)+\frac{1}{\mathcal{E}_{\Phi}(t)}V(t)\geq\amsmathbb{E}_{t,Z(t)}^{\ul{\Lambda}}\bigg[\int_0^T \frac{1}{\mathcal{E}_{\Phi}(s)}U^{\theta,\ul{\Lambda}}(\mathrm{d}s)+\frac{1}{\mathcal{E}_{\Phi}(t)}u_{Z(T)}^T(X^{\theta,\ul{\Lambda}}(T))\bigg].
  \end{align*}
  Since $V(t)=\mathcal{V}(t,\ul{\theta},\ul{\Lambda})$, rearranging yields
  \begin{align*}
    \mathcal{V}(t,\ul{\theta},\ul{\Lambda})\geq \amsmathbb{E}_{t,Z(t)}^{\ul{\Lambda}}\bigg[\int_t^T\frac{\mathcal{E}_{\Phi}(t)}{\mathcal{E}_{\Phi}(s)}U^{\theta,\ul{\Lambda}}(\mathrm{d}s)+\frac{\mathcal{E}_{\Phi}(t)}{\mathcal{E}_{\Phi}(T)}u_{Z(T)}^T(X^{\theta,\ul{\Lambda}}(T))\bigg]=\mathcal{V}(t,\theta,\ul{\Lambda}),\quad\forall\theta\in\mathcal{A}.
  \end{align*}
  Let $\Lambda\in\mathcal{M}$ be arbitrary. Since $M^{\ul{\theta},\Lambda}$ is a $(\amsmathbb{F},\amsmathbb{P}^{\Lambda})$-submartingale, we have that $M^{\ul{\theta},\ul{\Lambda}}(t)\leq \amsmathbb{E}^{\ul{\Lambda}}_{t,Z(t)}[M^{\ul{\theta},\ul{\Lambda}}(T)]$, which yields
  \begin{align*}
    \int_0^t\frac{1}{\mathcal{E}_{\Phi}(s)}U^{\ul{\theta},\Lambda}(\mathrm{d}s)+\frac{1}{\mathcal{E}_{\Phi}(t)}V(t)\leq\amsmathbb{E}_{t,Z(t)}^{\Lambda}\bigg[\int_0^T \frac{1}{\mathcal{E}_{\Phi}(s)}U^{\ul{\theta},\Lambda}(\mathrm{d}s)+\frac{1}{\mathcal{E}_{\Phi}(t)}u_{Z(T)}^T(X^{\ul{\theta},\Lambda}(T))\bigg].
  \end{align*}
  Since $V(t)=\mathcal{V}(t,\ul{\theta},\ul{\Lambda})$, rearranging yields
  \begin{align*}
    \mathcal{V}(t,\ul{\theta},\ul{\Lambda})\leq \amsmathbb{E}_{t,Z(t)}^{\Lambda}\bigg[\int_t^T \frac{\mathcal{E}_{\Phi}(t)}{\mathcal{E}_{\Phi}(s)}U^{\ul{\theta},\Lambda}(\mathrm{d}s)+\frac{\mathcal{E}_{\Phi}(t)}{\mathcal{E}_{\Phi}(T)}u_{Z(T)}^T(X^{\ul{\theta},\Lambda}(T))\bigg]=\mathcal{V}(t,\ul{\theta},\Lambda),\quad \forall \Lambda\in\mathcal{M}.
  \end{align*}
  Thus we can conclude that $(\ul{\theta},\ul{\Lambda})$ is optimal. 
\end{proof}

If one is interested only in the case of absolutely continuous cumulative transition rates determined by the predictable transition intensities $(\mu_{ij})_{i\neq j}$, see~(\ref{eq:absolutely_continuous_Lambda}), statistical methods or expert judgement can provide upper and lower bounds $(U_{ij})_{i\neq j}$ and $(L_{ij})_{i\neq j}$ of the same absolutely continuous form, with transition intensities $(l_{ij})_{i\neq j}$ and $(u_{ij})_{i\neq j}$ satisfying Assumption~\ref{ass:intensities}. Thus instead of (\ref{eq:LLU}) we assume that
\begin{equation}\label{eq:llu}
  l_{ij}(t,\omega)\leq u_{ij}(t,\omega),\quad \forall (t,\omega)\in [0,\infty)\times\Omega,\quad \forall i,j\in\mathcal{J}:i\neq j,
\end{equation}
which of course implies (\ref{eq:LLU}). The class of admissible cumulative transition rates can then be restricted to those of the form (\ref{eq:absolutely_continuous_Lambda}) using the following class of admissible predictable transition intensities:
\begin{definition}\label{def:adm_intensity}
  The processes $\mu=(\mu_{ij})_{i\neq j}$ are admissible transition intensities if they satisfy Assumption~\ref{ass:intensities} and if it holds that 
  \begin{align}\label{eq:llu_admissible}
    l_{ij}(t,\omega)\leq\mu_{ij}(t,\omega)\leq u_{ij}(t,\omega),\quad \forall (t,\omega)\in [0,\infty)\times\Omega,\quad\forall i,j\in\mathcal{J}:i\neq j.
  \end{align}
  The set of admissible $\mu$ is denoted by $\mathcal{M}$.
\end{definition}
In that case the problem~(\ref{eq:MaxMinProblem}) can be stated as
\begin{equation}\label{eq:MaxMinProblem_intensities}
  \max_{\theta\in\mathcal{A}}\min_{\mu\in\mathcal{M}}\amsmathbb{E}^{\mu}\left[\int_{(0,T]}\frac{1}{\mathcal{E}_{\Phi}(s)}U^{\theta,\mu}(\mathrm{d}s)+\frac{1}{\mathcal{E}_{\Phi}(T)}u_{Z(T)}^T(X^{\theta,\mu}(T))\right],
\end{equation}
where the biometric scenario is characterised directly by the transition intensities. Even though problem (\ref{eq:MaxMinProblem_intensities}) reduces the class of admissible biometric scenarios compared to the problem (\ref{eq:MaxMinProblem}), the objective is the same, and the martingale optimality principle of Theorem~\ref{th:MOP} remains valid.

\section{Best- and worst-case prospective reserves}
In life insurance, the process $Z$ can be interpreted as the health status of an individual and depending on the individual's state and the transitions between states, they receive contractual payments in accordance with an insurance contract. The accumulated insurance payments are given by a process $C\in\amsmathbb{Y}$. By Proposition~\ref{prop:Rep_dyn_J}, any such process can be decomposed as follows:
\begin{align*}
  C(\mathrm{d}t)=\sum_{i\in\mathcal{J}}I_i(t-)\bigg(C_i(\mathrm{d}t)+\sum_{j:j\neq i}\ti{\Delta} C_{ij}(t)N_{ij}(\mathrm{d}t)\bigg).
\end{align*}
The processes $(C_i)_{i\in\mathcal{J}}$ can be interpreted as the accumulated sojourn payments, which the individual receives/pays as long as the individual's health status $Z(t)$ equals $i$, while the processes $(\ti{\Delta} C_{ij})_{i\neq j}$ can be interpreted as the transition payments, which the individual receives/pays upon the transition of the individual from state $i$ to state $j$. We assume that there is a deterministic terminal time $T>0$ at which the insurance contract expires and after which no more payments are made. The present value of future payments is given by
\begin{align*}
  L(t,T):=\int_{(t,T]}\frac{\mathcal{E}_{\Phi}(t)}{\mathcal{E}_{\Phi}(s)}C(\mathrm{d}s),
\end{align*}
where the cumulative interest rate $\Phi$ is predictable, càdlàg, of finite variation and bounded on finite intervals, such that $\inf_{(t,\omega)\in[0,T]\times\Omega}\Delta \Phi(t,\omega)>-1$. The process $\mathcal{E}_{\Phi}$ given by
\begin{align*}
  \mathcal{E}_{\Phi}(t)=e^{\Phi^c(t)-\Phi^c(0)}\prod_{0<s\leq t}(1+\Delta \Phi(s)),
\end{align*}
can be thought of as a savings account. The object of interest for the insurance company is the expected present value of future payments (also called the prospective reserve), defined as the optional projection of the present value of future payments given by
\begin{align*}
  \mathcal{V}(t,\Lambda):=\amsmathbb{E}^{\Lambda}_{t,Z(t)}[L(t,T)]=\sum_{i\in\mathcal{J}}I_i(t)\amsmathbb{E}_{t,i}^{\Lambda}[L(t,T)],\quad t\in [0,T],
\end{align*}
for the true biometric scenario $\Lambda$. The processes $\ti{V}_i(t)=\amsmathbb{E}_{t,i}^{\Lambda}[L(t,T)]$ are called the state-wise prospective reserves.

In practice, the true biometric scenario $\Lambda$ is unknown. Statistical methods might provide a best estimate, which can be used for calculations, and confidence bounds, that can quantify the uncertainty on the level of the cumulative transition intensities. In light of this, we are now interested in the question: Given contractual payments $C\in\amsmathbb{Y}$, a cumulative interest rate $\Phi$, and the bounds (\ref{eq:LLU}) for the true biometric scenario, what are the worst- and best-case prospective reserves
\begin{align*}
  \ul{V}(t):=\sup_{\Lambda\in\mathcal{M}}\amsmathbb{E}^{\Lambda}_{t,Z(t)}[L(t,T)]\quad\text{and}\quad \bar{V}(t):=\inf_{\Lambda\in\mathcal{M}}\amsmathbb{E}^{\Lambda}_{t,Z(t)}[L(t,T)],
\end{align*}
and what are the corresponding worst- and best-case biometric scenarios $\ul{\Lambda}$ and $\bar{\Lambda}$, which realise the supremum and infimum? Note that from the insurer's point of view, the prospective reserve is a liability and thus the worst-case value is the largest value this liability can take. From the insured's point of view, it is of course exactly the opposite. Going forward, we will refer to the insurer's point of view. 

The two problems posed above can be recognised as a special case of the problem (\ref{eq:MaxMinProblem}) as follows. The control $\theta\in\amsmathbb{Y}$ is an insurance payment process, and the set of admissible controls is $\mathcal{A}=\{C\}$. The utility process $U^{\theta,\Lambda}$ is just equal to the one admissible payment process $C$ and $Y^{\theta,\Lambda}=0$. Thus, we can apply the martingale optimality principle to obtain the following two results.

\begin{proposition}\label{prop:wc-reserve}
  Assume that $C\in\amsmathbb{Y}$ and that $\Phi$ is predictable, càdlàg, of finite variation and bounded on finite intervals with $\inf_{(t,\omega)\in[0,T]\times\Omega}\Delta \Phi(t,\omega)>-1$. Let $\ul{V}$ be the solution of the BSDE
  \begin{equation}\label{eq:wc_thiele_BSDE}
    \ul{V}(\mathrm{d}t)=\ul{V}(t-)\Phi(\mathrm{d}t)-C(\mathrm{d}t)+\sum_{i,j:i\neq j}\ul{R}_{ij}(t)(N_{ij}(\mathrm{d}t)-I_i(t-)\ul{\Lambda}_{ij}(\mathrm{d}t)),\quad \ul{V}(T)=0,
  \end{equation}
  where $\ul{R}_{ij}(t):=\ti{\Delta} C_{ij}(t)+\ul{V}_{ij}(t)-\ul{V}_i(t)$ and 
  \begin{align*}
    \ul{\Lambda}_{ij}(\mathrm{d}t)=\mathds{1}_{(\ul{R}_{ij}(t)\geq 0)}U_{ij}(\mathrm{d}t)+\mathds{1}_{(\ul{R}_{ij}(t)< 0)}L_{ij}(\mathrm{d}t).
  \end{align*}
  Then $\ul{V}(t)=\sup_{\Lambda\in\mathcal{M}}\amsmathbb{E}_{t,Z(t)}^{\bar{\Lambda}}[L(t,T)]$ and $\ul{\Lambda}=(\bar{\Lambda}_{ij})_{i\neq j}$ is the worst-case scenario.
\end{proposition}

\begin{proposition}\label{prop:bc-reserve}
  Assume that $C\in\amsmathbb{Y}$ and that $\Phi$ is predictable, càdlàg, of finite variation and bounded on finite intervals with $\inf_{(t,\omega)\in[0,T]\times\Omega}\Delta \Phi(t,\omega)>-1$. Let $\bar{V}$ be a solution of the BSDE
  \begin{equation}\label{eq:bc_thiele_BSDE}
    \bar{V}(\mathrm{d}t)=\bar{V}(t-)\Phi(\mathrm{d}t)-C(\mathrm{d}t)+\sum_{i,j:i\neq j}\bar{R}_{ij}(t)(N_{ij}(\mathrm{d}t)-I_i(t-)\bar{\Lambda}_{ij}(\mathrm{d}t)),\quad \bar{V}(T)=0,
  \end{equation}
  where $\bar{R}_{ij}(t):=\ti{\Delta} C_{ij}(t)+\bar{V}_{ij}(t)-\bar{V}_i(t)$ and
  \begin{align*}
    \bar{\Lambda}_{ij}(\mathrm{d}t)=\mathds{1}_{(\bar{R}_{ij}(t)< 0)}U_{ij}(\mathrm{d}t)+\mathds{1}_{(\bar{R}_{ij}(t)\geq 0)}L_{ij}(\mathrm{d}t).
  \end{align*}
  Then $\bar{V}(t)=\inf_{\Lambda\in\mathcal{M}}\amsmathbb{E}_{t,Z(t)}^{\bar{\Lambda}}[L(t,T)]$ and $\bar{\Lambda}=(\bar{\Lambda}_{ij})_{i\neq j}$ is the best-case scenario.
\end{proposition}

\begin{proof}
  We start with the proof of Proposition~\ref{prop:bc-reserve}. First, we check whether $\bar{\Lambda}$ is admissible. By definition of $\bar{\Lambda}$ condition (\ref{eq:LLU_admissible}) is satisfied, and the conditions (ii)+(iii) of Assumption~\ref{ass:cumulative_tr} are inherited from $(L_{ij})_{i\neq j}$ and $(U_{ij})_{i\neq j}$, as they satisfy these exact conditions themselves. The $i$-predictability condition is satisfied since $(L_{ij})_{i\neq j}$ and $(U_{ij})_{i\neq j}$ themselves are $i$-predictable and since $\bar{R}_{ij}(t)$ is $i$-predictable. Thus $\bar{\Lambda}$ is admissible.

  In order to obtain optimality, the martingale optimality principle states that it suffices to show that
  \begin{align*}
    M^{\Lambda}(t):=\int_{(0,t]}\frac{1}{\mathcal{E}_{\Phi}(s)}C(\mathrm{d}s)+\frac{1}{\mathcal{E}_{\Phi}(t)}\bar{V}(t)
  \end{align*}
  is martingale for $\Lambda=\bar{\Lambda}$ and a submartingale for admissible $\Lambda\neq\bar{\Lambda}$. This is equivalent to showing that 
  \begin{align*}
    \bar{V}(t)\leq \amsmathbb{E}^{\Lambda}_{t,Z(t)}[L(t,T)],\quad \forall t\in [0,T],\,\forall\Lambda\in\mathcal{M},
  \end{align*}
  with equality if and only if $\Lambda=\bar{\Lambda}$ and since an application of Theorem~\ref{th:optional_projection_discounted} yields
  \begin{align*}
    \bar{V}(t)=\amsmathbb{E}_{t,Z(t)}^{\bar{\Lambda}}[L(t,T)],
  \end{align*}
  we can invoke Theorem~8.4~in~\cite{Christiansen&Furrer2024} to conclude that this is the case if
  \begin{align*}
    \mathds{1}_{(\ul{R}_{ij}(t)\geq 0)}(\ul{\Lambda}_{ij}-\Lambda_{ij})(\mathrm{d}t)&\leq 0,\quad t\geq 0,\,i,j\in\mathcal{J},\,i\neq j\\
    \mathds{1}_{(\ul{R}_{ij}(t)< 0)}(\ul{\Lambda}_{ij}-\Lambda_{ij})(\mathrm{d}t)&\geq 0,\quad t\geq 0,\,i,j\in\mathcal{J},\,i\neq j
  \end{align*}
  for all $\Lambda\in\mathcal{M}$, with equality in the inequalities if and only if $\Lambda=\bar{\Lambda}$. Plugging in our candidate for $\bar{\Lambda}$, we indeed get 
  \begin{align*}
    \mathds{1}_{(\bar{R}_{ij}(t)\geq 0)}(\bar{\Lambda}_{ij}-\Lambda_{ij})(\mathrm{d}t)&=\mathds{1}_{(\bar{R}_{ij}(t)\geq 0)}(L_{ij}-\Lambda_{ij})(\mathrm{d}t)\leq 0,\quad t\geq 0,\,i,j\in\mathcal{J},\,i\neq j\\
    \mathds{1}_{(\bar{R}_{ij}(t)< 0)}(\bar{\Lambda}_{ij}-\Lambda_{ij})(\mathrm{d}t)&=\mathds{1}_{(\bar{R}_{ij}(t)< 0)}(U_{ij}-\Lambda_{ij})(\mathrm{d}t)\geq 0,\quad t\geq 0,\,i,j\in\mathcal{J},\,i\neq j
  \end{align*}
  for all $\Lambda\in\mathcal{M}$, with equality in the inequalities if and only if $\Lambda=\bar{\Lambda}$. Thus Theorem~8.4 of~\cite{Christiansen&Furrer2024} yields
  \begin{align*}
    \bar{V}(t)=\sum_{i\in\mathcal{J}}I_i(t)\amsmathbb{E}^{\bar{\Lambda}}_{t,i}[L(t,T)]\leq\sum_{i\in\mathcal{J}}I_i(t)\amsmathbb{E}_{t,i}^{\Lambda}[L(t,T)]=\amsmathbb{E}^{\Lambda}_{t,Z(t)}[L(t,T)],\quad \forall t\geq 0,\,\Lambda\in\mathcal{M}.
  \end{align*}
  Taking the infimum yields the desired result. The worst-case result Proposition~\ref{prop:wc-reserve} now follows directly from the fact that 
  \begin{align*}
    \sup_{\Lambda\in\mathcal{M}}\amsmathbb{E}^{\Lambda}_{t,Z(t)}\bigg[\int_{(t,T]}\frac{\mathcal{E}_{\Phi}(t)}{\mathcal{E}_{\Phi}(s)}C(\mathrm{d}s)\bigg]=-\inf_{\Lambda\in\mathcal{M}}\amsmathbb{E}^{\Lambda}_{t,Z(t)}\bigg[\int_{(t,T]}\frac{\mathcal{E}_{\Phi}(t)}{\mathcal{E}_{\Phi}(s)}(-C)(\mathrm{d}s)\bigg].
  \end{align*}
\end{proof}

Note that the two Propositions assume the existence of a solution to the BSDEs~(\ref{eq:wc_thiele_BSDE}) and (\ref{eq:bc_thiele_BSDE}). They do not prove that such a solution exists. But since they both are of the form (\ref{eq:worst_case_BSDE_discounted}), Theorem~\ref{th:BSDE_existence_discounted} gives an existence and uniqueness result which is applicable here, even in the case where the contractual payments are reserve-dependent. Thus, Theorem~\ref{th:BSDE_existence_discounted} in conjunction with Proposition~\ref{prop:wc-reserve} and Proposition~\ref{prop:bc-reserve} ensures existence and uniqueness of the worst- and best-case prospective reserves and the accompanying worst- and best-case biometric scenario for a wide range of different life, health and disability insurance products.

Note also that the worst/best-case scenario intuitively makes sense. The quantities $R_{ij}(t)=\ti{\Delta}C_{ij}(t)+V_{ij}(t)-V_i(t)$ are the so-called sums-at-risk. Suppose that the insured jumps from state $i$ to state $j$ at time $t$. In that case, the insurer has to pay out the contractual payments $\ti{\Delta}C_{ij}(t)$ to the insured, while changing the prospective reserve associated with the insured from $V_i(t)$ to $V_{ij}(t)$. Thus, if $R_{ij}(t)>0$, the insurer has a positive sum of money at risk and a jump from $i$ to $j$ at time $t$ yields a loss. Therefore, whenever $R_{ij}(t)>0$, the worst-case biometric scenario is the one where jumps from $i$ to $j$ are most likely ($U_{ij}$), while the best-case biometric scenario is the one where such jumps are least likely ($L_{ij}$). On the other hand, if $R_{ij}(t)<0$, the insurer has a negative sum of money at risk and a jump from $i$ to $j$ at time $t$ yields a gain. Therefore, the choice of best- and worst-case scenarios is reversed.

\begin{example}[Semi-Markov model]
  We now return to the smooth semi-Markov process of Example~\ref{ex:semi-Markov_1}, and we assume that the following bounds for the transition intensities are given 
  \begin{align*}
    l_{ij}(t,t-\tau_i(t))\leq u_{ij}(t,t-\tau_i(t)),\quad \forall i,j\in\mathcal{J}:i\neq j,
  \end{align*}
  for bounded and measurable functions $l_{ij},u_{ij}:[0,T]\times[0,T]\rightarrow [0,\infty)$. The contractual sojourn and transition payments are assumed to be of the form
  \begin{align*}
    C_i(\mathrm{d}t)&=b_i(t,t-\tau_{i}(t),V_i(t))\mathrm{d}t\\
    \ti{\Delta} C_{ij}(\mathrm{d}t)&= b_{ij}(t,t-\tau_i(t),V_i(t),V_{ij}(t)),
  \end{align*}
  for measurable functions 
  \begin{align*}
    b_i&:[0,T]\times[0,T]\times\amsmathbb{R}\rightarrow\amsmathbb{R},\quad i\in\mathcal{J}\\
    b_{ij}&:[0,T]\times[0,T]\times\amsmathbb{R}\times\amsmathbb{R}\rightarrow\amsmathbb{R},\quad i,j\in\mathcal{J}:i\neq j,
  \end{align*}
  where each $b_i$ is bounded in its first two arguments and Lipschitz in its third argument, and each $b_{ij}$ is bounded in its first two arguments Lipschitz in its third and fourth argument. Finally, the cumulative interest rate $\Phi$ is given by,
  \begin{align*}
    \Phi(\mathrm{d}t)=\sum_{i\in\mathcal{J}}I_i(t-)r_i(t,t-\tau_i(t))\mathrm{d}t
  \end{align*}
  for bounded and measurable functions 
  \begin{align*}
    r_i:[0,T]\times[0,T]\rightarrow\amsmathbb{R},\quad i\in\mathcal{J}.
  \end{align*}
  Under these conditions, Theorem~\ref{th:BSDE_existence_discounted} applies, and we get existence and uniqueness of the worst- and best-case prospective reserves. Taking a closer look at the worst-case prospective reserve, Proposition~\ref{prop:wc-reserve} and Proposition~\ref{prop:BSDE_equiv_ODE_discounted} gives us that its jump-count representation satisfies the following system of differential equations:
  \begin{align*}
    \frac{\mathrm{d}}{\mathrm{d}t}f^n(t,\omega^n)=&\mathds{1}_{(t\geq t_n)}\bigg(r_{z_n}(t,t-t_n)f^n(t,\omega^n)-b_{z_n}(t,t-t_n,f^n(t,\omega^n))+\sum_{j:j\neq z_n}R^{n+1}_{t,z_nj}(t,\omega^n)\ul{\mu}_{z_nj}(t,\omega^n)\bigg),
  \end{align*}
  with $f^n(T,\omega^n)=0$ for $t\in[0,T]$, $\omega^n\in\Omega^n$ and $n\in\amsmathbb{N}_0$, where
  \begin{align*}
    R^{n+1}_{t,z_nj}(t,\omega^n)=&b_{z_n j}(t,t-t_n,f^{n+1}(t,\omega^{n+1}_{t,z_n j}(\omega^n)),f^{n}(t,\omega^n))+f^{n+1}(t,\omega^{n+1}_{t,z_n j}(\omega^n))-f^n(t,\omega^n)\\
    \ul{\mu}_{z_nj}(t,\omega^n)=&\mathds{1}_{(R_{z_nj}(t,\omega^n)\geq 0)}u_{z_nj}(t,t-t_n)+\mathds{1}_{(R_{z_nj}(t,\omega^n)< 0)}l_{z_nj}(t,t-t_n).
  \end{align*}
  By our assumptions, Theorem~\ref{th:BSDE_existence_discounted} implies that there exists a unique solution of this system. Since all the driving functions depend on $\omega^n$ only via the most recent jump time and state $(t_n,z_n)$ and are independent of $n$, it can be verified by insertion that this system of equations has a solution given by
  \begin{align*}
    f^n(t,\omega^n)=\sum_{i\in\mathcal{J}}\mathds{1}_{(z_n=i)}V_i(t,t-t_n),\quad \forall n\in\amsmathbb{N}_0,
  \end{align*}
  where the functions $V_i:[0,T]\times[0,T]\rightarrow\amsmathbb{R}$ are a solution of the system 
  \begin{align*}
    \frac{\mathrm{d}}{\mathrm{d}t}V_i(t,t-\tau)=&\mathds{1}_{(t\geq\tau)}\bigg(r_{i}(t,t-\tau)V_i(t,t-\tau)-b_{i}(t,t-\tau,f_i(t,t-\tau))\\
    &+\sum_{j:j\neq i}b_{ij}(t,t-\tau,V_j(t,0),V_i(t,t-\tau))+V_{j}(t,0)-V_i(t,t-\tau)\ul{\mu}_{ij}(t,t-\tau)\bigg),
  \end{align*}
  with $V_i(T,T-\tau)=0$ for $t\in[0,T]$, $i\in\mathcal{J}$ and $\tau\in [0,T]$. The key to seeing this is to realise that 
  \begin{align*}
    f^{n+1}(t,\omega^{n+1}_{t,ij}(\omega^n))=V_j(t,t-t)=V_j(t,0),\quad \forall t\in[0,T],\,\forall n\in\amsmathbb{N}_0.
  \end{align*}
  We can thus conclude that the worst-case prospective reserve is given by 
  \begin{align*}
    \ul{V}(t)=\sum_{n\in\amsmathbb{N}_0}I^n(t)\sum_{i\in\mathcal{J}}\mathds{1}_{(\zeta_n=i)}V_i(t,t-\tau_n)=V_{Z(t)}\bigg(t,t-\sum_{n\in\amsmathbb{N}_0}I^n(t)\tau_n\bigg)=V_{Z(t)}(t,U(t)),
  \end{align*}
  where $U(t):=t-\sum_{n\in\amsmathbb{N}_0}I^n(t)\tau_n$ is the duration since the last jump. Thus in conjunction with (\ref{eq:CF_jump}) and (\ref{eq:CF_sojourn}) we obtain the counterfactuals $\ul{V}_i(t)=V_i(t,t-\tau_i(t))$ and $\ul{V}_{ij}(t)=V_j(t,0)$, which by Proposition~\ref{prop:Rep_adapted_process_J}(i) and Theorem~\ref{th:optional_projection_discounted} implies that the the state-wise prospective reserves are given by
  \begin{align*}
    \ti{\ul{V}}_i(t)=I_i(t-)V_i(t-\tau_i(t))+\sum_{j:j\neq i}I_j(t-)V_j(t,0)=\amsmathbb{E}_{t,i}^{\ul{\mu}}[L(t,T)],\quad i\in\mathcal{J}.
  \end{align*}
  We note that this form resembles the expressions obtained in Section~9 of~\cite{ChristiansenFurrer2021}. Finally, note that the system of differential equations satisfied by the functions $V_i:[0,T]\times[0,T]\rightarrow\amsmathbb{R}$ closely resembles the Thiele differential equations in the smooth semi-Markov model without reserve-dependency of intensities and payments, and that the same numerical methods apply, see~\cite{AdekambiChristiansen2017}.
\end{example}

\section{The robust consumption-insurance problem}
In the previous section, we were interested in determining the worst- and best-case prospective reserves for a given insurance payment process and the associated worst- and best-case biometric scenario, given the bounds (\ref{eq:LLU}). In this section, we are interested in finding the consumption process and insurance coverage that maximise an individual's utility in a worst-case biometric scenario. For this, we will restrict ourselves to the absolutely continuous case (\ref{eq:absolutely_continuous_Lambda}) with the bounds (\ref{eq:llu}), and we require admissible transition intensities to satisfy (\ref{eq:llu_admissible}).

The consumption processes that the policyholder can choose from are of the form 
\begin{align*}
  C(\mathrm{d}t)=\sum_{i\in\mathcal{J}}I_i(t-)\bigg(c_i(t)\mathrm{d}t+\sum_{j:j\neq i}c_{ij}(t) N_{ij}(\mathrm{d}t)\bigg),
\end{align*}
for $i$-predictable consumption rates $(c_i)_{i\in\mathcal{J}}$ consumed continuously and $i$-predictable lump sum consumption $(c_{ij})_{i\neq j}$ consumed upon transition. Under each admissible biometric scenario $\mu$, the individual has access to an insurance market, consisting of the coverages
\begin{align*}
    M_{ij}^*(\mathrm{d}t)=N_{ij}(\mathrm{d}t)-I_i(t-)\phi_{ij}(t)\mu_{ij}(t)\mathrm{d}t,\quad M_{ij}^*(0)=0,\quad i\neq j,
\end{align*}
where $(\phi_{ij})_{i\neq j}$ are $i$-predictable, strictly positive and bounded on finite intervals. If one holds one unit of the coverage $M_{ij}^*$, one continuously pays the premium rate $\mu_{ij}^*(t):=\phi_{ij}(t)\mu_{ij}(t)$ while in state $i$ and receives a payment of one unit upon transition from $i$ to $j$. Note that the processes $(\mu_{ij}^*)_{i\neq j}$ constitute a set of transition intensities, which specify a pricing measure $\amsmathbb{Q}^{\mu}$, and that the processes $(M_{ij}^*)_{i\neq j}$ are $(\amsmathbb{F},\amsmathbb{Q}^{\mu})$-martingales. The processes $(\phi_{ij})_{i\neq j}$ can be viewed as price loadings, where $\phi_{ij}=1$ corresponds to the actuarially fair case. Given the insurance market, the individual has to decide on an insurance coverage strategy of the form 
\begin{align*}
  B(\mathrm{d}t)=\sum_{i\in\mathcal{J}}I_i(t-)\sum_{j:j\neq i}b_{ij}(t)M_{ij}^*(\mathrm{d}t),
\end{align*}
where $(b_{ij})_{i\neq j}$ are $i$-predictable. Set $\theta:=((c_i)_{i\in\mathcal{J}},(c_{ij})_{i\neq j},(b_{ij})_{i\neq j})$. With all this in place, the wealth of the individual is given by 
\begin{align*}
  X^{\theta,\mu}(\mathrm{d}t)=X(t-)\Phi(\mathrm{d}t)+B(\mathrm{d}t)-C(\mathrm{d}t),\quad X^{\theta,\mu}(0)=x_0>0,
\end{align*}
where $\Phi$ is a predictable cumulative interest rate process of the form
\begin{align*}
  \Phi(\mathrm{d}t)=r(t)\mathrm{d}t,\quad \Phi(0)=0,
\end{align*}
for an optional interest rate processes $r$ that is bounded on finite intervals. Utilising Proposition~\ref{prop:Rep_adapted_process_J} and Proposition~\ref{prop:Rep_dyn_J}, $X$ can also be written as
\begin{align*}
  X^{\theta,\Lambda}(\mathrm{d}t)=\sum_{i\in\mathcal{J}}I_i(t-)\Bigg(r_i(t)X_i(t-)-c_i(t)-\sum_{j:j\neq i}b_{ij}(t)\mu_{ij}^*(t)\Bigg)\mathrm{d}t+\sum_{i,j:i\neq j}(b_{ij}(t)-c_{ij}(t))N_{ij}(\mathrm{d}t).
\end{align*}
The utility process $U^{\theta,\mu}$ specifying the preferences of the individual is given by 
\begin{align*}
    U^{\theta}(\mathrm{d}t)=\sum_{i\in\mathcal{J}}I_i(t-)\bigg(u_i(t,c_i(t))\mathrm{d}t+\sum_{j:j\neq i}u_{ij}(t,c_{ij}(t))N_{ij}(\mathrm{d}t)\bigg).
\end{align*}
for suitable $i$-predictable $u_i,u_{ij}:[0,T]\times\Omega\times\amsmathbb{R}\rightarrow\amsmathbb{R}$. The processes $u_i$ determine how much utility the individual obtains from continuous consumption while sojourning in state $i$, and the processes $u_{ij}$ determine how much utility the individual obtains from lump sum consumption when transitioning from state $i$ to state $j$. The problem which the individual faces is given by 
\begin{align}\label{eq:consumption-insurance-problem}
    \max_{\theta\in \mathcal{A}}\min_{\mu\in\mathcal{M}}\amsmathbb{E}^{\mu}\bigg[\int_{(0,T]} e^{-\int_0^s\rho(u)\mathrm{d}u}U^{\theta}(\mathrm{d}s)+e^{-\int_0^T\rho(u)\mathrm{d}u}u_{Z_T}^T(X^{\theta,\mu}(T))\bigg],
\end{align}
where $\rho$ is an optional process that is bounded on finite intervals and $u^T_{Z(T)}:\Omega\times\amsmathbb{R}\rightarrow\amsmathbb{R}$ is $\mathcal{F}_T$-measurable. Here $\rho$ is the subjective discount rate, modelling the time-value of utility and $u^T_{Z(T)}$ is the utility of terminal wealth. This is a problem of the form~(\ref{eq:MaxMinProblem_intensities}) and thus the martingale optimality principle is applicable. We will now utilise this in order to derive a verification theorem, inspired by the approach of~\cite{RiederWopperer2012}, which treats the case of a non-Markovian diffusion process. 

Let $(V_i)_{i\in\mathcal{J}}$ be a family of functions $V_i:[0,T]\times [0,\infty)\times\amsmathbb{R}\rightarrow\amsmathbb{R}$ which are $C^{1}$ in all arguments and let $Y$ be an optional process given by
\begin{align*}
  Y(\mathrm{d}t)=\sum_{i\in\mathcal{J}}I_i(t-)\bigg(Y_i(\mathrm{d}t)+\sum_{j:j\neq i}Y_{ij}(t)-Y_i(t)N_{ij}(\mathrm{d}t)\bigg),\quad Y(T)=\kappa,
\end{align*}
where $\kappa$ is $\mathcal{F}_T$-measurable and $Y_i(\mathrm{d}t)=y_i(t)\mathrm{d}t$ for $i$-predictable processes $(y_i)_{i\in\mathcal{J}}$. Define 
\begin{align*}
  A^{\theta,\mu}(\mathrm{d}t):=\sum_{i,j\in\mathcal{J}:i\neq j}V_j(t,X_i(t)+b_{ij}(t)-c_{ij}(t),Y_{ij}(t))-V_i(t,X_i(t),Y_i(t))M_{ij}^{\mu}(\mathrm{d}t),\quad A^{\theta,\mu}(0)=0,
\end{align*}
where $M_{ij}^{\mu}(\mathrm{d}t)=N_{ij}(\mathrm{d}t)-I_i(t-)\mu_{ij}(t)\mathrm{d}t$ are the jump process martingales and for each $(t,x,\omega)\in[0,T]\times\amsmathbb{R}\times\Omega$ let 
\begin{align*}
  \mathcal{G}_i^{\theta,\mu}V_i(t,x,Y_i(t)):=&\frac{\partial V_i}{\partial t}(t,x,Y_i(t)) -\rho_i(t) V_i(t,x,Y_i(t))+ y_i(t)\frac{\partial V_i}{\partial y}(t,x,Y_i(t))\\
  &+(x r_i(t)-c_i(t))\frac{\partial V_i}{\partial x}(t,x,Y_i(t))+\sum_{j:j\neq i}R^V_{ij}(t,x)\mu_{ij}(t)
\end{align*}
where $R^V_{ij}(t,x):=V_j(t,x+b_{ij}(t)-c_{ij}(t),Y_{ij}(t))-V_i(t,x,Y_i(t))-\phi_{ij}(t)b_{ij}(t)\frac{\partial V_i}{\partial x}(t,x,Y_i(t))$. Finally we define the predictable process $\lambda^{\theta,\mu}$ given by
\begin{align*}
    \lambda^{\theta,\mu}(t)=\sum_{i\in\mathcal{J}}I_i(t-)\lambda^{\theta,\mu}_i(t),\quad \text{with}\quad \lambda_i^{\mu,\theta}(t)=u_{i}(t,c_i(t))+\sum_{j:j\neq i}u_{ij}(t,c_{ij}(t))\mu_{ij}(t).
\end{align*} 
\begin{definition}\label{def:con_ins_adm}
  $\theta=((c_i)_{i\in\mathcal{J}},(c_{ij})_{i\neq j},(b_{ij})_{i\neq j})$ are admissible if 
  \begin{enumerate}
    \item[(i)] For all $i\in\mathcal{J}$, the processes $c_i$, $(c_{ij})_{j:j\neq i}$, and $(b_{ij})_{j:j\neq i}$ are $i$-predictable.
    \item[(ii)] $X^{\theta,\mu}(t)\geq 0$ for all $t\in [0,T]$ and $\mu\in\mathcal{M}$
    \item[(iii)] $A^{\theta,\mu}$ is a $(\amsmathbb{F},\amsmathbb{P}^{\mu})$-martingale for all $\mu\in\mathcal{M}$
    \item[(iv)] For all $\mu\in\mathcal{M}$ it holds that 
    \begin{align*}
      \sup_{t\in[0,T]}\amsmathbb{E}_{t,i}^{\mu}\bigg[\int_{(t,T]}|\lambda^{\theta,\mu}(s)|\mathrm{d}s+|u_{Z(T)}^T(X^{\theta,\mu}(T))|\bigg](\omega)<\infty,\quad \forall\omega\in\Omega,\,\forall i\in\mathcal{J}.
    \end{align*} 
  \end{enumerate}
\end{definition}
\begin{theorem}[Verification theorem]\label{th:verification}
  Suppose that the functions $(V_i)_{i\in\mathcal{J}}$ and the process $Y$ satisfy 
  \begin{align*}
    \sup_{\theta\in\mathcal{A}}\inf_{\mu\in\mathcal{M}}\Big\{\lambda_i^{\theta,\mu}(t)+\mathcal{G}_i^{\theta,\mu}V_i(t,x, Y_i(t))\Big\}&=0\\
    V_i(T,x,Y_i(T))&=u^T_{i}(x)
  \end{align*}
  for each $i\in\mathcal{J}$ and all $(t,x,\omega)\in[0,T]\times\amsmathbb{R}\times\Omega$, while $(\ul{\theta},\ul{\mu})$ is the corresponding unique saddle point. If the pair $(\ul{\theta},\ul{\mu})$ is admissible, then $\ul{\theta}$ is the unique robust optimal control, $\ul{\mu}$ is the worst-case biometric scenario and
  \begin{align*}
    V_{Z(t)}(t,X^{\ul{\theta},\ul{\mu}}(t),Y(t))=\sup_{\theta\in\mathcal{A}}\inf_{\mu\in\mathcal{M}}\mathcal{V}(t,\theta,\mu)=\inf_{\mu\in\mathcal{M}}\sup_{\theta\in\mathcal{A}}\mathcal{V}(t,\theta,\mu).
  \end{align*}
\end{theorem}

\begin{proof}
  Let $\mathcal{E}_{\rho}(t):=e^{\int_0^t\rho(u)\mathrm{du}}$. For any admissible $(\theta,\mu)$ define 
  \begin{align*}
    M^{\theta,\mu}(t):=\int_0^t\frac{1}{\mathcal{E}_{\rho}(s)}\lambda^{\theta,\mu}(s)\mathrm{d}s+\frac{1}{\mathcal{E}_{\rho}(t)}V_{Z(t)}(t,X^{\theta,\mu}(t),Y(t)).
  \end{align*}
  By the change of variables formula for finite variation functions, the fact that $Z(t)\neq Z(t-)$ on a Lebesgue null set and the boundary condition, we get
  \begin{align*}
    \frac{1}{\mathcal{E}_{\rho}(t)}V_{Z(t)}(t,X^{\theta,\mu}(t),Y(t))=&\frac{1}{\mathcal{E}_{\rho}(T)}u_{Z(T)}^T(X^{\theta,\mu}(T))- \int_t^T \frac{1}{\mathcal{E}_{\rho}(s)}A^{\theta,\mu}(\mathrm{d}s)\\
    &-\int_{t}^T\frac{1}{\mathcal{E}_{\rho}(s)}\mathcal{G}_{Z(s)}^{\theta,\mu} W_{Z(s)}(s,X^{\theta,\mu}(s),Y(s))\mathrm{d}s.
  \end{align*}
  Since $A^{\theta,\mu}$ is a martingale, taking conditional expectation and applying Proposition~\ref{prop:conditional_expectation}(iii) yields 
  \begin{align*}
    \frac{1}{\mathcal{E}_{\rho}(t)}V_{Z(t)}(t,X^{\theta,\mu}(t),Y(t))&=\amsmathbb{E}_{t,Z(t)}^{\mu}\bigg[\frac{1}{\mathcal{E}_{\rho}(T)}u_{Z(T)}^T(X^{\theta,\mu}(T))-\int_{t}^T\frac{1}{\mathcal{E}_{\rho}(s)}\mathcal{G}_{Z(s)}^{\theta,\mu} V_{Z(s)}(s,X^{\theta,\mu}(s),Y(s))\mathrm{d}s\bigg],
  \end{align*}
  and we thus obtain
  \begin{align*}
    M^{\theta,\mu}(t)&=\int_0^t\frac{1}{\mathcal{E}_{\rho}(s)}\lambda^{\theta,\mu}(s)\mathrm{d}s+\amsmathbb{E}^{\mu}_{t,Z(t)}\bigg[\frac{1}{\mathcal{E}_{\rho}(T)}U_{Z(T)}^T(X^{\theta,\mu}(T))-\int_t^T\frac{1}{\mathcal{E}_{\rho}(s)}\mathcal{G}_{Z(s)}^{\theta,\mu} V_{Z(s)}(s,X^{\theta,\mu}(s),Y(s))\mathrm{d}s\bigg].
  \end{align*}
  Since $(\ul{\theta},\ul{\mu})$ is a saddle point it holds for all $i\in\mathcal{J}$ that
  \begin{align*}
    &\lambda^{\ul{\theta},\ul{\mu}}(t)+\mathcal{G}_i^{\ul{\theta},\ul{\mu}}V_i(t,x,Y_i(t))=0,\quad \forall (t,x,\omega)\in [0,T]\times\amsmathbb{R}\times\Omega,\\
    &\lambda^{\theta,\ul{\mu}}(t)+\mathcal{G}_i^{\theta,\ul{\mu}}V_i(t,x,Y_i(t))\leq 0,\quad \forall (t,x,\omega)\in [0,T]\times\amsmathbb{R}\times\Omega,\,\forall\theta\in\mathcal{A}\\
    &\lambda^{\ul{\theta},\mu}(t)+\mathcal{G}_i^{\ul{\theta},\mu}V_i(t,x,Y_i(t))\geq 0,\quad \forall (t,x,\omega)\in [0,T]\times\amsmathbb{R}\times\Omega,\,\forall\mu\in\mathcal{M},
  \end{align*}
  which implies 
  \begin{align*}
    &\lambda^{\ul{\theta},\ul{\mu}}(t)=-\mathcal{G}_{Z(t)}^{\ul{\theta},\ul{\mu}}V_{Z(t)}(t,X^{\ul{\theta},\ul{\mu}}(t),Y(t)),\quad \forall(t,\omega)\in [0,T]\times\Omega\\
    &\lambda^{\theta,\ul{\mu}}(t)\leq-\mathcal{G}_{Z(t)}^{\theta,\ul{\mu}}V_{Z(t)}(t,X^{\theta,\ul{\mu}}(t),Y(t)),\quad \forall(t,\omega)\in [0,T]\times\Omega,\,\forall\theta\in\mathcal{A}\\
    &\lambda^{\ul{\theta},\mu}(t)\geq-\mathcal{G}_{Z(t)}^{\ul{\theta},\mu}V_{Z(t)}(t,X^{\ul{\theta},\mu}(t),Y(t)),\quad \forall(t,\omega)\in [0,T]\times\Omega,\,\forall\mu\in\mathcal{M}.
  \end{align*}
  Let $0\leq t\leq\tau\leq T$ be arbitrary. By the Tower Property (see Proposition~4.4(ii) of~\cite{Christiansen&Furrer2024}) and Proposition~\ref{prop:conditional_expectation}(ii) we thus have 
  \begin{align*}
    \amsmathbb{E}^{\ul{\mu}}_{t,Z(t)}[M^{\ul{\theta},\ul{\mu}}(\tau)]=&\int_0^{t}\frac{1}{\mathcal{E}_{\rho}(s)}\lambda^{\ul{\theta},\ul{\mu}}(s)\mathrm{d}s+\amsmathbb{E}_{t,Z(t)}^{\ul{\mu}}\bigg[\frac{1}{\mathcal{E}_{\rho}(T)}u_{Z(T)}^T(X^{\ul{\theta},\ul{\mu}})\bigg]\\
    &+\amsmathbb{E}_{t,Z(t)}^{\ul{\mu}}\bigg[\int_{t}^{\tau}\frac{1}{\mathcal{E}_{\rho}(s)}\lambda^{\ul{\theta},\ul{\mu}}(s)\mathrm{d}s-\int_{\tau}^{T}\frac{1}{\mathcal{E}_{\rho}(s)}\mathcal{G}_{Z(s)}^{\theta,\mu} V_{Z(s)}(s,X^{\ul{\theta},\ul{\mu}}(s),Y(s))\mathrm{d}s\bigg]\\
    =&\int_0^{t}\frac{1}{\mathcal{E}_{\rho}(s)}\lambda^{\ul{\theta},\ul{\mu}}(s)\mathrm{d}s+\frac{1}{\mathcal{E}_{\rho}(t)}V_{Z(t)}(t,X^{\ul{\theta},\ul{\theta}}(t),Y(t))=M^{\ul{\theta},\ul{\mu}}(t).
  \end{align*}
  Similarly we obtain $\amsmathbb{E}_t^{\ul{\mu}}[M^{\theta,\ul{\mu}}(\tau)]\leq M^{\theta,\ul{\mu}}(t)$ for all $\theta\in\mathcal{A}$ and $\amsmathbb{E}_t^{\mu}[M^{\ul{\theta},\mu}(\tau)]\geq M^{\ul{\theta},\mu}(t)$ for all $\mu\in\mathcal{M}$. The martinale optimality principle yields the desired result.
\end{proof}

\subsection{Power utility}
We now consider the problem (\ref{eq:consumption-insurance-problem}) with power utility preferences. For $i,j\in\mathcal{J}$ with $i\neq j$ we choose the utility functions
\begin{align*}
  u_i(t,c)=\frac{w_i^{\gamma}(t)}{1-\gamma}c^{1-\gamma}\quad\text{and}\quad u_{ij}(t,c)=\frac{w_{ij}^{\gamma}(t)}{1-\gamma}c^{1-\gamma}\quad\text{and}\quad u_i^T(x)=\frac{w_{i,T}^{\gamma}}{1-\gamma}x^{1-\gamma}
\end{align*}
for some $\gamma\in (0,1)\cup (1,\infty)$, which is the relative risk aversion parameter. For each $i\in\mathcal{J}$, the non-negative processes $(w_i,(w_{ij})_{j:j\neq i})$ are $i$-predictable and can be interpreted as utility weights. The larger $w_i$, the more utility the individual gains from consumption while in state $i$ and the larger $w_{ij}$, the more utility the individual gains from consumption upon transition from state $i$ to state $j$. The random variables $(w_{i,T})_{i\in\mathcal{J}}$ are non-negative and $\mathcal{F}_T$-measurable and determine the utility gained from terminal wealth in state $i$. Under the following boundedness assumptions, the problem with these power utility preferences can be solved explicitly.
\begin{assumption}\label{ass:power_utility}
  Assume that 
  \begin{enumerate}
    \item[(i)] The processes $(w_i,(w_{ij})_{j:j\neq i})_{i\in\mathcal{J}}$ are bounded on finite intervals
    \item[(ii)] The random variables $(w_{i,T})_{i\in\mathcal{J}}$ are bounded and $\inf_{\omega\in\Omega}w_{i,T}(\omega)>0$
    \item[(iii)] The processes $r$ and $\rho$ are bounded on finite intervals
    \item[(iv)] The processes $(\phi_{ij})_{i\neq j}$ satisfy $\inf_{\omega\in\Omega}\phi_{ij}(\omega)>0$.  
  \end{enumerate}
\end{assumption}

For each $i\in\mathcal{J}$ let $V_i:[0,\infty)\times\amsmathbb{R}\rightarrow\amsmathbb{R}$ be given by 
\begin{align*}
  V_i(x,y)=\frac{1}{1-\gamma}y^{\gamma}x^{1-\gamma},\quad i\in\mathcal{J}
\end{align*}
and let the process $Y$ be given by the BDS
\begin{align}
  \begin{split}\label{eq:power_utility_Y}
  Y(\mathrm{d}t)&=\hat{r}(t)Y(t-)\mathrm{d}t-W(\mathrm{d}t)+\sum_{i,j:i\neq j}R_{ij}(t)(N_{ij}(\mathrm{d}t)-I_i(t-)\ul{\mu}_{ij}(t)\mathrm{d}t),\\
  Y(T)&=w_{Z(T)}(T),
  \end{split}
\end{align}
where $R_{ij}(t)=(\phi_{ij}^{1-\frac{1}{\gamma}}(t)w_{ij}(t)+(1-\hat{\phi}_{ij}(t))Y_{i}(t)+(\phi_{ij}^{1-\frac{1}{\gamma}}(t)-1)Y_{ij}(t))+Y_{ij}(t)-Y_i(t)$. The processes $\hat{r}_i$ and $\hat{\phi}_{ij}$ are given by
\begin{align*}
  \hat{r}_i(t)=\frac{1}{\gamma}\rho_i(t)+\bigg(1-\frac{1}{\gamma}\bigg)r_i(t)\quad\text{and}\quad\hat{\phi}_{ij}(t)=\frac{1}{\gamma}+\bigg(1-\frac{1}{\gamma}\bigg)\phi_{ij}(t),
\end{align*}
while the process $W$ is given by 
\begin{align*}
  W(\mathrm{d}t)=\sum_{i\in\mathcal{J}}I_i(t-)\bigg(w_i(t)\mathrm{d}t+\sum_{j:j\neq i}\big(\phi_{ij}^{1-\frac{1}{\gamma}}(t)w_{ij}(t)+(1-\hat{\phi}_{ij}(t))Y_{i}(t)+(\phi_{ij}^{1-\frac{1}{\gamma}}(t)-1)Y_{ij}(t)\big)N_{ij}(\mathrm{d}t)\bigg),
\end{align*}
with $W(0)=0$. Then define the intensities $\ul{\mu}_{ij}(t)$ by 
\begin{align*}
  \ul{\mu}_{ij}(t)=
  \begin{cases}
    l_{ij}(t) & \text{if } \frac{\gamma}{1-\gamma}R_{ij}(t)\geq 0\\
    u_{ij}(t) & \text{if } \frac{\gamma}{1-\gamma}R_{ij}(t)< 0,
  \end{cases}
\end{align*}
the consumption processes $(\ul{c}_{i},(\ul{c}_{ij})_{j:\neq i})_{i\in\mathcal{J}}$ by
\begin{align*}
  \ul{c}_{i}(t)=w_i(t)\frac{X_i(t)}{Y_i(t)}\quad\text{and}\quad \ul{c}_{ij}(t)=\phi_{ij}^{-\frac{1}{\gamma}}w_{ij}(t)\frac{X_i(t)}{Y_i(t)},\quad i,j\in\mathcal{J}:i\neq j,
\end{align*}
and the insurance allocation $((\ul{b}_{ij})_{j:\neq i})_{i\in\mathcal{J}}$ by 
\begin{align*}
  \ul{b}_{ij}(t)=\bigg(\phi_{ij}^{-\frac{1}{\gamma}}(t)\frac{w_{ij}(t)+Y_{ij}(t)}{Y_{i}(t)}-1\bigg)X_i(t),\quad \quad i,j\in\mathcal{J}:i\neq j.
\end{align*}
\begin{proposition}
  Assume that Assumption~\ref{ass:power_utility} holds. Then the processes $(\ul{c}_{i},(\ul{c}_{ij})_{j:\neq i},(\ul{b}_{ij})_{j:\neq i})_{i\in\mathcal{J}}$ are the robust optimal controls, the processes $(\ul{\mu}_{ij})_{i\neq j}$ constitute the biometric worst-case scenario and 
  \begin{align*}
    V_{Z(t)}(X^{\ul{\theta}(t),\ul{\mu}},Y(t))=\sup_{\theta\in\mathcal{A}}\inf_{\mu\in\mathcal{J}}\amsmathbb{E}_{t,Z(t)}^{\ul{\mu}}\bigg[\int_{(t,T]} e^{-\int_t^s\rho(u)\mathrm{d}u}U^{\theta}(\mathrm{d}s)+e^{-\int_t^T\rho(u)\mathrm{d}u}u_{Z_T}^T(X^{\theta,\mu}(T))\bigg].
  \end{align*}
\end{proposition}
\begin{proof}
  See Subsection~\ref{subsection:verification}.
\end{proof}

From (\ref{eq:power_utility_Y}) we see that the process $Y$ satisfies a BSDE of the form (\ref{eq:worst_case_BSDE_discounted}), and our assumptions and Theorem~\ref{th:BSDE_existence_discounted} yield that $Y$ exists and is unique. Theorem~\ref{th:optional_projection_discounted} yields that 
\begin{align*}
  Y(t)=\amsmathbb{E}^{\ul{\mu}}_{t,Z(t)}\bigg[\int_{(t,T]}e^{-\int_t^s\hat{r}(u)\mathrm{d}u}W(\mathrm{d}s)+e^{-\int_t^T\hat{r}(u)\mathrm{d}u}w_{Z(T),T}\bigg],
\end{align*}
which means that $Y$ can be interpreted as the expected present value of the future flow of utility, where the expectation is taken under the worst-case measure. The quantity $R_{ij}$ is the utility at risk, and if $R_{ij}$ is positive, then the individual gains utility when transitioning from $i$ to $j$, and if $R_{ij}$ is negative, the individual loses utility when transitioning from $i$ to $j$. Interestingly, the worst-case biometric scenario not only depends on the sign of $R_{ij}$, but also on whether $\gamma>1$ or $\gamma<1$. If $\gamma<1$ and $R_{ij}(t)>0$, then the worst-case biometric scenario is $l_{ij}$, where transitions from $i$ to $j$ are less likely, while if $R_{ij}(t)<0$ then the worst-case biometric scenario is $u_{ij}$, where the transitions from $i$ to $j$ are most likely. When $\gamma>1$, then the situation is reversed. If $R_{ij}(t)>0$, then the worst-case biometric scenario is $u_{ij}$ and if $R_{ij}(t)<0$, the worst-case biometric scenario is $l_{ij}$. This indicates that, depending on the risk aversion parameter, the perception of the worst-case changes. When $\gamma>1$ we can see that (\ref{eq:power_utility_Y}) is of the form (\ref{eq:wc_thiele_BSDE}) and if $\gamma<1$ we see that (\ref{eq:power_utility_Y}) is of the form (\ref{eq:bc_thiele_BSDE}). By Proposition~\ref{prop:wc-reserve} and Proposition~\ref{prop:bc-reserve} we can thus conclude that
\begin{align*}
  Y(t)&=\sup_{\mu\in\mathcal{M}}\amsmathbb{E}^{\mu}_{t,Z(t)}\bigg[\int_{(t,T]}e^{-\int_t^s\hat{r}(u)\mathrm{d}u}W(\mathrm{d}s)+e^{-\int_t^T\hat{r}(u)\mathrm{d}u}w_{Z(T),T}\bigg],\quad \gamma>1\\
  Y(t)&=\inf_{\mu\in\mathcal{M}}\amsmathbb{E}^{\mu}_{t,Z(t)}\bigg[\int_{(t,T]}e^{-\int_t^s\hat{r}(u)\mathrm{d}u}W(\mathrm{d}s)+e^{-\int_t^T\hat{r}(u)\mathrm{d}u}w_{Z(T),T}\bigg],\quad \gamma<1.
\end{align*}
This shows that an individual with $\gamma>1$ views the future flow of utility as a liability and therefore maximises the expected present value thereof, which, as a consequence, reduces all robust optimal consumption processes $(c_i,(c_{ij})_{j:j\neq i})_{i\in\mathcal{J}}$ compared to other admissible biometric scenarios. Thus, the worst-case scenario of an individual with $\gamma>1$ is to run out of money to spend, and therefore, they consume conservatively. The individual with $\gamma<1$ is exactly the opposite. They view the future flow of utility as an asset and therefore minimise the expected present value thereof, which, as a consequence, increases all robust optimal consumption processes $(c_i,(c_{ij})_{j:j\neq i})_{i\in\mathcal{J}}$ compared to other admissible biometric scenarios. The worst-case scenario of an individual with $\gamma<1$ is thus not to get to spend all of their money, and therefore they consume aggressively. Thus if we set $\mathcal{J}=\{0=\text{alive},1=\text{dead}\}$ and assume that $R_{01}(t)<0$ (most individuals lose utility when dying), then the worst-case scenario of an individual with $\gamma<1$ is to die before spending all their money, while the worst-case scenario of an individual with $\gamma>1$ is to spend all their money before they die.

\subsection{Verification}\label{subsection:verification}
We start by noting that the process $Y$ is given by a BSDE of the form (\ref{eq:worst_case_BSDE_discounted}), which by Theorem~\ref{th:BSDE_existence_discounted} has a unique solution, bounded on finite intervals. Thus, our candidate solution is well-defined. Next, we have to check that our candidate optimal controls are admissible. For this note that for all $i\in\mathcal{J}$, the processes $\ul{c}_i$, $(\ul{c}_{ij})_{j:j\neq i}$, $(\ul{b}_{ij})_{j:j\neq i})$ are indeed $i$-predictable, while $\ul{\mu}_{ij}$ satisfies Assumption~\ref{ass:intensities} and (\ref{eq:llu_admissible}). Thus point $\ul{\mu}$ is admissible and (i) of Definition~\ref{def:con_ins_adm} is satisfied. Furthermore, we note that by rewriting (\ref{eq:power_utility_Y}) a little, we obtain
\begin{align*}
  Y(\mathrm{d}t)=\psi(t)Y(t)\mathrm{d}t-\hat{W}(\mathrm{d}t)+\sum_{i,j:i\neq j}\hat{R}_{ij}(t)(N_{ij}(\mathrm{d}t)-I_i(t-)\hat{\ul{\mu}}_{ij}(t)\mathrm{d}t),\quad Y(T)=w_{Z(T),T},
\end{align*}
where $\hat{R}_{ij}(t)=w_{ij}(t)+Y_{ij}(t)-Y_{i}(t)$ and $\psi$ and $\hat{W}$ are given by
\begin{align*}
  \hat{W}(\mathrm{d}t)&=\sum_{i\in\mathcal{J}}I_i(t-)\bigg(w_i(t)\mathrm{d}t+\sum_{j:j\neq i}w_{ij}(t)N_{ij}(\mathrm{d}t)\bigg),\quad \hat{W}(0)=0,\\
  \psi(t)&=\sum_{i\in\mathcal{J}}I_i(t-)\bigg(\hat{r}_i(t)-\Big(1-\hat{\phi}_{ij}(t)\phi_{ij}^{\frac{1}{\gamma}-1}(t)\Big)\bigg).
\end{align*}
Since the processes $(\phi_{ij})_{i\neq j}$ are bounded, bounded away from zero, and $i$-predictable, the processes $\hat{\ul{\mu}}_{ij}=\phi_{ij}^{1-\frac{1}{\gamma}}\ul{\mu}_{ij}(t)$ satisfy Assumption~\ref{ass:intensities} and are therefore valid transition intensities. By Theorem~\ref{th:optional_projection_discounted}, we thus get that 
\begin{align*}
  Y(t)=\amsmathbb{E}_{t,Z(t)}^{\hat{\ul{\mu}}}\bigg[\int_{(t,T]}e^{-\int_{t}^s\psi(u)\mathrm{d}u}\hat{W}(\mathrm{d}s)+e^{-\int_{t}^T\psi(u)\mathrm{d}u}w_{Z(T),T}\bigg].
\end{align*}
As all the utility weights are non-negative, the integral in the expectation is always non-negative, and since $\psi$ is bounded and $w_{i,T}$ is bounded away from zero, there exists a constant $K>0$ such that
\begin{align*}
  \inf_{(t,\omega)\in[0,T]\times\Omega} Y(t,\omega)\geq\inf_{(t,\omega)\in[0,T]\times\Omega}\amsmathbb{E}_{t,Z(t)}^{\ul{\mu}}\big[e^{-\int_{t}^T\psi(u)\mathrm{d}u}w_{Z(T),T}\big](\omega)\geq K>0.
\end{align*}
Thus, we have that $Y$ is strictly positive and bounded away from zero. For any $\mu\in\mathcal{M}$ the optimal wealth process is given by 
\begin{align*}
  X^{\ul{\theta},\mu}(t)=&x_0\exp\bigg(\int_0^t\sum_{i\in\mathcal{J}}I_i(s)\bigg(r_{i}(s)-w_{i}(s)-\sum_{j:j\neq i}\frac{w_{ij}(s)+Y_{ij}(s)}{Y_i(s)}\phi_{ij}^{-\frac{1}{\gamma}}(s)\mu_{ij}^*(s)\bigg)\mathrm{d}s\bigg)\\
  &\cdot\exp\bigg(\int_{(0,t]}\sum_{i,j:i\neq j}\log\bigg(\frac{Y_{ij}(s)}{Y_{i}(s)}\phi_{ij}^{-\frac{1}{\gamma}}(s)\bigg)N_{ij}(\mathrm{d}s)\bigg).
\end{align*}
Since $\phi_{ij}$ and $Y$ are strictly positive, we can conclude that $X^{\ul{\theta},\mu}(t)\geq 0$ and condition (ii) of Definition~\ref{def:con_ins_adm} is satisfied. Next we verify that the process $A^{\ul{\theta},\mu}$ given by 
\begin{align*}
  A^{\ul{\theta},\mu}(\mathrm{d}t)=\sum_{i,j:i\neq j}\frac{1}{1-\gamma}Y_i^{\gamma-1}(t)X_i^{1-\gamma}(t)\big(Y_{ij}(t)\phi_{ij}^{1-\frac{1}{\gamma}}-Y_i(t)\big)M_{ij}^{\mu}(\mathrm{d}t),\quad A^{\ul{\theta},\ul{\mu}}(0)=0
\end{align*}
is a $(\amsmathbb{F},\amsmathbb{P}^{\mu})$-martingale for any $\mu\in\mathcal{M}$. By Proposition~4.2 in~\cite{Christiansen&Furrer2024}, Assumption~\ref{ass:intensities}(ii) and Tonelli's theorem, this is the case if we can show for all $i,j,k\in\mathcal{J}$ that
\begin{align*}
  \int_{(s,t]}\amsmathbb{E}_{s,i}^{\mu}\bigg[\bigg|\frac{1}{1-\gamma}Y_j^{\gamma-1}(u)X_j^{1-\gamma}(u)\big(Y_{jk}(u)\phi_{jk}^{1-\frac{1}{\gamma}}(u)-Y_j(u)\big)\bigg|\bigg]\alpha_{jk}(u)\mathrm{d}u<\infty,\quad \forall 0\leq s\leq t\leq T.
\end{align*}
Due to our assumptions, $\phi_{ij}^{1-\frac{1}{\gamma}}$ is bounded regardless of whether $\gamma>1$ or $\gamma<1$. Since $Y$ is bounded and bounded away from zero, the processes $Y_i$, $Y_{ij}$ and $Y_i^{\gamma-1}$ are bounded, regardless of whether $\gamma>1$ or $\gamma<1$. The form of $X^{\ul{\theta},\mu}$, all our boundedness assumptions, and the inequality $N(s)\leq N(s-)+1$ for all $s>0$ allow us to obtain the bound
\begin{align*}
  (X^{\ul{\theta},\mu})^{1-\gamma}(s)\leq x_0e^{(K_1(1+N(s)))}\leq x_0 e^{2K_1+K_1N(t-)}e^{K_1(N(s)-N(t))},\quad 0\leq t\leq s\leq T,
\end{align*}
for some constant $K_1>0$. Using this bound and Proposition~\ref{prop:exp_bound_N}, we can for all $\gamma\in(0,1)\cup(0,\infty)$ obtain a bound of the form
\begin{align*}
  \amsmathbb{E}_{t,i}[(X_j^{\ul{\theta},\mu})^{1-\gamma}(s)](\omega)\leq x_0 e^{2K_1+K_1N(T,\omega)}K_2<\infty,\quad \forall 0\leq t\leq s\leq T,\,\forall\omega\in\Omega,
\end{align*}
for some constant $K_2>0$. Thus we can conclude that $A^{\ul{\theta},\mu}$ is a $(\amsmathbb{F},\amsmathbb{P}^{\mu})$-martingale for all $\mu\in\mathcal{M}$ and condition (iii) of Definition~\ref{def:con_ins_adm} is satisfied. Finally, by using the very same bound in conjunction with the other boundedness assumptions and Assumption~\ref{ass:intensities}(ii), we obtain for all $\mu\in\mathcal{M}$
\begin{align*}
  \sup_{t\in[0,T]}\amsmathbb{E}_{t,i}^{\mu}\bigg[\int_{(t,T]}|\lambda^{\ul{\theta},\mu}(s)|\mathrm{d}s+|u_{Z(T)}^T(X^{\ul{\theta},\mu}(T))|\bigg](\omega)<\infty,\quad \forall\omega\in\Omega,\,\forall i\in\mathcal{J}.
\end{align*}
Thus, condition (iv) of Definition~\ref{def:con_ins_adm} is satisfied, and we can conclude that our candidate optimal robust control is admissible as well. Last but not least, we note that 
\begin{align*}
  \lambda_i^U(x,Y_i(t))+\mathcal{G}_i^{\theta,\mu}V_i(x, Y_i(t)),\quad i\in\mathcal{J}
\end{align*}
is concave as a function of $c_i$, $c_{ij}$ and $b_{ij}$ and linear (convex-like) on a compact domain $[l_{ij}(t),u_{ij}(t)]$ in $\mu_{ij}$. Thus by Theorem~4.2 of~\cite{Sion1958} the problem 
\begin{align*}
  \sup_{\theta\in\mathcal{A}}\inf_{\mu\in\mathcal{M}}\Big\{\lambda_i^U(x,Y_i(t))+\mathcal{G}_i^{\theta,\mu}V_i(x, Y_i(t))\Big\},\quad i\in\mathcal{J}
\end{align*}
has a unique saddle point, which by simple calculations can be shown to equal the above optimal controls. Furthermore, by inserting everything, it can be shown that 
\begin{align*}
  \sup_{\theta\in\mathcal{A}}\inf_{\mu\in\mathcal{M}}\Big\{\lambda_i^U(x,Y_i(t))+\mathcal{G}_i^{\theta,\mu}V_i(x, Y_i(t))\Big\}=0,\quad i\in\mathcal{J}.
\end{align*}
Thus, by Theorem~\ref{th:verification}, we can conclude that our candidate solution indeed is optimal.

\section*{Acknowledgements}
The second author has carried out this research in association with the project frame InterAct. The authors would like to thank Christian Furrer for fruitful discussions that improved the quality of this paper. Moreover, the second author acknowledges the hospitality of the Institute of Financial Mathematics and Applied Number Theory at the Johannes Kepler University in Linz, where parts of this research have been conducted during a 3-month research stay.

\bibliographystyle{plainnat}
\bibliography{references.bib}

\appendix

\section{The space of finite variation functions}\label{appendix:FV}
\begin{definition}
  Let $f:[0,T]\rightarrow\amsmathbb{R}$ be a function and let 
  \begin{align*}
    |f|_a^b:=\sup\bigg\{\sum_{i=1}^n|f(t_i)-f(t_{i-1})|:a=t_0<t_1<\dots<t_n=b\bigg\}
  \end{align*}
  be the variation of $f$ on $[a,b]\subset[0,T]$. Then $f$ is of finite variation on $[a,b]$, if $|f|_a^b<\infty$.
\end{definition}
 
\begin{proposition}\label{prop:variation_properties}
  For any $0\leq a\leq b\leq T$, the variation $|f|_a^b$ has the following properties:
  \begin{enumerate}
    \item[(i)] $t\mapsto |f|_a^t$ is increasing
    \item[(ii)] If $t\mapsto f(t)$ is càdlàg so is $t\mapsto |f|_a^t$ and $\Delta|f|_a^t=|\Delta f(t)|$
    \item[(iii)] $|f(t)-f(a)|\leq |f|_a^t$ and $|f(b)-f(t)|\leq |f|_a^b-|f|_a^t=|f|_t^b$
    \item[(iv)] $|f(t)|\leq |f|_t^b + |f(b)|\leq |f|_a^b+|f(b)|$ 
  \end{enumerate}
\end{proposition}

Set $\text{BV}_{[0,T]}^{\text{all}}:=\{f:[0,T]\rightarrow\amsmathbb{R}: |f|_0^T<\infty\}$ and define the norms
\begin{align*}
  \|f\|_{BV,0}:=|f(0)|+|f|_0^T \quad \text{and}\quad \|f\|_{BV,T}:=|f(T)|+|f|_0^T.
\end{align*}
The space $\text{BV}_{[0,T]}^{\text{all}}$ equipped with either $\|f\|_{BV,0}$ is a Banach space (see Theorem 2.2.2 in~\cite{Monteiro2018}).
\begin{lemma}
  The norms $\|f\|_{BV,0}$ and $\|f\|_{BV,T}$ are equivalent.
\end{lemma}
\begin{proof}
  Note that by Proposition~\ref{prop:variation_properties}(iii)
  \begin{align*}
    |f(0)|\leq |f(T)-f(0)|+|f(T)|\leq |f|_0^T+|f(T)|
  \end{align*}
  and thus $\|f\|_{BV,0}\leq 2(\|f\|_{BV,T})$. Similarly
  \begin{align*}
    |f(T)|\leq |f(T)-f(0)|+|f(0)|\leq |f|_0^T+|f(0)|
  \end{align*}
  and thus $\|f\|_{BV,T}\leq 2(\|f\|_{BV,0})$. Combining this yields $\frac{1}{2}\|f\|_{BV,0}\leq \|f\|_{BV,T}\leq 2\|f\|_{BV,0}$.
\end{proof}
Let $\text{BV}_{[0,T]}$ be the subspace of finite variation functions that are càdlàg:
\begin{align*}
  \text{BV}_{[0,T]}:=\{f\in \text{BV}^{\text{all}}_{[0,T]}: f\text{ is càdlàg}\}.
\end{align*}
\begin{lemma}
  The space $\text{BV}_{[0,T]}$ with the norm $\|f\|_{BV,T}$ is a Banach space.
\end{lemma}
\begin{proof}
  Clearly $BV_{[0,T]}$ is a subspace of $BV_{[0,T]}^{\text{all}}$. It remains to show that $BV_{[0,T]}$ is closed, since closed subspaces of Banach spaces are Banach spaces themselves. Let $(f_n)_{n\in\amsmathbb{N}}\subset BV_{[0,T]}$, such that $\lim_{n\rightarrow\infty}\|f_n-f\|_{BV,T}=0$ for some $f\in BV_{[0,T]}^{\text{all}}$. Then by Proposition~\ref{prop:variation_properties}(iv) we have that 
  \begin{align*}
    \lim_{n\rightarrow\infty}\sup_{t\in[0,T]}|f_n(t)-f(t)|\leq  \lim_{n\rightarrow\infty} |f_n-f|_0^T+|f_n(T)-f(T)|=\lim_{n\rightarrow\infty}\|f_n-f\|_{BV,T}=0,
  \end{align*}
  implying that $f_n\rightarrow f$ uniformly. Since the limit of a uniformly convergent sequence of càdlàg functions is càdlag, we can conclude that $f\in BV_{[0,T]}$ and therefore that $BV_{[0,T]}$ is closed.
\end{proof}

We will also use a discounted version of the norm $\|f\|_{BV,T}$. Let $\nu:[0,T]\rightarrow[0,\infty)$ be an increasing càdlàg function, let $c>0$ be an arbitrary constant and define
\begin{align*}
  \|f\|_{BV,T}^{c,\nu}:=|f(T)|+\int_{(0,T]}e^{-c(\nu(T)-\nu(s-))}|f|(\mathrm{d}s).
\end{align*}
\begin{lemma}\label{lem:norm_equivlance}
  The norms $\|f\|_{BV,T}$ and $\|f\|_{BV,T}^{c,\nu}$ are equivalent.
\end{lemma}
\begin{proof}
  As $\nu$ is increasing, we have that 
  \begin{align*}
    \|f\|_{BV,T}^{c,\nu}\leq |f(T)|+\int_{(0,T]}|f|(\mathrm{d}s)=\|f\|_{BV,T}
  \end{align*}
  and 
  \begin{align*}
    \|f\|_{BV,T}^{c,\nu}\geq e^{-c(\nu(T)-\nu(0-))}\bigg(|f(T)|+\int_{(0,T]}|f|(\mathrm{d}s)\bigg)=e^{-c(\nu(T)-\nu(0-))}\|f\|_{BV,T},
  \end{align*}
  so $e^{-c(\nu(T)-\nu(0-))}\|f\|_{BV,T}\leq \|f\|_{BV,T}^{c,\nu}\leq \|f\|_{BV,T}$.
\end{proof}
The reason for the usefulness of $\|\cdot\|_{BV}^{c,\nu}$ is the following bound:
\begin{lemma}\label{lem:important_bound}
  Let $f:[0,T]\rightarrow\amsmathbb{R}$ be a non-decreasing càdlàg function and let $c>0$ be a constant. Then it holds that 
  \begin{align*}
    \int_{(t_1,t_2)}e^{-c(f(T)-f(t-))}f(\mathrm{d}t)\leq\frac{1}{c}e^{-c(f(T)-f(t_2-))},\quad 0\leq t_1\leq t_2\leq T
  \end{align*}
  and 
  \begin{align*}
    \int_{(t,T]}e^{-c(f(T)-f(t-))}f(\mathrm{d}t)\leq \frac{1}{c},\quad 0\leq t\leq T.
  \end{align*}
\end{lemma}
\begin{proof}
  Applying the change of variables formula for finite variation functions, we obtain that 
  \begin{align*}
    e^{cf(t_2-)}-e^{cf(t_1)}=\int_{(t_1,t_2)}ce^{cf(t-)}f^c(\mathrm{d}t)+\sum_{t_1<t<t_2}(e^{cf(t)}-e^{cf(t-)}).
  \end{align*}
  If $f$ has a jump at time $t$, then $\Delta f(t)>0$, since $f$ is non-decreasing. By the mean-value theorem there exists $\xi\in[f(t-),f(t)]$ such that 
  \begin{align*}
    e^{cf(t)}-e^{cf(t-)}=ce^{c\xi}\Delta f(t)\geq ce^{cf(t-)}\Delta f(t).
  \end{align*}
  Inserting this in the previous equation and rearranging, we obtain 
  \begin{align*}
    \int_{(t_1,t_2)}e^{cf(t-)}f(dt)\leq \frac{1}{c}\big(e^{cf(t_2-)}-e^{cf(t_1)}\big)\leq \frac{1}{c}e^{cf(t_2-)}.
  \end{align*}
  Multiplying both sides by $e^{-cf(T)}$ yields the first result. For the second result, we similarly obtain that 
  \begin{align*}
    \int_{(t,T]}e^{cf(s-)}f(\mathrm{d}s)\leq\frac{1}{c}e^{cf(T)},
  \end{align*}
  which upon multiplying with $e^{-cf(T)}$ yields the desired result.
\end{proof}

Let $f:[0,T]\rightarrow\amsmathbb{R}$ be a càdlàg function of finite variation. The Doleàn-Dade exponential starting at time $t\in[0,T]$ is given by 
\begin{align*}
  \mathcal{E}_f(t,\mathrm{d}s)=\mathcal{E}_f(t,s-)f(\mathrm{d}s),\quad \mathcal{E}_f(t,t)=1,\quad s\in [t,T]
\end{align*}
and has the explicit form
\begin{align*}
  \mathcal{E}_f(t,s)=e^{f^c(s)-f^c(t)}\prod_{t<u\leq s}(1+\Delta f(u)).
\end{align*}

\begin{lemma}[Gronwall's inequality]\label{lem:gronwall}
  Let $f:[0,T]\rightarrow[0,\infty)$ be a non-decreasing càdlàg function, let $\alpha\geq 0$ be a constant. If $g:[t,T]\rightarrow [0,\infty)$ satisfies
  \begin{align*}
    g(s)\leq \alpha +\int_{(t,s]}g(u-) f(\mathrm{d}u),\quad\forall s\in [t,T],
  \end{align*}
  then it holds that 
  \begin{align*}
    g(s)\leq \alpha \mathcal{E}_f(t,s),\quad \forall s\in[t,T].
  \end{align*}
\end{lemma}
\begin{proof}
  Let $V(s):=\alpha+\int_{(t,s]}g(u-)f(\mathrm{d}u)$ and note that since $g$ is non-negative and $f$ is non-decreasing, the function $V$ is non-decreasing. Then, using the integration by parts formula for càdlàg finite variation functions, we get
 \begin{align*}
  \frac{V(s)}{\mathcal{E}_f(t,s)}-\alpha=\int_{(t,s]}\bigg(\frac{g(u-)}{\mathcal{E}_f(t,u-)}-\frac{V(u)}{\mathcal{E}_f(t,u)}\bigg)f(\mathrm{d}u),\quad \forall s\in [t,T].
 \end{align*}
 Using that $V(u)=V(u-)+g(u-)\Delta f(u)$ and $\mathcal{E}_f(t,u)=(1+\Delta f(u))\mathcal{E}_f(t,u-)$ in conjunction with the fact that $g(u-)\leq V(u-)$, we obtain 
 \begin{align*}
  \frac{g(u-)}{\mathcal{E}_f(t,u-)}-\frac{V(u)}{\mathcal{E}_f(t,u)}=\frac{g(u-)-V(u-)}{\mathcal{E}_f(t,u)}\leq 0,
 \end{align*}
 which implies the desired result.
\end{proof}

\end{document}